\RequirePackage{fix-cm}
\documentclass[smallextended]{svjour3}  
\smartqed 
\usepackage[T1]{fontenc}
\usepackage[utf8]{inputenc}
\usepackage{amssymb}
\usepackage{graphicx}
\usepackage{xcolor}
\usepackage{amsmath}
\usepackage{pifont}
\usepackage{enumitem}
\usepackage{subfigure}
\usepackage{geometry}
\usepackage{caption}
\usepackage{longtable}
\usepackage{array}
\usepackage{comment}
\usepackage{url}
\usepackage[normalem]{ulem}
\usepackage{hyperref}
\usepackage[labelfont=bf,labelsep=space]{caption}

\begin{document}
\title{Topological Characteristics for the Analysis of Vector Fields}
\subtitle{New stable characteristics for discrete and continuous dynamical systems}

\author{Marta Marszewska \and 
Justyna Signerska \and 
Pawe{\l} D{\l}otko}

\institute{
           Corresponding author: Marta Marszewska \at
              Dioscuri Center in Topological Data Analysis, Institute of Mathematics, Polish Academy of Sciences, 8 \'Sniadeckich Street, 00-656 Warsaw, Poland \& \\
              Gda\'nsk University of Technology, 11/12 Gabriela Narutowicza Street, 80-233 Gda\'nsk
              Poland \\
              \email{marta.marszewska@pg.edu.pl}
              \and
           Justyna Signerska \at
              Gda\'nsk University of Technology, Faculty of Applied Physics and Mathematics, 11/12 Gabriela Narutowicza Street, 80-233 Gda\'nsk
              Poland \& \\
              Dioscuri Center in Topological Data Analysis, Institute of Mathematics, Polish Academy of Sciences, 8 \'Sniadeckich Street, 00-656 Warsaw, Poland
              \and
              Pawe{\l} D{\l}otko \at
             Centre of Trustfully Artificial Inteligence for Life Sciences, Warsaw University\\
           \and
}

\date{Received: date / Accepted: date}

\maketitle

\begin{abstract}
The qualitative analysis of nonlinear dynamical systems relies on identifying geometric and topological structures that govern phase-space organization. We develop a computational framework based on novel topological and geometric descriptors of vector fields that enables robust comparison of dynamical regimes from both analytical models and sampled data.
The framework comprises the Begin--End Point Embedding (BEPE), Density of Directions (DOD) and Euler characteristic--based summaries, including Euler Characteristic Curve (ECC) and Euler Characteristic Profile (ECP). BEPE captures local variations in vector field, while DoD characterizes the global organization of vector directions. ECCs and ECPs provide filtration-based topological summaries of vector-field structure, with ECPs extending these summaries to multiparameter settings. Together, these descriptors encode complementary geometric and topological information about flow organization. We formulate both continuous and sampled versions of the descriptors and establish their robustness with respect to perturbations, finite sampling, and selected transformations.
The effectiveness of the proposed approach is demonstrated on a broad collection of benchmark systems, including the Hopf bifurcation, the FitzHugh–Nagumo model and the Lorenz system. Further experiments on the discrete Hénon map demonstrate the applicability of the framework to displacement fields generated by discrete dynamical systems, while experiments on high-dimensional turbulent flows assess its scalability beyond low-dimensional phase spaces. 
Comparisons with classical approaches, including Conley-index-based methods and $L_p$-type metrics, show that the proposed descriptors provide complementary structural information while offering favorable robustness, interpretability and computational efficiency, establishing a practical connection between dynamical systems theory and topological data analysis.


\end{abstract}

\section{Introduction}
\label{sec:introduction}
Qualitative analysis of dynamical systems, and in particular of systems exhibiting chaotic behavior, remains a central problem in mathematics, physics, and applied sciences~\cite{katok1997introduction,strogatz1994}. Classical tools of dynamical systems theory, such as Lyapunov exponents, metric and topological entropy~\cite{katok1980entropy}, or bifurcation analysis, provide quantitative insight into stability, sensitivity to initial conditions, and long-term evolution of trajectories. Despite their effectiveness, these methods often rely on strong regularity assumptions, require long and accurate trajectories, and may exhibit limited robustness under noise or finite sampling. These limitations become especially pronounced in high-dimensional systems or when the dynamics is inferred from empirical data.

Topological approaches offer a complementary qualitative perspective by focusing on global properties of dynamics that are invariant under continuous deformations ~\cite{carlsson2009topology,ghrist2014,zomorodian2005}. One of the most influential frameworks in this direction is the Conley index theory, which associates algebraic invariants to isolated invariant sets of a dynamical system~\cite{conley1978isolated}. However, practical computation of the Conley index is often challenging, as it requires the construction of isolating neighborhoods and index pairs, a task that becomes difficult in large-scale or data-driven settings~\cite{mischaikow2002conley}.

In recent years, Topological Data Analysis (TDA) has emerged as a powerful methodology for extracting robust topological signatures from complex datasets~\cite{edelsbrunner2010}. The most established tool within TDA is persistent homology (PH), which captures the evolution of homological features, such as connected components, cycles, and voids, across a filtration induced by a scalar parameter~\cite{edelsbrunner2008persistent}. Persistent homology has found numerous applications in the analysis of time series and dynamical systems \cite{khasawneh2016}. Nevertheless, its computational cost, limited scalability, and the inherent difficulties of extending the theory to multiparameter filtrations (multifiltration) significantly restrict its applicability to high-dimensional vector fields and large datasets.

These challenges motivate the search for alternative topological descriptors for dynamical systems that are computationally lightweight, stable under perturbations, sensitive to qualitative dynamical structure, and applicable both in continuous and sampled settings. Rather than tracking individual trajectories, such descriptors should provide global summaries of the geometry and organization of the vector field itself. An alternative class of descriptors is based on the Euler characteristic, a classical topological invariant that is computationally lightweight and well suited for large-scale problems. In recent work, D\l{}otko \& Gurnari ~\cite{dlotko2023ecc} introduced Euler Characteristic Curves (ECCs) for one-parameter filtrations and Euler Characteristic Profiles (ECPs) for multiparameter filtrations, demonstrating that these invariants exhibit strong stability properties while avoiding many of the computational challenges associated with PH. Importantly, ECCs and ECPs extend naturally to multiparameter settings, where a complete persistence theory is still lacking.

When applied to filtrations derived from vector fields rather than scalar fields alone, Euler characteristic–based invariants yield expressive signatures that vary smoothly across parameter space and respond sensitively to qualitative changes in the dynamics, including bifurcations and transitions between dynamical regimes.

In the present article, we develop a framework for the qualitative analysis and comparison of vector fields based on topological and geometric descriptors. Our main emphasis is on Euler characteristic--based constructions, in particular Euler characteristic curves and Euler characteristic profiles, which provide computable and stable summaries of filtrations derived from the underlying field. At the same time, we also introduce complementary characteristics, such as the begin--end point embedding and densities of directions, which capture additional geometric information and may be useful in applications where different aspects of the field are of primary interest.

Throughout the manuscript, we assume that the phase space $X$ is a connected compact subset of $\mathbb{R}^n$. This assumption is only mildly restrictive and is adopted mainly to simplify the exposition. For each point $x \in X$, we consider an autonomous vector field $\dot{x} = f(x)$. Depending on the available information, we work either with the full vector field in the continuous setting or with a finite sampled representation given by pairs $(x,f(x))$. This dual perspective allows us to treat both analytically defined systems and data-driven situations arising from simulations, discretizations, or empirical measurements. Even in the continuous setting, the proposed constructions are intended to remain compatible with sampled representations, reflecting the practical limitations of numerical and experimental work.

A central goal of the paper is to show that Euler characteristic--based descriptors provide an effective tool for detecting and comparing qualitative features of dynamics, including invariant structures, bifurcations, and transitions to more complex behavior. At the same time, the other descriptors introduced here offer additional and complementary ways of encoding the organization of vector fields. We compare these approaches with more classical methods, including the Conley index and $L_p$-type metrics, and highlight in particular the advantages of Euler characteristic curves and profiles in terms of robustness, interpretability, and computational efficiency.

The remainder of the paper is organized as follows. We start with recalling a couple of preliminary definitions. Section \ref{sec:distamces} reviews existing approaches to the comparison of dynamical systems, including pointwise and optimal transport distances between vector fields, invariant-measure-based methods, and topological techniques such as the Conley index. Section \ref{sec:new_met} introduces the proposed descriptors for vector fields, namely the Density of Directions (DoD), the Begin--End Point Embedding (BEPE), and Euler Characteristic Curves and Profiles (ECCs and ECPs) and develop the theoretical foundations of the proposed methodology by introducing the continuous Euler Characteristic Profile, studying its approximation and providing a geometric characterization for linear maps. Section \ref{sec:ECP_linear} presents numerical experiments on low-dimensional vector fields: investigates the topological signatures of closed orbits and linear two-dimensional systems with stationary points. Theoretical properties of ECPs, including their behaviour under dynamical equivalence and related transformations, are analysed in Section \ref{sec:prop_ECP}. Sections \ref{sec:ECP_nonlinear} - \ref{sec:TDS} demonstrate the applicability of the proposed methodology to nonlinear and higher-dimensional dynamical systems, including the Hopf bifurcation, the FitzHugh--Nagumo model, the Lorenz system, the Hénon map, and high-dimensional turbulent flow data from the Johns Hopkins Turbulence Databases. Finally, the paper concludes with a discussion of the main findings, limitations and directions for future research. Optimized algorithms for ECP computation and visualization are presented in Appendices A and B.

\subsection{Motivation and Background}
\label{sec:ideas}
The study of dynamical systems hinges on the ability to extract informative invariants that capture the long-term behavior of trajectories, the organization of phase space, and the qualitative changes induced by parameter variation. Classical invariants-such as Lyapunov exponents or the Conley index provide deep insight but often come with significant limitations. Many of them require explicit knowledge of trajectories, depend sensitively on numerical integration schemes, or are difficult to compute robustly from sampled data. Furthermore, they typically focus on specific subsets of the dynamics (e.g.\ isolated invariant sets) rather than providing a global description of the vector field.

Recent progress in TDA suggests a complementary perspective. Instead of tracking individual trajectories, one may analyses the geometric and topological organization of the vector field itself. From this viewpoint, a vector field can be studied through the structure of its direction map, the behavior of the flow under perturbation, or the organization of its integral curves. However, many topological invariants traditionally used in TDA (most notably persistent homology) are computationally expensive or difficult to scale in settings where the domain is high-dimensional or must be sampled densely in order to represent the flow reliably.

This leads us to ask what properties a useful topological descriptor for vector fields should possess.
This motivates the search for alternative topological descriptors that are:
\begin{enumerate}
    \item computationally efficient, enabling analysis on dense grids or large samples;
    \item stable under perturbations, so that small changes in the vector field do not drastically change the invariant;
    \item sensitive to dynamical structure, capable of detecting bifurcations, coherent regions, or directional organization in the flow; 
    \item applicable in both continuous and sampled settings.
\end{enumerate}

While ECCs and ECPs play a central role in this work, they are not the only descriptors we consider. We focus on them in particular because they satisfy many of the requirements outlined above: they are computationally less demanding than persistent homology, scale well to high-resolution discretizations, and naturally extend to multiparameter settings. Moreover, when applied not merely to scalar-valued filtrations but to filtrations derived from vector fields and their associated geometric features, they yield expressive signatures that vary smoothly across parameter space and respond sensitively to qualitative changes in the dynamics.

At the same time, the ECC/ECP framework is complemented here by additional constructions, such as begin--end point embeddings and densities of directions, which capture other aspects of the organization of vector fields. Together, these descriptors form a broader family of tools for qualitative analysis and comparison. Among them, Euler characteristic--based invariants provide the main topological backbone of the paper, while the remaining descriptors offer useful geometric and data-driven perspectives that may be advantageous in specific applications.

The key idea of this work is therefore to use the ECC/ECP framework as a foundation for a broader class of descriptors for vector fields. These descriptors encode how selected topological and geometric features of the field evolve under parametrized filtrations or related constructions derived from the flow. In contrast to trajectory-based or index-based methods, this approach provides a more global perspective on the vector field, remains robust under sampling noise and perturbations, and is computationally feasible even for fine discretizations of phase space.

Our motivation is twofold. First, we aim to bridge ideas from dynamical systems theory and topological data analysis by showing that Euler characteristic--based invariants can serve as practical and theoretically grounded tools for the study of vector fields. Second, we show that these invariants, together with the complementary descriptors introduced in this paper, can be effectively applied in situations where classical approaches face limitations, including parameter continuation, bifurcation detection, and comparison of vector fields. The resulting framework is simple, scalable, and broadly applicable.

\section{Preliminaries}
\label{sec:Preliminaries}
In this section we collect the basic geometric and topological ingredients used throughout the paper. Our main goal is to construct descriptors of vector fields from filtrations derived either directly from the vector values or from associated quantities such as orientation, curl, divergence, or local eigenstructure. To do so, we first introduce the basic framework underlying the proposed constructions and then recall several standard quantities associated with vector fields that will serve as natural candidates for filtration functions in later sections.

Let us begin by introducing the basic definitions used in the sequel. For clarity of exposition, most of the constructions in this section are presented first in the two-dimensional setting, which is sufficient for many of the examples considered later. At the same time, the underlying ideas extend naturally to vector fields on subsets of $\mathbb{R}^n$, and whenever the higher-dimensional form is relevant we indicate it explicitly. Thus, the two-dimensional presentation should be understood primarily as a notational simplification rather than a restriction of scope.

\begin{definition}[Cubical complex in $\mathbb{R}^n$] \label{def:cubical_complex}
Let $\mathcal{G}=\{x^{(1)}_1,\ldots,x^{(1)}_{m_1}\}\times
\{x^{(2)}_1,\ldots,x^{(2)}_{m_2}\}\times\cdots\times
\{x^{(n)}_1,\ldots,x^{(n)}_{m_n}\}\subset\mathbb{R}^n$
denote a regular Cartesian grid, where each $\{x^{(k)}_i\}_{i=1}^{m_k}, k=1,\ldots,n,$ is a finite, uniformly spaced sequence with spacing $h > 0$. Each grid point in $\mathcal{G}$ corresponds to a closed axis-aligned $n$-dimensional cube of side length $ h $.
We define an adjacency structure on the set of such cubes as follows:
\begin{enumerate}
    \item Two cubes are said to be \emph{facet-adjacent} if their corresponding grid indices differ by exactly one in one coordinate and are equal in all remaining coordinates; that is, there exists exactly one index $k$ such that
    $|i_k-i_k'|=1,$  while $i_\ell=i_\ell'$ for all $\ell\neq k.$
    \item Two cubes are said to be \emph{vertex-adjacent} if their corresponding grid indices differ by exactly one in every coordinate; that is, $|i_k-i_k'|=1, k=1,\ldots,n.$
\end{enumerate}

This structure endows the grid $\mathcal{G}$ with the combinatorics of a \emph{cubical complex}, where $n$-cells correspond to cubes, $(n-1)$-cells to their shared facets, and lower-dimensional cells arise from the intersections of these facets.
\end{definition}

In our calculation we are focusing on $\mathbb{R}^2$, so \textit{facet-adjacent} is call \textit{edge-adjacent}.

Another essential object is a notion of filtration which will be formally defined in Subsection \ref{subsec:ECC intro}. However let us list here well-known quantities connected with vector fields such as orientation angle, curl, divergence and eigenvalues, which will be later used in constructions of filtration functions. 

\begin{definition}[Orientation Angle]
The \emph{orientation angle} of the vector field $F(x, y) = (f(x, y), g(x, y))$ is defined as
\[
\theta(x,y)=\arctan\!\left(\frac{g(x,y)}{f(x,y)}\right),
\]
whenever $f(x,y)\neq 0$, which gives the direction of $F$ in the plane, measured counterclockwise from the positive $x$-axis.    
\end{definition}

Note that for numerical computations, the orientation angle is evaluated using the two-argument inverse tangent function
\[
\theta(x,y)=\operatorname{atan2}\!\big(g(x,y),\,f(x,y)\big),
\]
which is mathematically equivalent to the above definition, while additionally providing the correct quadrant of the vector and remaining well-defined when $f(x,y)=0$. Consequently, $\operatorname{atan2}$ yields a unique orientation angle over the interval $(-\pi,\pi]$.

\begin{definition}[Curl] Let $F : \mathbb{R}^2 \to \mathbb{R}^2$ be a smooth vector field, i.e.,  $F(x, y) = (f(x, y),\, g(x, y))$.
The \emph{curl} of $F$ is a scalar function representing the out-of-plane component of the vorticity vector, which measures the infinitesimal rotation of the field around each point given by the formal determinant
\[
\nabla \times F = 
\begin{vmatrix}
\frac{\partial}{\partial x} & \frac{\partial}{\partial y}\\
f(x, y) & g(x, y)
\end{vmatrix}
=
\frac{\partial g}{\partial x}(x, y) - \frac{\partial f}{\partial y}(x, y).
\]
\end{definition}
Recall that curl is typically defined for $3$-dimensional vector fields. However, as in most of the examples we deal with planar vector fields, we adapted its $2$-dimensional version.  

\begin{definition} [Divergence]
Let $ F : \mathbb{R}^2 \to \mathbb{R}^2 $ be a smooth vector field, written as $F(x, y) = (f(x, y), g(x, y))$ The \emph{divergence} of $ F $ at a point $ (x,y) $ is defined as the scalar quantity
\[
\mathrm{div}\, F(x, y) = 
\frac{\partial f}{\partial x}(x, y) + 
\frac{\partial g}{\partial y}(x, y).
\]
\end{definition}
It quantifies the infinitesimal rate of volumetric expansion or contraction of the field around the point $(x, y)$. A positive value indicates a local source, while a negative value corresponds to a sink.

Generalizing the formula for $n$-dimensional space, the divergence is given by the formula
\[
\mathrm{div}\, F(x) = \sum_{i=1}^{n} \frac{\partial f_i}{\partial x_i}(x).
\]

\begin{definition}[Jacobian Matrix and Local Eigenstructure]
The \emph{Jacobian matrix} of the field $F$ at a point $(x, y)$ is
\[
J(x, y) =
\begin{bmatrix}
\dfrac{\partial f}{\partial x} & \dfrac{\partial f}{\partial y} \\[1em]
\dfrac{\partial g}{\partial x} & \dfrac{\partial g}{\partial y}
\end{bmatrix}.
\]
The eigenvalues of $J(x, y)$, denoted $\lambda_1, \lambda_2$, characterize the local linear behavior of $F$ near $(x, y)$:
\[
\lambda_{1,2} = 
\frac{\operatorname{tr}(J) \pm \sqrt{\operatorname{tr}(J)^2 - 4\,\det(J)}}{2},
\]
where $\operatorname{tr}(J) = \frac{\partial f}{\partial x} + \frac{\partial g}{\partial y}$ and 
$\det(J) = \frac{\partial f}{\partial x}\frac{\partial g}{\partial y} - \frac{\partial f}{\partial y}\frac{\partial g}{\partial x}$.
The real parts of $\lambda_{1,2}$ indicate local stretching or compression, whereas their imaginary parts describe rotational motion around the point.
\end{definition}

The above-mentioned definitions can be easily generalized to higher-dimensional examples.

\section{Distances and Comparisons Between Dynamical Systems}
\label{sec:distamces}
The problem of comparing dynamical systems arises naturally in many areas of mathematics, physics, and data analysis. When two systems are defined on the same phase space, it is often desirable to quantify how different their dynamics are, either in terms of their vector fields, their flows, or statistical properties of their trajectories.
One of the most direct approaches is to measure pointwise differences between vector fields using norms in suitable function spaces. In this framework vector fields are treated as elements of $L_p$ and distances are defined through the corresponding norms. While such metrics provide a simple quantitative comparison, they do not necessarily capture qualitative differences in the global organization of the dynamics.
Another class of methods compares the flows generated by vector fields. Since vector fields induce dynamical systems through ordinary differential equations, one can measure the discrepancy between trajectories starting from the same initial conditions. Such comparisons are closely related to stability theory and estimates for perturbations of differential equations.
A complementary perspective is provided by statistical descriptions of dynamical systems. In many cases the long-term behavior of trajectories is characterized by invariant probability measures. Comparing dynamical systems can therefore be formulated as comparing the corresponding invariant measures using probability metrics or information-theoretic divergences.
Finally, topological methods aim to capture qualitative structures of the dynamics that remain invariant under continuous perturbations. A prominent example is the Conley index, which associates algebraic topological invariants with isolated invariant sets of dynamical systems. These invariants allow one to detect equilibria, periodic orbits and more complicated recurrent dynamics.
Despite their usefulness, many of these approaches either rely on trajectory integration, depend on detailed knowledge of invariant measures, or focus on specific invariant structures. This motivates the search for alternative descriptors that are computationally efficient, robust under perturbations, and capable of capturing global geometric organization of vector fields.

\subsection{Pointwise and Lp distances Between Vector Fields}
If the dynamical systems are continuous-time systems governed by vector fields (e.g., \( \dot{x} = f(x) \) and \( \dot{x} = g(x) \)), where $f,g: X\to \mathbb{R}^n$ are continuous, a natural metric is the \textit{pointwise distance} between the vector fields:
\[
d_{\text{VF}}(f, g) = \sup_{x \in X} \| f(x) - g(x) \|
\]
where \( X \) is the phase space and \( \| \cdot \| \) is a norm (e.g., the Euclidean norm). This measures the maximum difference in the vector fields at any point in the phase space. It is also possible to define several standard measures between $f$ and $g$, such as \(L^1\) and \(L^2\).

The \textit{$L^1$ distance} between $f$ and $g$ is defined by
\[
d_{L^1}(f,g) = \int_X \| f(x) - g(x) \| \, d\mu(x),
\]
where $\mu$ is a $\sigma$-finite measure on $X$ and $f$ and $g$ are measurable with respect to $\mu$, e.g. a Lebesgue measure. This metric quantifies the total absolute deviation of the vector fields over the domain and is sensitive to differences on sets of positive measure, while isolated pointwise errors have limited impact.

For measurable functions $f,g : X \to \mathbb{R}^n$, the \textit{$L^2$ distance} is defined by
\[
d_{L^2}(f,g) = \left( \int_X \| f(x) - g(x) \|^{2} \, d\mu(x) \right)^{1/2}.
\]
This metric captures the root-mean-square deviation and assigns greater weight to large local differences due to the squaring.

In general, the choice of metric depends on the analytical framework and the type of stability or convergence one seeks to establish.  
For instance, convergence in \( L^\infty \) implies convergence in both \( L^1 \) and \( L^2 \) if $\mu(X) < \infty$), but not conversely.  
In dynamical systems theory, the \( L^\infty \) distance often provides the strongest control, ensuring that all trajectories experience uniformly small perturbations of the underlying vector field.

\subsection{Maximum Norm of the Flow Difference}
If we are interested in the flows \( \phi^f_t(x) \) and \( \phi^g_t(x) \) generated by the two systems, a natural metric is the \textit{supremum} of the pointwise difference between the flows over time:
\[
d_{\text{flow}}(f, g) = \sup_{x \in X, t \in [0,T]} \| \phi^f_t(x) - \phi^g_t(x) \|
\]
This measures the maximum distance between the trajectories of the two systems over time for points in the phase space.

\subsection{Wasserstein Distance (Optimal Transport)}

An alternative approach compares probability measures associated with dynamical systems using optimal transport distances such as the Wasserstein metric \cite{villani2008}. For instance, if two vector fields generate flows $\phi_t^f$ and $\phi_t^g$, one may compare invariant or otherwise dynamically relevant probability measures associated with these systems.

Let $(X,d)$ be a metric space, and let $\mu,\nu \in \mathcal P_p(X)$ be Borel probability measures with finite $p$-th moments. The $p$-Wasserstein distance between $\mu$ and $\nu$ is defined by
\[
W_p(\mu, \nu)
=
\left(
\inf_{\gamma \in \Gamma(\mu, \nu)}
\int_{X \times X} d(x,y)^p \, d\gamma(x,y)
\right)^{1/p}.
\]
Here $\Gamma(\mu,\nu)$ denotes the set of \emph{couplings} of $\mu$ and $\nu$, that is, probability measures on $X\times X$ whose marginals are $\mu$ and $\nu$. Intuitively, a coupling specifies how mass is transported from $\mu$ to $\nu$, while the Wasserstein distance measures the minimal transportation cost required to transform one distribution into the other.

The finite $p$-th moment assumption is needed to ensure that the above quantity is well defined and finite. In the dynamical setting, this approach also depends on the choice of measures being compared, since a given dynamical system may admit more than one invariant probability measure.

A closely related optimal-transport construction is used in topological data analysis to compare persistence diagrams (see~\cite{edelsbrunner2010}). Let $\mathrm{Dgm}_1$ and $\mathrm{Dgm}_2$ be two persistence diagrams, and let \[ \Delta = \{(b,d)\in\mathbb{R}^2:b=d\} \] denote the diagonal, considered with infinite multiplicity. The inclusion of the diagonal allows persistence features that do not have a corresponding feature in the other diagram to be matched to the diagonal. For $1\leq p<\infty$, the $p$-Wasserstein distance between the persistence diagrams is defined as \[ W_p(\mathrm{Dgm}_1,\mathrm{Dgm}_2) = \left( \inf_{\gamma} \sum_{x\in \mathrm{Dgm}_1\cup\Delta} \|x-\gamma(x)\|_\infty^p \right)^{1/p}, \] where the infimum is taken over all bijections \[ \gamma: \mathrm{Dgm}_1\cup\Delta \longrightarrow \mathrm{Dgm}_2\cup\Delta, \] and $\|\cdot\|_\infty$ denotes the $\ell_\infty$ norm on $\mathbb{R}^2$. The diagonal points may be matched to one another at zero cost, so that only the matching of off-diagonal features contributes to the distance. The $p$-Wasserstein distance aggregates the matching costs over all persistence features and is therefore sensitive to the combined contribution of multiple differences between the diagrams. The bottleneck distance is defined using the same class of matchings, but records only the largest matching cost: \[ d_B(\mathrm{Dgm}_1,\mathrm{Dgm}_2) = \inf_{\gamma} \sup_{x\in \mathrm{Dgm}_1\cup\Delta} \|x-\gamma(x)\|_\infty. \] Thus, unlike the $p$-Wasserstein distance, which aggregates the costs over all matched points, the bottleneck distance is determined by the worst matched pair and captures the largest topological discrepancy between the diagrams~\cite{cohen2007}. Equivalently, the bottleneck distance is the smallest value $\varepsilon$ for which every feature in one persistence diagram can be matched either to a feature in the other diagram or to the diagonal, with all matched pairs lying within $\varepsilon$ in the $\ell_\infty$ metric. Features with low persistence are represented by points close to the diagonal and may therefore be matched to it at low cost.

\subsection{Divergence of Invariant Measures}
Another approach to comparing dynamical systems focuses on the statistical properties of their trajectories. Let $\phi_t^f$ and $\phi_t^g$ denote the flows generated by the vector fields
$f$ and $g$ on a compact phase space $X$. 
A probability measure $\mu$ on $X$ is called \emph{invariant} for the flow $\phi_t$ if
$\mu(\phi_t^{-1}(A)) = \mu(A)$ for every measurable set $A \subset X$ and every $t$ \cite{walters}.
Invariant measures describe statistical properties of trajectories of the dynamical system. In general, however, a system may admit many invariant measures.
To obtain a physically meaningful comparison between systems, it is natural to consider \emph{physical (SRB) measures} \cite{young2002} . A probability measure $\mu$ is called a physical measure if for a set of initial conditions of positive Lebesgue measure the time averages along trajectories converge to spatial averages with respect to $\mu$, i.e.
\[
\lim_{T \to \infty} \frac{1}{T}\int_0^T \varphi(x(t))\,dt =
\int_X \varphi \, d\mu
\]
for every continuous observable $\varphi$.
If the systems generated by $f$ and $g$ admit physical measures $\mu_f$ and $\mu_g$, we can compare them using a metric on probability measures. For instance, the total variation distance is defined as
\[
d_{\mathrm{TV}}(\mu_f,\mu_g) =
\sup_{A \subset X} |\mu_f(A) - \mu_g(A)|.
\]
This provides a quantitative way to measure how differently the two systems distribute long–time trajectories over the phase space.

\subsection{The Conley Index}
The \textit{Conley index}, introduced by Charles Conley \cite{conley1978isolated}, is a topological invariant associated with isolated invariant sets of dynamical systems. It generalizes the Morse index to non-gradient flows and provides a powerful framework for detecting and classifying invariant dynamics such as equilibria, periodic orbits and chaotic sets.

Let $\phi: \mathbb{R} \times X \to X$ be a continuous flow on a locally compact metric space.  
A compact set $N \subset X$ is called an \emph{isolating neighborhood} if the maximal invariant set inside $N$,
\[
\operatorname{Inv}(N, \phi) = \{x \in N : \phi(t,x) \in N \ \forall t \in \mathbb{R}\},
\]
lies in the interior of $N$.

To define the Conley index one introduces an \emph{index pair} $(N,L)$ consisting of compact sets $L \subset N$ such that
\begin{enumerate}
    \item[(i)] $S = \operatorname{Inv}(N,\phi) \subset \operatorname{int}(N\setminus L)$,
    \item[(ii)] $L$ is positively invariant relative to $N$, and
    \item[(iii)] every trajectory that leaves $N$ must first pass through $L$.
\end{enumerate}
Intuitively, the set $L$ captures the part of the boundary of $N$ through which trajectories exit the neighborhood.

The \emph{homotopy Conley index} of the isolated invariant set 
$S = \operatorname{Inv}(N, \phi)$ is defined as the pointed homotopy type
\[
h(S) = [\,N / L\,],
\]
and the corresponding \emph{homology Conley index} as
\[
CH_\ast(S) = H_\ast(N, L).
\]
A fundamental property of the Conley index is the \emph{continuation property}: the index is invariant under continuous deformations of the dynamical system as long as the isolated invariant set persists.

Consequently, if the Conley index of an isolated invariant set changes when a system parameter varies, the invariant set cannot persist unchanged and a bifurcation must occur. This makes the Conley index a useful tool for detecting qualitative changes in dynamical systems.

However, computing the Conley index typically requires constructing isolating neighborhoods and index pairs, which can be challenging for high-dimensional or data-driven systems.

Recent work of Marian Mrozek and collaborators has developed a combinatorial counterpart of Conley theory that is particularly well suited for finite and data-derived models of dynamics. Starting from Forman's combinatorial vector fields, they established a rigorous bridge between combinatorial and classical vector-field dynamics by associating to a combinatorial vector field a multivalued dynamical system on the geometric realization of the underlying complex, and proving a correspondence of isolated invariant sets, Conley indices, Morse decompositions, and Conley--Morse graphs. This program was subsequently extended to the Conley--Morse--Forman theory of combinatorial multivector fields, first on Lefschetz complexes and later in a more general form on finite topological spaces. From the point of view of applications, these developments are especially relevant because they make it possible to compute a discrete Conley-type invariant directly from finite vector-field data, once such data have been converted into an appropriate combinatorial vector or multivector field. In the present paper, this combinatorial Conley-index framework will serve as one of the principal reference methods against which we compare our descriptors \cite{batko2020linking,desjardinscote2024finite,lipinski2023generalized,mrozek2017cmf}.

\section{New methods for measuring distances between dynamics}
\label{sec:new_met}
In many applications, information about a dynamical system is available not only through trajectories, but also through local measurements or numerical samples of the underlying vector field. In such settings, estimating directional information at many points of the phase space may be substantially easier, more robust, and more economical than reconstructing trajectories over long time intervals. This is especially relevant in data-driven problems, where trajectory information may be noisy, incomplete, or computationally expensive to obtain, whereas local vector information can often be inferred directly from simulations, measurements, or local regression procedures.

These considerations motivate the development of methods that compare dynamics directly at the level of vector fields. Ideally, such methods should not depend on the fields being defined on exactly the same phase space, on a common discretization, or in perfectly aligned coordinates. From this perspective, classical pointwise distances such as $L_p$ norms are often too restrictive, since they rely on a shared embedding and a reliable point-to-point correspondence between domains. The descriptors introduced below are designed to overcome this limitation by encoding vector fields through geometric and topological features that remain comparable across different representations. Several of them require only that the fields take values in spaces of the same dimension, while others can be adapted even more flexibly. In this way, they reduce or eliminate the need for explicit alignment while avoiding exclusive reliance on long trajectory data.

\subsection{Densities of Directions (DoD)}
\label{sec:DoD}

The \emph{Densities of Directions} (DoD) method associates to a vector field a distribution on the unit sphere that captures the global organization of its directions in the vector field independently of their magnitudes.

More precisely, let us consider the projection: 
\[
\nu : \mathbb{R}^n \setminus \{0\} \to S^{n-1},
\qquad
\nu(v)=\frac{v}{\|v\|},
\]
i.e. $\nu$ denotes the normalization map, which assigns to each nonzero vector its direction on the unit sphere. In the context of a vector field, this map forgets the magnitude of the vector and retains only its orientation. Such a reduction is natural when one is interested primarily in the directional organization of the dynamics rather than in local speed. The map $\nu$ is not defined at zero vectors, that is, at stationary points of the field; in practice, these points are excluded from the construction in this method. 

Given a vector field $f$ on a phase space $X \subset \mathbb{R}^n$, one may apply the normalization map pointwise and thus associate to the field a subset of the sphere $S^{n-1}$. In the sampled setting, for a finite set $S \subset X$, this yields the collection of directions
\[
\left\{ \frac{f(x)}{\|f(x)\|} : x \in S,\ f(x)\neq 0 \right\} \subset S^{n-1},
\]
which may be viewed as a discrete directional distribution of the field. In the continuous setting, the same construction produces a continuous family of directions indexed by points of the phase space, and therefore gives rise, at least formally, to a directional density on the sphere. In this way, both sampled and continuous vector fields can be transformed into objects on $S^{n-1}$ that summarize how directions are distributed throughout the domain.

To obtain an actual probability distribution on the sphere, one may normalize the projected sample by assigning equal weights to all sample points. Thus, for a finite sample
\[
S=\{x_1,\dots,x_N\}\subset X
\]
and a vector field $f$ satisfying $f(x_i)\neq 0$ for all $i$, we define the empirical directional measure
\[
\mu_f^S := \frac{1}{N}\sum_{i=1}^N \delta_{\nu(f(x_i))},
 \]
which is a probability measure on $S^{n-1}$.

In the continuous setting, let $\mu$ be a reference measure on $X$ (for instance, Lebesgue measure restricted to $X$). To avoid instability near stationary points, fix a threshold $\tau>0$ and consider the subset
\[
X_{\tau}:=\{x\in X:\|f(x)\|\ge \tau\}.
\]
The corresponding directional distribution on the sphere is then defined by normalizing the measure induced by the map $x\mapsto \nu(f(x))$. More precisely, for every measurable set $A\subset S^{n-1}$, we set
\[
\mu_f(A):=
\frac{\mu\big(\{x\in X_{\tau}:\nu(f(x))\in A\}\big)}
{\mu(X_{\tau})},
\qquad
\nu(v):=\frac{v}{\|v\|}.
\]
Thus, $\mu_f$ is a probability measure on $S^{n-1}$ obtained by recording the fraction of points in $X_{\tau}$ whose normalized vectors point into the set $A$.

In both cases, the total mass is equal to one, and the resulting objects arising from different vector fields may be compared using probability metrics on the sphere, such as the Wasserstein distance.

Under these assumptions, the projected sample on the sphere provides a stable and consistent approximation of the directional distribution of the vector field. The following lemma makes this precise in terms of Wasserstein stability.

\begin{lemma}
Let $(X,\mu)$ be a probability space, let $1\le p<\infty$, and let $f,g:X\to\mathbb{R}^n$ be measurable vector fields such that
\[
\|f(x)\|\ge \tau,
\qquad
\|g(x)\|\ge \tau
\]
for $\mu$-almost every $x\in X$, for some $\tau>0$. 
Then
\[
W_p(\mu_f,\mu_g)\le \frac{2}{\tau}\,\|f-g\|_{L^p(\mu)}.
\]
In particular, small perturbations of the vector field in $L^p$ produce small perturbations of the corresponding directional distributions in the $p$-Wasserstein metric.
\end{lemma}

An entirely analogous argument applies in the discrete setting, where $\mu$ is replaced by the empirical measure on the sample $S$. Thus, provided the vectors are bounded away from zero or regularized below a fixed threshold, the density-of-directions construction is stable under perturbations of the underlying sampled vector field.

There are several ways to further process the normalized directions obtained from a vector field. In both the discrete and continuous settings, the normalization procedure produces directional data on the sphere $S^{n-1}$: in the sampled case this takes the form of finitely many projected directions, while in the continuous case it gives a spherical directional measure or density associated with the normalized field. From such data one may construct a smooth probability density on $S^{n-1}$ by applying a suitable spherical kernel and normalizing the result. A natural choice is the von~Mises--Fisher kernel, which may be viewed as the spherical analogue of the Gaussian distribution. For $\mu \in S^{n-1}$, its density at a point $x \in S^{n-1}$ is given by
\[
K_\kappa(x,\mu)=C_n(\kappa)\exp\!\bigl(\kappa\langle x,\mu\rangle\bigr),
\]
where $\langle x,\mu\rangle$ denotes the Euclidean inner product, $\kappa>0$ is a concentration parameter controlling the spread of the kernel, and $C_n(\kappa)$ is the normalizing constant ensuring that the density integrates to one over $S^{n-1}$. In the case of the unit sphere $S^2$, this takes the explicit form
\[
K_\kappa(x,\mu)=\frac{\kappa}{4\pi\sinh(\kappa)}\exp\!\bigl(\kappa\langle x,\mu\rangle\bigr).
\]

In the present work we use this technique to operate on two dimensional vector filed. In this case  the normalized directions lie on the unit circle $S^1$. Consequently, the directional density is estimated using the von Mises kernel, which is the one-dimensional instance of the von Mises–Fisher family.

Applying such kernels to the directional data yields, in both the discrete and continuous settings, a smoothed spherical representation of the vector field in the form of a probability density on $S^{n-1}$. In the discrete case this amounts to a finite weighted sum of kernels centered at the projected sample points; in the continuous case it is obtained by averaging the same kernel against the corresponding directional measure on the sphere. The resulting density may be interpreted as a generalized histogram of directions and used as a feature for comparing vector fields. Distances between the corresponding spherical densities then provide surrogate distances between the original fields.

Moreover, since this representation lives on the sphere, it can also be compared modulo rotations by allowing the natural action of the rotation group $SO(n)$ on $S^{n-1}$. In this way, part of the burden associated with rotational misalignment or changes of coordinates may be transferred from the original phase space to a simpler optimization problem on the sphere. Thus, the density-of-directions construction turns a vector field into a compact and flexible feature, suitable both for direct comparison and for comparison up to selected geometric transformations. In practice, this comparison may be carried out either directly in the space of densities, using metrics such as $L^p$ or Wasserstein distances, or after optimization over rotations of the sphere.

\subsection{Begin-End Point Embedding (BEPE)}
\label{sec:BEPE}
A simple way to encode a vector field is to record, at each point of the phase space, both the location of the point and the value of the vector attached to it. This leads to the \emph{Begin--End Point Embedding} (BEPE), which represents the field as a subset of a higher-dimensional Euclidean space.

More preciselly, let $X\subset \mathbb{R}^n$ be a phase space and let $f:X\to\mathbb{R}^n$ be a vector field. We define the BEPE associated with $f$ by
\[
\operatorname{BEPE}_f : X \to \mathbb{R}^{2n},
\qquad
\operatorname{BEPE}_f(x):=(x,f(x)).
\]
In the sampled setting, if $S\subset X$ is a finite sample, then the corresponding embedded point cloud is
\[
\operatorname{BEPE}_f(S):=\{(x,f(x)) : x\in S\}\subset\mathbb{R}^{2n}.
\]
Thus, the original vector field is transformed into a geometric object in $\mathbb{R}^{2n}$ that simultaneously encodes where the vector is observed and what its value is.

This representation is particularly convenient because it allows one to use a broad range of comparison tools developed for point clouds and empirical measures. In particular, two vector fields may be compared by applying tools from topological data anaylyis, optimal transport methods, such as the Earth Mover's Distance \cite{EMD}, to their embedded representations. One may also use goodness-of-fit tests or other statistical procedures in the ambient space $\mathbb{R}^{2n}$ in order to quantify similarity between fields. The main price paid for this flexibility is the doubling of dimension: a vector field on $\mathbb{R}^n$ is represented in $\mathbb{R}^{2n}$. Nevertheless, the construction is simple, direct and compatible with both geometric and statistical comparison methods.

A useful feature of the begin--end point embedding is its stability under perturbations of the vector field. When two vector fields are defined on the same phase space, the first coordinate in the embedding remains unchanged and all variation is confined to the second coordinate. As a consequence, small perturbations of the vector field induce correspondingly small perturbations of the embedded representation in Wasserstein distance.

\begin{lemma}
Let $S=\{x_1,\dots,x_N\}\subset X$ be a finite sample and let $f,g:S\to\mathbb{R}^n$ be two sampled vector fields. Define the empirical BEPE measures
\[
\mu_f^S:=\frac1N\sum_{i=1}^N \delta_{(x_i,f(x_i))},
\qquad
\mu_g^S:=\frac1N\sum_{i=1}^N \delta_{(x_i,g(x_i))}.
\]
Then, with respect to the Euclidean metric on $\mathbb{R}^{2n}$,
\[
W_p(\mu_f^S,\mu_g^S)
\le
\left(
\frac1N\sum_{i=1}^N \|f(x_i)-g(x_i)\|^p
\right)^{1/p}
\]
for every $1\le p<\infty$.
\end{lemma}

\begin{proof}
Consider the coupling that matches $(x_i,f(x_i))$ with $(x_i,g(x_i))$ for each $i=1,\dots,N$. Since the first coordinates agree, the transport cost between the matched points is
\[
\|(x_i,f(x_i))-(x_i,g(x_i))\| = \|f(x_i)-g(x_i)\|.
\]
Therefore,
\[
W_p(\mu_f^S,\mu_g^S)^p
\le
\frac1N\sum_{i=1}^N \|f(x_i)-g(x_i)\|^p,
\]
and taking the $p$-th root yields the claim. \qed
\end{proof}

Thus, the BEPE representation preserves quantitative proximity of sampled vector fields while translating the comparison problem into a geometric one in the ambient space $\mathbb{R}^{2n}$.

\subsection{Euler Characteristic Curves and Profiles (ECCs and ECPs)}\label{subsec:ECC intro}
The Euler characteristic is a fundamental topological invariant that provides a compact summary of the global structure of a space. In the context of filtered complexes, it gives rise to simple yet expressive descriptors obtained by tracking its value across the filtration parameter space. Owing to its additivity and low computational complexity, the Euler characteristic is particularly well suited for scalable topological data analysis. In this section we introduce two such descriptors: the \emph{Euler Characteristic Curve} (ECC), associated with single-parameter filtrations and the \emph{Euler Characteristic Profile} (ECP), defined for multiparameter filtrations~\cite{dlotko2023ecc}. Despite their conceptual simplicity, these invariants have proved useful in a range of applications, including medical data analysis~\cite{dlotko2023ecc}, fluid dynamics~\cite{fluid} and granular media analysis~\cite{ardanza2025granular}.

The constructions introduced below are formulated at the level of filtered cell complexes. We therefore begin by recalling the underlying combinatorial framework and the associated notion of filtration.

\begin{definition}[Cell complex]
A \emph{cell complex} $K$ is a finite collection of cells such that:
\begin{enumerate}
    \item if $\sigma \in K$, then every face of $\sigma$ also belongs to $K$;
    \item the intersection of any two cells of $K$ is either empty or a union of cells in $K$.
\end{enumerate}
A cell $\sigma \in K$ is $d$-dimensional if it is homeomorphic to an open unit ball in Euclidean space $\mathbb{R}^d$.  A complex $K$ is $d$-dimensional if its highest-dimension cells have a dimension of $d$. 
\end{definition}

\begin{definition}[Euler characteristic]
Let $K$ be a cell complex.  
The \emph{Euler characteristic} of $K$ is defined as the alternating sum
\[
\chi(K) \;=\; \sum_{d \ge 0} (-1)^d\, |K_d|,
\]
where $|K_d|$ denotes the number of $d$-dimensional cells of $K$.   
Equivalently, by the Euler-Poincar\'e formula, one may express it in terms of Betti numbers $\beta_d(K)$:
\begin{equation}
  \chi(K) \;=\; \sum_{d \ge 0} (-1)^d \, \beta_d(K),
\end{equation}
where $\beta_d(K)$ is the rank of the $d$-th homology group, see~\cite{edelsbrunner2010} for details.  
This dual combinatorial--homological nature is one of the main reasons why the Euler characteristic is attractive in applications: it is inexpensive to compute yet tightly linked to the deeper topological structure of the data.

\end{definition}

\begin{definition}[Filtration]
Let $K$ be a cell complex and let $P$ be a partially ordered set (poset). A \emph{$P$-filtration} of $K$ is a family of subcomplexes
\[
\{ K_p \}_{p \in P}
\]
such that
\[
p \le q \quad \Rightarrow \quad K_p \subseteq K_q .
\]
Intuitively, the parameter $p$ determines which cells of $K$ have entered the complex at a given stage of the filtration. 
\end{definition}
When the parameter set is totally ordered, the filtration becomes a nested sequence of complexes.

\begin{definition}[Single-parameter filtration]
A \emph{single-parameter filtration} is a filtration indexed by a totally ordered set, typically $P=\mathbb{R}$ or $P=\mathbb{Z}$. Thus, for any $s,t\in P$,
\[
s\le t \quad \Longrightarrow \quad K_s\subseteq K_t.
\]
When $K$ is finite, only finitely many distinct subcomplexes appear along the filtration.
\end{definition}

\begin{definition}[Multiparameter filtration]
Let $P = \mathbb{R}^n$ equipped with the product partial order
\[
(a_1,\dots,a_n) \le (b_1,\dots,b_n)
\quad \text{iff} \quad
a_i \le b_i \text{ for all } i .
\]
A \emph{multiparameter filtration} is a $P$-filtration
\[
\{K_p\}_{p \in \mathbb{R}^n}.
\]
Such filtrations are often called \emph{multifiltrations}. In this work we use the terms \emph{multiparameter filtration} and \emph{multifiltration} interchangeably. In particular, when $P=\mathbb{R}^2$ the filtration is commonly referred to as a \emph{bifiltration}.
\end{definition}

The multifiltration need not be defined directly by the components of the vector field. More generally, one may choose any collection of real-valued features derived from the field, \[ \Psi_F=(\psi_1,\ldots,\psi_m):D\to\mathbb{R}^m, \] and use their values as the filtration parameters. Depending on the application, these features may include the components of the vector field, its magnitude or orientation, divergence, curl, or quantities derived from the Jacobian. A single scalar feature defines a single-parameter filtration, whereas a collection of $m$ features defines an $m$-parameter multifiltration.

\begin{definition}[1-critical filtration]
Let $\{K_p\}_{p\in P}$ be a $P$-filtration of a cell complex $K$. We say that the filtration is \emph{1-critical} if for every cell $\sigma\in K$ there exists a unique minimal parameter value $p_\sigma\in P$ such that
\[
\sigma\in K_p \quad \Longleftrightarrow \quad p_\sigma \le p .
\]
In other words, each cell enters the filtration at a unique birth parameter. In case when $\sigma$ appears at more than one non comparable values of filtration, the multifiltration is called multi--critical.
\end{definition}

Let us point out that since in this work  multifiltrations are computed from vector fields, the multi--critical case is  not considered here.

\begin{definition}[Euler Characteristic Curve (ECC)]
Let $K$ be a cell complex equipped with a 1-critical filtration indexed by $\mathbb{R}$.
For each $t\in\mathbb{R}$, let $K_t$ denote the corresponding subcomplex.
The \emph{Euler Characteristic Curve} is the function
\[
\mathrm{ECC}(t)=\chi(K_t), \qquad t\in\mathbb{R}.
\]
\end{definition}

The ECC admits a particularly simple incremental representation. Whenever a cell $\sigma$ of dimension $d$ enters the filtration at its birth parameter $p_\sigma$, the Euler characteristic changes by $(-1)^d$.

\begin{definition}[Euler Characteristic Profile (ECP)]
Let $K$ be a cell complex equipped with a multiparameter filtration indexed by a poset $P$ (typically $P=\mathbb{R}^n$).
For each parameter value $p\in P$ we consider the corresponding subcomplex $K_p$ and define the \emph{Euler Characteristic Profile}
\[
\mathrm{ECP}(p) \;=\; \chi(K_p), \qquad p \in P.
\]
Thus, ECP generalizes ECC to arbitrary multiparameter filtrations, providing an integer-valued multidimensional invariant.
\end{definition}

Euler characteristic curves and profiles provide computationally efficient topological descriptors. In contrast to persistent homology, their evaluation requires only counting cells with alternating signs, which makes them particularly attractive for large-scale or high-dimensional datasets \cite{dlotko2023ecc} 

\subsection{Continuous Euler Characteristic and its approximation}
The ECP is a topological descriptor associated with a vector field through the Euler characteristics of its coordinatewise sublevel sets. This construction may be considered both in discrete and in continuous settings, and in the one dimensional case it reduces to the ECC. In this subsection, we define its continuous version, compute several elementary examples explicitly, and then show how it can be approximated from finite cubical data.

\begin{definition}[Continuous Euler Characteristic Curve (cECC)]
Let $X \subset \mathbb{R}^n$ be a compact domain and let $f : X \to \mathbb{R}^m$ be a continuous map.
For a point $y \in \mathbb{R}^m$, define the coordinatewise sublevel set
\[
    D^f_y \;=\; \{\, x \in X \mid f(x) \leq y \,\},
\]
where $f(x) \leq y$ means $f_i(x) \leq y_i$ for all coordinates $i=1,\dots,m$.
Assume that each sublevel set $D^f_y$ has finite Euler characteristic.
The \emph{continuous Euler Characteristic Profile} (cECP) of $f$ is the function
\[
    \mathrm{cECP}_f : \mathbb{R}^m \to \mathbb{Z},
    \qquad
    \mathrm{cECP}_f(y) \;=\; \chi(D^f_y).
\]
In the scalar case $m=1$, we write
\[
    \mathrm{cECC}_f(t) \;=\; \chi\bigl(f^{-1}((-\infty,t])\bigr),
\]
and refer to $\mathrm{cECC}_f$ as the \emph{continuous Euler Characteristic Curve}.
\end{definition}

\subsubsection{General geometric characterization of the cECP for linear maps}\label{General}

For simple classes of vector fields, the continuous ECP, cECP, can be computed analytically and characterized exactly. We illustrate this in the case of linear maps on rectangular domains, which provides a useful baseline for the more complex examples considered later.

Consider a rectangular domain
\[
D=[x_{\min},x_{\max}] \times [y_{\min},y_{\max}]=[w_1,w_2]\times[w_3,w_4]
\subset\mathbb{R}^2
\]
in the $(x_1,x_2)$-plane and a linear transformation
$f:\mathbb{R}^2\to\mathbb{R}^2$:
\[
f(x)=Mx,
\]
where
\[
M=
\begin{pmatrix}
a & b\\
c & d 
\end{pmatrix},
\qquad
x=(x_1,x_2),
\qquad 
f(x)=(y_1,y_2)\in\mathbb{R}^2.
\]

We ask for which values of $(y_1, y_2) \in \mathbb{R}^2$ the cECP satisfies
\[
\mathrm{cECP}_f(y_1, y_2)=1.
\]
 Note that cECP may attain only the values $0$ or $1$. More precisely, 
\[
\mathrm{cECP}_f(y_1, y_2)=1
\iff
\underset{ x\in D}\exists
\text{ such that }
f(x)\le y.
\]
Notice that, whenever $\det(M)\neq 0$, the image $f(D)$ of a rectangular domain under a linear transformation is a parallelogram, as schematically illustrated in Fig. \ref{fig:caseM0}. If $\det(M)=0$, then the image degenerates into a line segment.
\begin{figure}[h]
    \centering
    \includegraphics[width=0.6\textwidth]{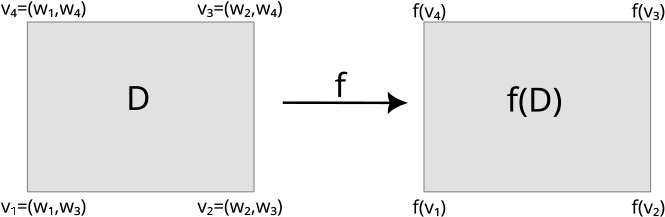}
    \caption{Rectangular domain under a linear transformation}
    \label{fig:caseM0}
\end{figure}

The geometry of the cECP transition region is completely determined by the relative position of the transformed vertices
$f(v_1),
f(v_2),
f(v_3),
f(v_4),$ 
where $\{v_1,v_2,v_3,v_4\} = \mathrm{Vert}(D).$ 

Depending on the geometry of the image $f(D)$, three different cases may occur. The three cases described below illustrate all the possibilities for $\det(M)\neq 0$.

\paragraph{Case 1: Single minimal vertex.}
Suppose that there exists a unique point $f(v^*) = (y_1^{min}, y_2^{min})$,  where $v^*\in\{v_1,v_2,v_3,v_4\}$ (without the loss of generality let it be $f(v_1)$ shown in Fig. \ref{fig:casesM=/=0} (a)), among the four vertices $\big(f(v_1), f(v_2), f(v_3),f(v_4)\big)$, such that for every $(y_1, y_2) \in f(D)$ we have 
\[
f_1(v_i)\ge f_1(v^*),
\qquad
f_2(v_i)\ge f_2(v^*),
\qquad
i\in\{1,2,3,4\}.
\]
In other words, the same transformed vertex minimizes simultaneously both coordinates. Then
\[
\mathrm{cECP}_f(y_1,y_2)=
\begin{cases}
1,
&
y_1\ge y_1^{\min},
\quad
y_2\ge y_2^{\min},
\\[1ex]
0,
&
\text{otherwise}.
\end{cases}
\]
Geometrically, the transition region is an orthant-type corner generated by the minimal vertex as shown in Fig.\ref{fig:casesM=/=0} (b). Note that this case also covers the situation illustrated in Fig.\ref{fig:caseM0}.

\paragraph{Case 2: Two neighboring minimal vertices.}
Suppose there exist two distinct neighboring points  $v^{1,*},v^{2,*}\in\{v_1,v_2,v_3,v_4\}$, $ v^{1,*}\neq v^{2,*}$ such that for every $(y_1, y_2) \in f(D)$ we have 
\[
f_1(v_i)\ge f_1(v^{1,*}),
\qquad
f_2(v_i)\ge f_2(v^{2,*}),
\qquad
i\in\{1,2,3,4\}
\]
and $v^{1,*}$ with $v^{2,*}$ form an edge of $D$ i.e. the vertex which minimizes $y_1$ is different than the one which minimizes $y_2$, but they are adjacent in the rectangular domain $D$. Then their images form an edge of the parallelogram $f(D)$.

Let $y_2=Ay_1+B$ be the equation of the affine line joining the points 
$f(v^{1,*})=f(v^{1,*}_1, v^{1,*}_2)$ and $f(v^{2,*})=f(v^{2,*}_1, v^{2,*}_2)$ then
\[
\mathrm{cECP}_f(y_1,y_2)=
\begin{cases}
1,
&
y_1\ge y_1^{\min},
\quad
y_2\ge y_2^{\min},
\quad
y_2\ge Ay_1+B,
\\[1ex]
0,
&
\text{otherwise}.
\end{cases}
\]
where $y_1^{\min}: = f_{1}(v^{1,*})=(av^{1,*}_1+bv^{1,*}_2$ and $y_2^{\min} :=  f_{2}(v^{2,*})=cv^{2,*}_1+dv^{2,*}_2$.
Geometrically, the transition region is a truncated affine corner bounded by a single oblique edge as shown in Fig. \ref{fig:casesM=/=0} (d).

\vspace{0.4cm}
Note that if points $v^{1,*}$ and $v^{2,*}$ coincide, then \textit{Case 2 }reduces to \textit{Case 1}.

\paragraph{Case 3: Two non-neighboring minimal vertices.}
Finally, suppose there exist two distinct non-neighboring points
$v^{1,*},v^{2,*}\in\{v_1,v_2,v_3,v_4\}$, $ v^{1,*}\neq v^{2,*}$ such that for every $(y_1, y_2) \in f(D)$ we have 
\[
f_1(v_i)\ge f_1(v^{1,*}),
\qquad
f_2(v_i)\ge f_2(v^{2,*}),
\qquad
i\in\{1,2,3,4\}
\]
but $v^{1,*}$ and $v^{2,*}$ are not adjacent in $D$. Then two affine boundary edges contribute to the transition region.

Let $\tilde v$ be one of the other two remaining vertices from the set $\{v_1,v_2,v_3,v_4\}$ which image $f(\tilde v)$ is closer to the boundary of the area $\big\{\{y_1, y_2\}: y_1\ge y_1^{\min}, y_2\geq y_2^{\min} \big\}$ than the image of the second remaining vertex.

Let $y_2=A y_1+B$ be the affine equation of the line joining $f(v^{1,*})$ with
$f(\tilde v),$ and let $y_2=C y_1+D$ be the affine equation of the line joining
$f(v^{2,*})$ with $f(\tilde v)$. Then
\[
\mathrm{cECP}_f(y_1,y_2)=
\begin{cases}
1,
&
y_1\ge y_1^{\min},
\quad
y_2\ge y_2^{\min},
\quad
y_2\ge Ay_1+B,
\quad
y_2\ge Cy_1+D,
\\[1ex]
0,
&
\text{otherwise}.
\end{cases}
\]
Geometrically, this corresponds to a double-truncated affine corner generated by two distinct oblique boundary edges as shown in Fig. \ref{fig:casesM=/=0} (f).

\begin{figure}[h]
    \centering
    \subfigure[]{\includegraphics[width=0.3\textwidth]{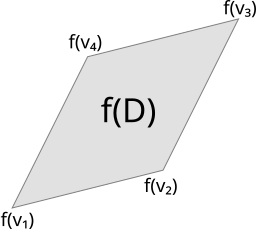}}
    \hspace{2cm}
    \subfigure[]{\includegraphics[width=0.35\textwidth]{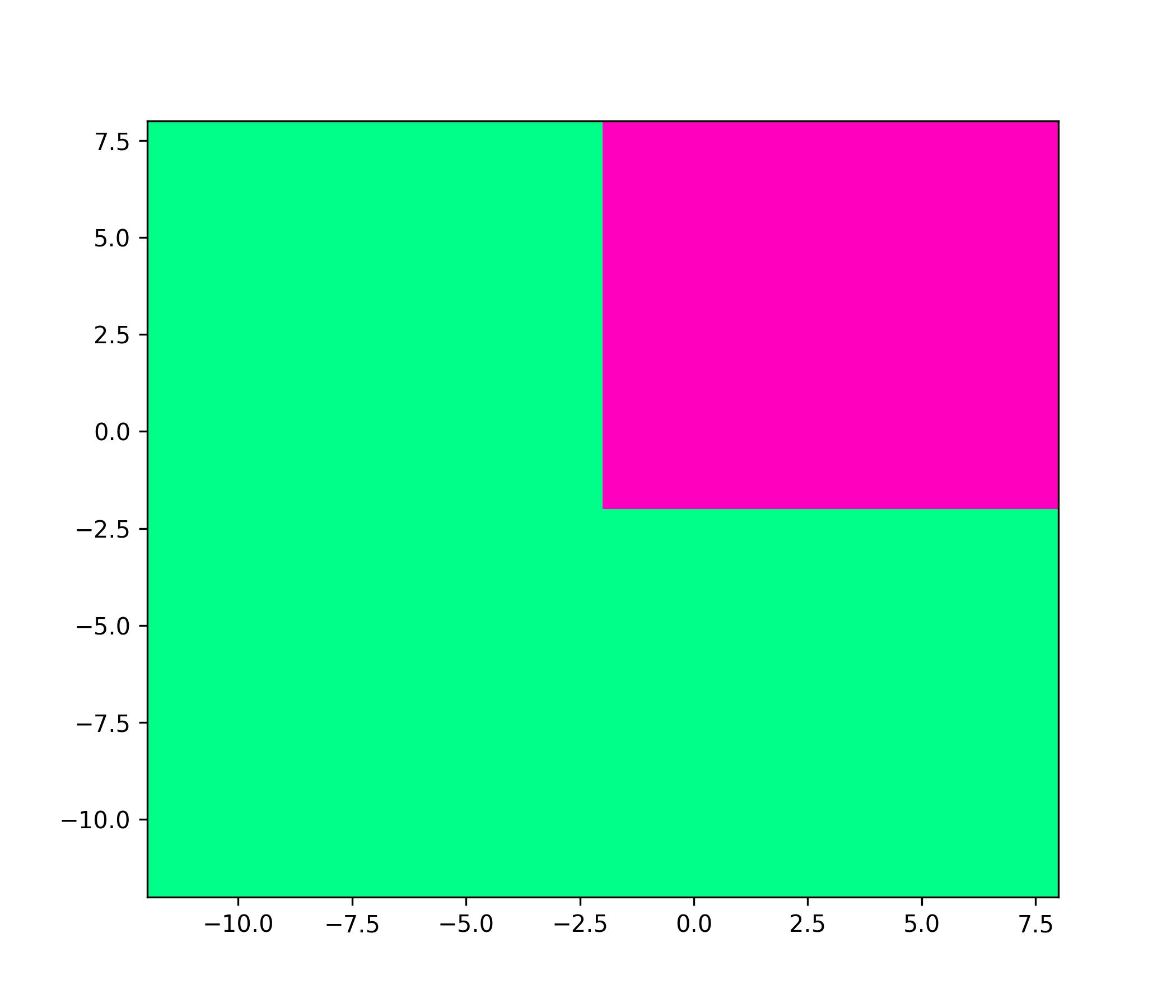}} \\
    \subfigure[]{\includegraphics[width=0.3\textwidth]{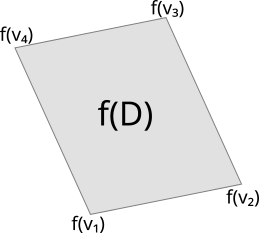}}
    \hspace{2cm}
    \subfigure[]{\includegraphics[width=0.35\textwidth]{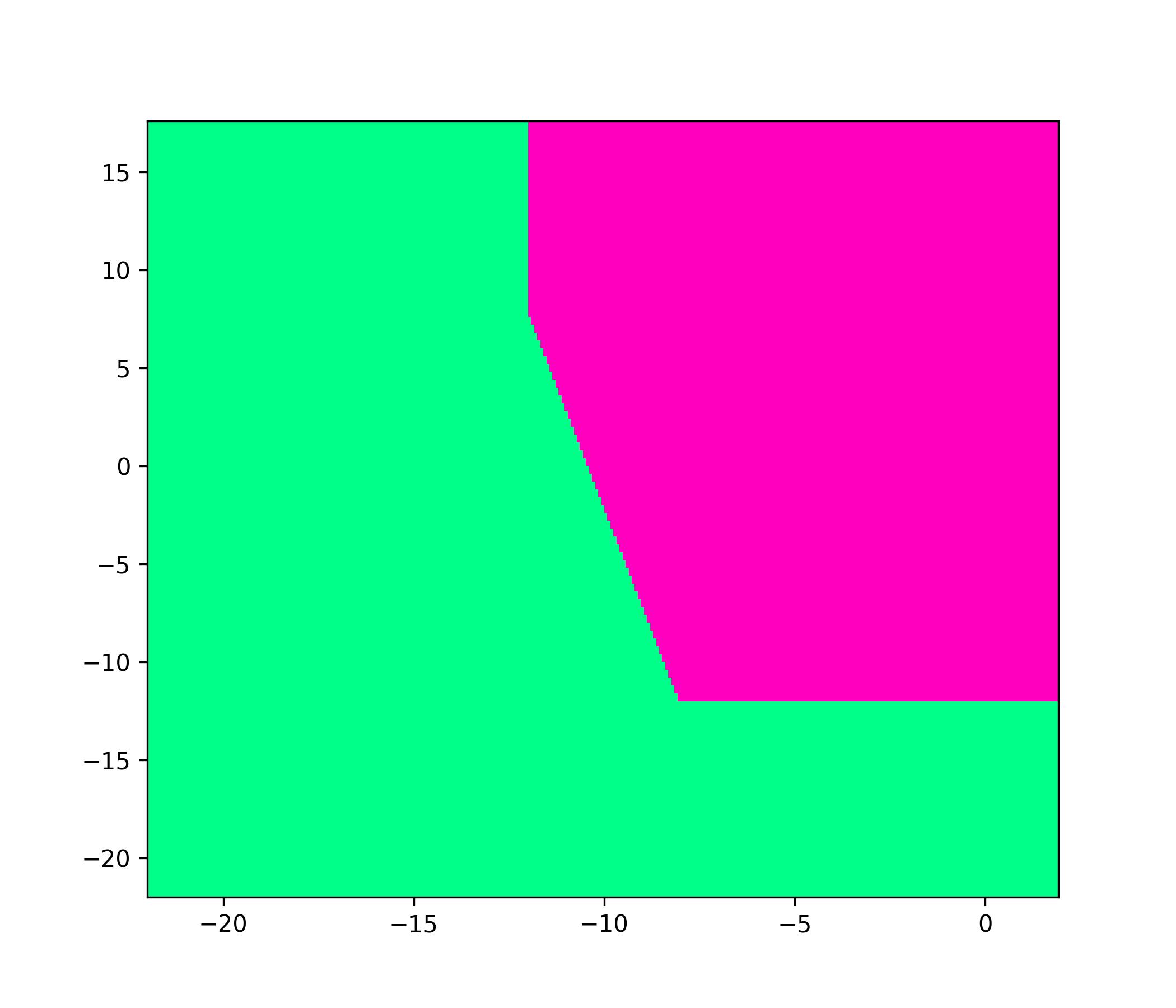}}    \\
    \subfigure[]{\includegraphics[width=0.3\textwidth]{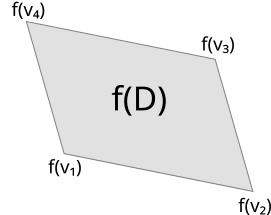}}
    \hspace{2cm}
    \subfigure[]{\includegraphics[width=0.35\textwidth]{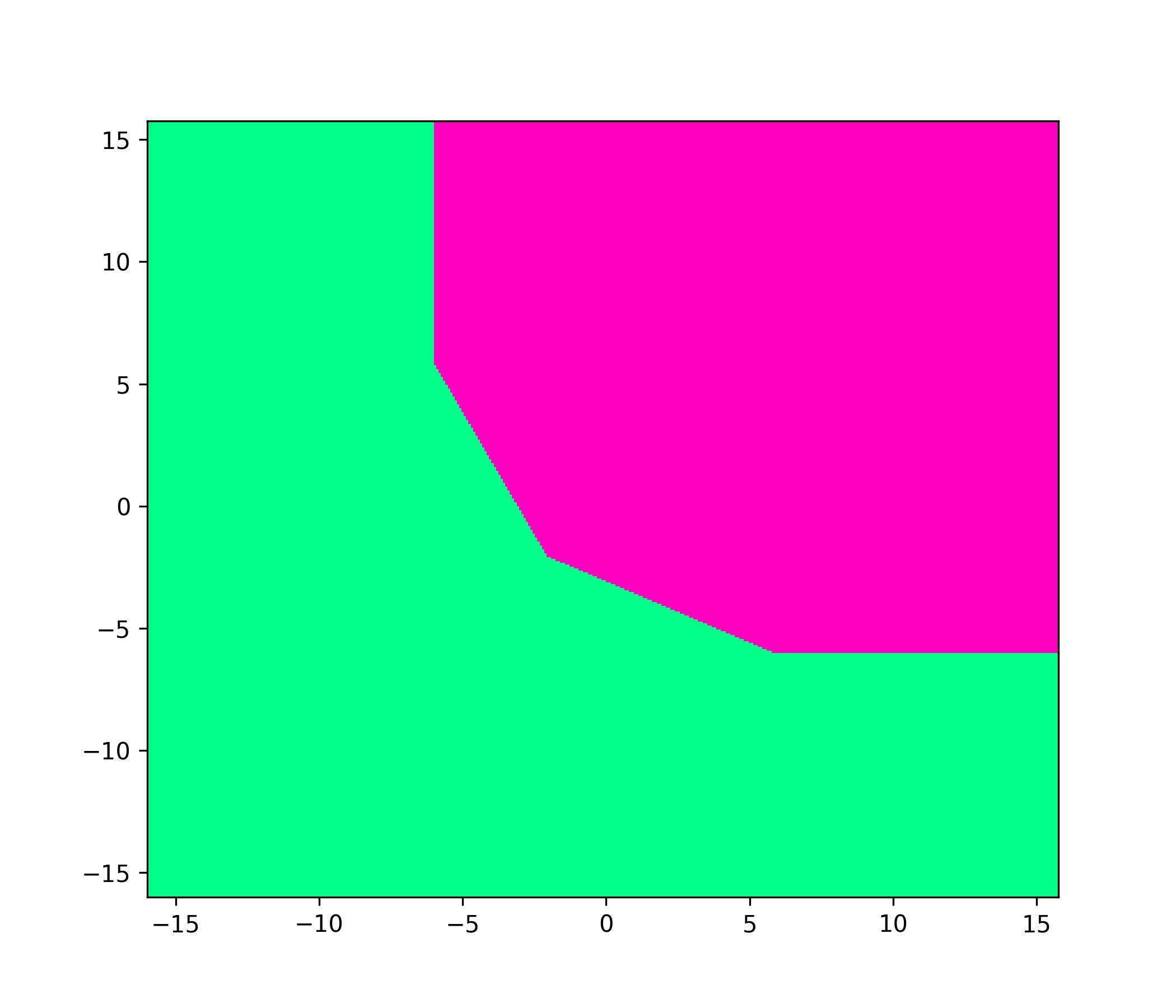}}
    \caption{Linear transformation of an rectangular domain and corresponding cECPs. Green corresponds to an cECP value of 0, while pink corresponds to an cECP value of 1}
    \label{fig:casesM=/=0}
\end{figure}

\subsubsection{Approximation}
The object $\mathrm{cECP}_f$, and in particular its scalar version $\mathrm{cECC}_f$, is an ideal one.
In applications one usually has access only to a finite approximation of the underlying function.
In this section we will focus on the one-parameter case $f:X\to\mathbb{R}$ and show that computable cubical approximations (see more details ~\cite{dlotko2018rigorous}) yield ECCs converging to $\mathrm{cECC}_f$ in the integral metric. 

\medskip

Let $f:X\to\mathbb{R}$ be continuous and let $\mathcal K$ be a cubical subdivision of $X$ (see~\cite{kaczynski2004computational} for details).
Assume that for every maximal cube $K\in\mathcal K$ one can compute an interval enclosure
\[
    [f(K)_{\min},\,f(K)_{\max}]
\]
such that
\[
    f(K)_{\min} \leq f(x) \leq f(K)_{\max}
    \qquad \text{for all } x\in K.
\]
Such enclosures can be computed with interval arithmetic and gives a piecewise constant lower semicontinuous approximation $\hat f$ of $f$. It is obtained by assigning to the cube $K$ the value

For each maximal cube (K), the value
\[
\frac{f(K)_{\min}+f(K)_{\max}}{2}
\]
is assigned to its interior. On intersections of neighboring cubes, the approximation is assigned the minimum of the values associated with all incident maximal cubes, following the construction described in~\cite{dlotko2018rigorous}.

In particular, if we ensure that for every $K \in \mathcal{K}$, $|f(K)_{\max} - f(K)_{\min}| \leq 2\varepsilon$ then the obtained approximation satisfy
\[
    \|f-\hat f\|_\infty \le \varepsilon.
\]

The above construction provides a way to approximate sufficiently regular continuous functions by finite data. Such a $\hat f$ will be called a finite lower semicontinuous piecewise constant cubical approximation of $f$.
We now turn to the approximation of the topology of their sublevel sets. For that we need additional mildness assumptions.

\begin{definition}[Tame]
A continuous function $f:X\to\mathbb{R}$ is called \emph{tame} if it has only finitely many homological critical values and if all homology groups of all sublevel sets
\[
D_t^f := f^{-1}((-\infty,t])
\]
have finite rank.
\end{definition}

For tame functions, persistence diagrams are well-defined and finite in each relevant dimension.
A particularly natural sufficient condition is that $f$ is Morse and has only finitely many critical points.
In that situation the barcode consists of finitely many intervals.
For each homological dimension $k$, we denote by $M_k$ an upper bound on the number of off-diagonal points in the persistence diagram in dimension $k$.

\begin{proposition}(see~\cite{dlotko2018rigorous})
\label{prop:cubical_approximation}
Let $f:X\to\mathbb{R}$ be tame, and let $\hat f_\varepsilon$ be a lower semicontinuous piecewise constant cubical approximation satisfying
\[
\|f-\hat f_\varepsilon\|_\infty \le \varepsilon.
\]
Then for every homological dimension $k$,
\[
d_B\!\left(\mathrm{Dgm}_k(f),\mathrm{Dgm}_k(\hat f_\varepsilon)\right)\le \varepsilon.
\]
\end{proposition}

Proposition~\ref{prop:cubical_approximation} provides control in bottleneck distance of persistence diagrams of tame functions approximated with interval arithmetic. However, the stability of ECCs of filtered complexes $C$ and $D$ is expressed in the \(1\)-Wasserstein metric, namely
\[
    \|ECC(C)-ECC(D)\|_{L^1}
    \le 2\sum_k W_1\!\left(\mathrm{Dgm}_k(C),\mathrm{Dgm}_k(D)\right),
\]
see~\cite{dlotko2023ecc}. Therefore in order to use the stability result for ECCs, we need to bound 1-Wasserstein distance with a bottleneck distance of persistence diagrams.

\begin{lemma}\label{lem:W1_from_bottleneck}
Let $C$ and $D$ be persistence diagrams, each having at most $M$ off-diagonal points.
Then
\[
    W_1(C,D) \le 2M\, d_B(C,D).
\]
\end{lemma}

\begin{proof}
Fix an optimal bottleneck matching between $C$ and $D$.
Under this matching, every off-diagonal point (i.e. a point with non-zero persistent homology)  of either diagram is moved by at most $d_B(C,D)$ in the $\ell^\infty$ norm. Since each diagram has at most $M$ off-diagonal points, there are at most $2M$ nontrivial matched points contributing to the 1-Wasserstein cost.
Therefore,
\[
    W_1(C,D)
    \le 2M\,d_B(C,D). 
\] \qed
\end{proof}
We now combine the bottleneck control coming from cubical approximation with the stability of ECCs with respect to the 1-Wasserstein distance.

\begin{theorem}\label{thm:cECC_convergence}
Let $X\subset \mathbb{R}^n$ be compact and let $f:X\to\mathbb{R}$ be a continuous tame function.
Assume that for every $\varepsilon>0$ there exists a finite lower semicontinuous piecewise constant cubical approximation $\hat f_\varepsilon$ such that
\[
    \|f-\hat f_\varepsilon\|_\infty \le \varepsilon.
\]
Assume moreover that only finitely many homological dimensions contribute to persistence, and that for each such dimension $k$ the diagrams
\[
    \mathrm{Dgm}_k(f)
    \qquad \text{and} \qquad
    \mathrm{Dgm}_k(\hat f_\varepsilon)
\]
have at most $M_k$ off-diagonal points for every $\varepsilon > 0$.

Then the ECCs satisfy
\[
    \|\mathrm{cECC}_f - ECC_{\hat f_\varepsilon}\|_{L^1}
    \le
    4\varepsilon \sum_k M_k,
\]
where
\[
    ECC_{\hat f_\varepsilon}(t)
    :=
    \sum_k (-1)^k \beta_k\!\left((\hat f_\varepsilon)^{-1}((-\infty,t])\right).
\]
In particular,
\[
    \|\mathrm{cECC}_f - ECC_{\hat f_\varepsilon}\|_{L^1} \longrightarrow 0
    \qquad \text{as } \varepsilon \to 0.
\]
\end{theorem}

\begin{proof}
By the persistence stability result,
\[
    d_B\!\left(\mathrm{Dgm}_k(f),\mathrm{Dgm}_k(\hat f_\varepsilon)\right)\le \varepsilon
    \qquad \text{for every } k.
\]
Applying Lemma~\ref{lem:W1_from_bottleneck} in each dimension gives
\[
    W_1\!\left(\mathrm{Dgm}_k(f),\mathrm{Dgm}_k(\hat f_\varepsilon)\right)
    \le 2M_k\,\varepsilon.
\]
The stability estimate for ECCs yields
\[
    \|\mathrm{cECC}_f - ECC_{\hat f_\varepsilon}\|_{L^1}
    \le
    2\sum_k
    W_1\!\left(\mathrm{Dgm}_k(f),\mathrm{Dgm}_k(\hat f_\varepsilon)\right).
\]
Combining the last two inequalities, we obtain
\[
    \|\mathrm{cECC}_f - ECC_{\hat f_\varepsilon}\|_{L^1}
    \le
    2\sum_k 2M_k\,\varepsilon
    =
    4\varepsilon \sum_k M_k.
\]
Since the right-hand side tends to zero as $\varepsilon\to 0$, the convergence in $L^1$ follows. \qed
\end{proof}

The above theorem shows that for sufficiently mild functions the ECC obtained from a verified cubical approximation converges to the continuous ECC of the original function as the approximation is refined.
Thus, by decreasing $\varepsilon$, one can make the computed curve arbitrarily close to the ideal object $\mathrm{cECC}_f$ in the integral metric.

\begin{remark}
The role of the finiteness assumption on the persistence diagrams is essential.
Bottleneck distance controls only the largest displacement of matched points, whereas the 1-Wasserstein distance accumulates all such displacements.
Therefore, passing from bottleneck stability to $L^1$-stability of ECCs requires a uniform bound on the number of off-diagonal points of both function $f$ and its approximation $\hat f_{\epsilon}$.
A Morse function with finitely many critical points provides a natural sufficient setting for this assumption.
\end{remark}

\begin{remark}
The previous argument is specific to the one-parameter case.
For multiparameter filtrations, one may define the continuous ECP coordinatewise, but an analogous convergence theorem requires a separate stability theory and is not addressed here.
\end{remark}

\subsection{ECP for sampled vector fields}

The construction of a multifiltration from a sampled vector field
depends on the geometry of the sampling set. For vector fields sampled
on regular Cartesian grids, we use cubical complexes: filtration values
are assigned to the top-dimensional cubes and extended to their faces
by taking coordinatewise minima. For irregularly sampled data, one may
instead construct a simplicial complex, such as an alpha complex, assign
the sampled vector-field values to its vertices, and extend them to
higher-dimensional simplices by taking coordinatewise maxima. In the
present work, all numerical experiments use regular grids and cubical
multifiltrations; the simplicial construction is included below only for
completeness.

\begin{definition}[Multifiltration induced by a sampled vector field]
\label{def:vector_field_multifiltration}
Let $D \subset \mathbb{R}^{n}$ be a bounded domain, let
$\mathcal K_D$ be the cubical complex as in
Definition~\ref{def:cubical_complex}, and let
\[
F=(f_1,\ldots,f_n)\colon D \to \mathbb R^{n}
\]
be a vector field sampled on the grid $\mathcal G_D$.

For every $n$-dimensional cell $Q\in\mathcal K_D$, let $c(Q)$ denote
the barycenter, of $Q$. We define the
multifiltration value of $Q$ by
\[
\varphi(Q)
=
F(c(Q))
=
\bigl(f_1(c(Q)),\ldots,f_n(c(Q))\bigr)
\in \mathbb R^{n}.
\]

The function $\varphi$ is extended to all cells of $\mathcal K_D$ as
follows. For every cell $\sigma\in\mathcal K_D$, let
\[
\mathcal C(\sigma)
=
\{Q\in\mathcal K_D \mid \sigma \subseteq Q,\ \dim Q=n\}
\]
denote the collection of incident $n$-dimensional cells. We then set
\[
\varphi(\sigma)
=
\min_{Q\in\mathcal C(\sigma)}
\varphi(Q),
\]
where the minimum is taken coordinatewise, that is,
\[
\varphi(\sigma)
=
\left(
\min_{Q\in\mathcal C(\sigma)} \varphi_1(Q),
\ldots,
\min_{Q\in\mathcal C(\sigma)} \varphi_n(Q)
\right).
\]

The resulting map
\[
\varphi \colon \mathcal K_D \to \mathbb R^{n}
\]
is called the \emph{multifiltration function induced by the vector
field} $F$. It satisfies
\[
\tau \subseteq \sigma
\quad\Longrightarrow\quad
\varphi(\tau)\leq \varphi(\sigma),
\]
where $\leq$ denotes the product order on $\mathbb R^n$. Consequently,
for every $a\in\mathbb R^n$, the collection
\[
K_a
=
\{\sigma\in\mathcal K_D \mid \varphi(\sigma)\leq a\}
\]
forms a subcomplex of $\mathcal K_D$, and the family
$\{K_a\}_{a\in\mathbb R^n}$ defines a multifiltration. In the particular
case $n=2$, the resulting multifiltration is called a
\emph{bifiltration}.
\end{definition}

\begin{remark}[Irregularly sampled vector fields]
Let $S\subset D$ be a finite, irregular sample and suppose that the
vector-field values
\[
F(v)\in\mathbb R^n,
\qquad v\in S,
\]
are known at the sample points. One may construct a simplicial complex
$\mathcal A_\alpha(S)$ on $S$, for example an alpha complex~\cite{edelsbrunner2010} at a fixed
scale $\alpha$. The filtration values are first assigned to the vertices
by
\[
\varphi(v)=F(v),
\qquad v\in S,
\]
and then extended to every simplex
$\sigma\in\mathcal A_\alpha(S)$ by
\[
\varphi(\sigma)
=
\max_{v\in\sigma}\varphi(v),
\]
where the maximum is taken coordinatewise. Thus,
\[
\varphi(\sigma)
=
\left(
\max_{v\in\sigma}\varphi_1(v),
\ldots,
\max_{v\in\sigma}\varphi_n(v)
\right).
\]
If $\tau\subseteq\sigma$, then
\[
\varphi(\tau)\leq\varphi(\sigma),
\]
and therefore the coordinatewise sublevel sets of $\varphi$ form
simplicial subcomplexes and define a multifiltration. This simplicial
construction is not used in the numerical experiments presented in this
work.
\end{remark}

\section{Numerical Examples: Low-Dimensional Vector Fields}
\label{sec:ECP_linear}
We begin by considering a number of simple low-dimensional dynamical systems whose behavior is well understood analytically and geometrically. These examples provide a convenient test bench for the proposed method, making it possible to assess in a controlled setting how ECPs fundamental dynamical patterns before turning to more complicated systems. In particular, we examine examples displaying characteristic structures such as closed trajectories and isolated stationary points, and study the effect of different filtration choices on the resulting ECPs. 

\begin{remark}
To carry out the numerical experiments considered in this section, the code provided in~\cite{dlotko2023ecc} proved insufficient for our purposes. We therefore developed a more efficient algorithm for computing ECPs. This algorithm forms part of the methodological contribution of the present paper, and its details are presented in Appendix~A.
\end{remark}

To facilitate reproducibility and further development of the proposed methodology, the basic implementation for computing ECPs for vector fields has been made publicly available in an open-source repository. The repository contains the core algorithms together with illustrative examples for representative classes of dynamical systems. It is actively maintained and continuously extended with additional examples, numerical experiments, and new computational functionalities developed throughout the course of this research. 
The source code is publicly available in the GitHub repository~\cite{github}, which is continuously updated with new examples and implementations.

\subsection{Signatures of Closed Orbits}
We begin by investigating the ECP signatures associated with periodic behavior in planar vector fields. Since the proposed constructions are defined in the continuous setting but require discrete representations for numerical evaluation, we approximate the vector field on a uniform spatial grid defined below. All simulations are performed on a uniform grid over the domain $[-2,2]^2$ with resolution $200 \times 200$, and the curves are sampled using $600$ equidistant parameter values. The tangent vectors are normalized prior to further processing.

\subsubsection{ECPs of closed orbits}
The Euler Characteristic Profile satisfies the inclusion--exclusion
formula pointwise in the filtration parameter. Let $A$ and $B$ be
subcomplexes equipped with the restrictions of the same filtration
function. Then, for every parameter value $u$,
\[
\operatorname{ECP}_{A\cup B}(u)
=
\operatorname{ECP}_{A}(u)
+
\operatorname{ECP}_{B}(u)
-
\operatorname{ECP}_{A\cap B}(u).
\]
The formula follows from the inclusion--exclusion property of the
Euler characteristic. In particular, the ECP is additive when the
filtered parts $A_u$ and $B_u$ are disjoint for every $u$.

In our setting, we decompose the phase space into the closed orbit and the remaining background region. Along the orbit, the vector field is given by the normalized tangent vector to the parametrized curve, so that it reflects the local geometry of the periodic trajectory. Outside the orbit, we prescribe a constant vector field with value $(v_x,v_y)$, which we vary in order to study its influence on the resulting profile. This background field has a particularly simple ECP: it is equal to $1$ at all points above $(v_x,v_y)$ and $0$ elsewhere. Again, the ECP of the full system is obtained by combining this contribution with the one generated by the orbit itself using inclusion-exclusion property.

Our goal in this section is to identify characteristic ECP fingerprints of closed orbits, starting from the simplest case of a circular trajectory and then passing to geometrically more complex closed curves. To this end, we consider two parametrized closed curves: a circle and a blob-like deformation of a circle, both given by
\begin{equation}
    \begin{aligned}
        x(t) &= r(t)\cos(t),\\
        y(t) &= r(t)\sin(t),
    \end{aligned}
    \label{eq:curves}
\end{equation}
where for the circle we take \(r(t)=r\), while for the blob-like curve we set
\[
r(t)=r_0\Bigl(1+\sum_i a_i\cos(f_i t)\Bigr).
\]

The associated vector field is constructed from the tangent vectors along the curve, while the background grid is filled with zero vectors, that is $v_x = v_y = 0$. Tangent vectors were normalized before constructing the filtration, allowing experiments to focus on the geometric organization of the trajectories rather than their local magnitude. Fig. ~\ref{fig:circle_blob} shows the corresponding phase portraits and ECPs for a circle of radius \(r=1\) and for a blob-like curve with parameters \(r_0=1\), \(a_1=0.15\), \(a_2=0.1\), \(f_1=3\) and \(f_2=5\).

\begin{figure}[h]
    \centering
    \subfigure[]{\includegraphics[width=0.45\textwidth]{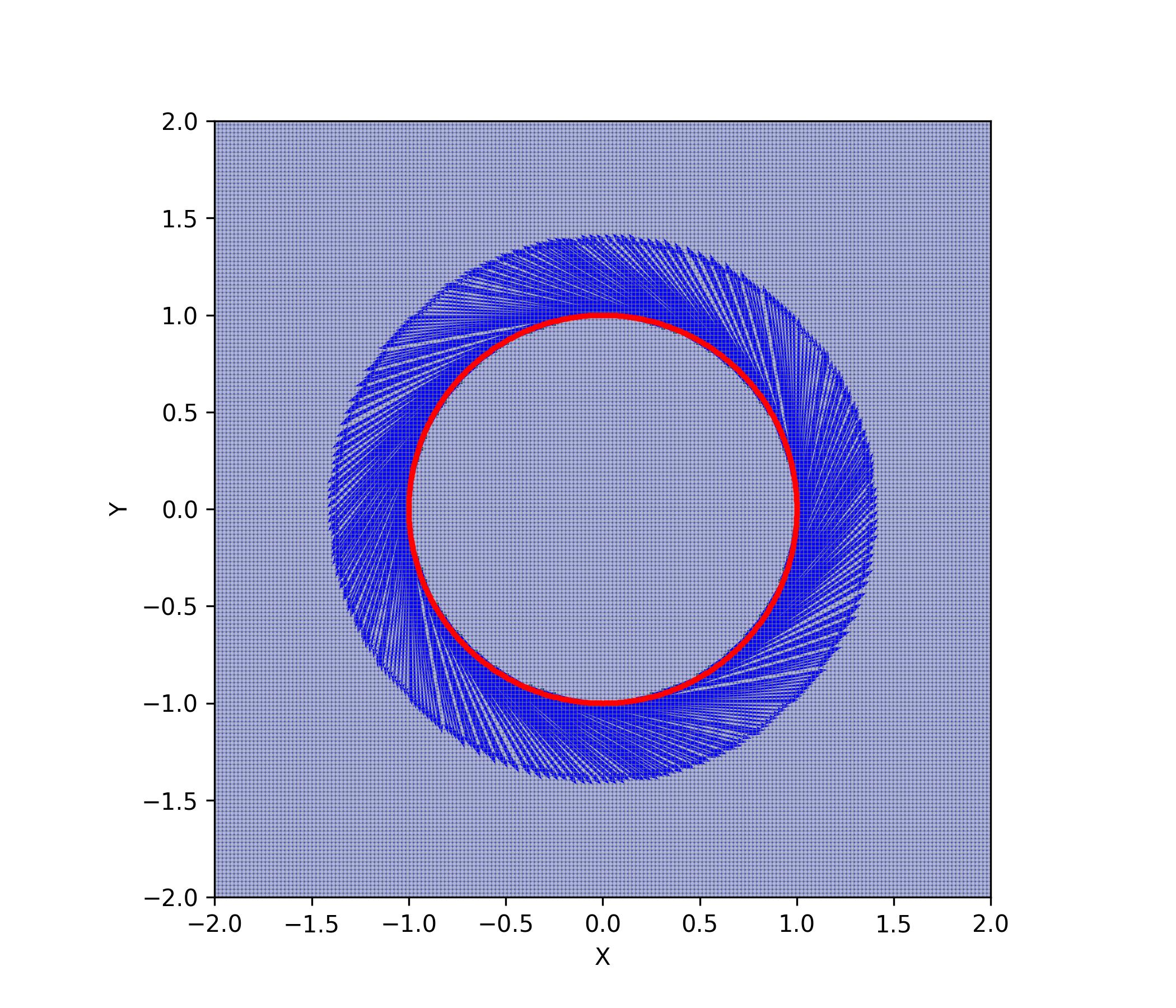}}
    \subfigure[]{\includegraphics[width=0.45\textwidth]{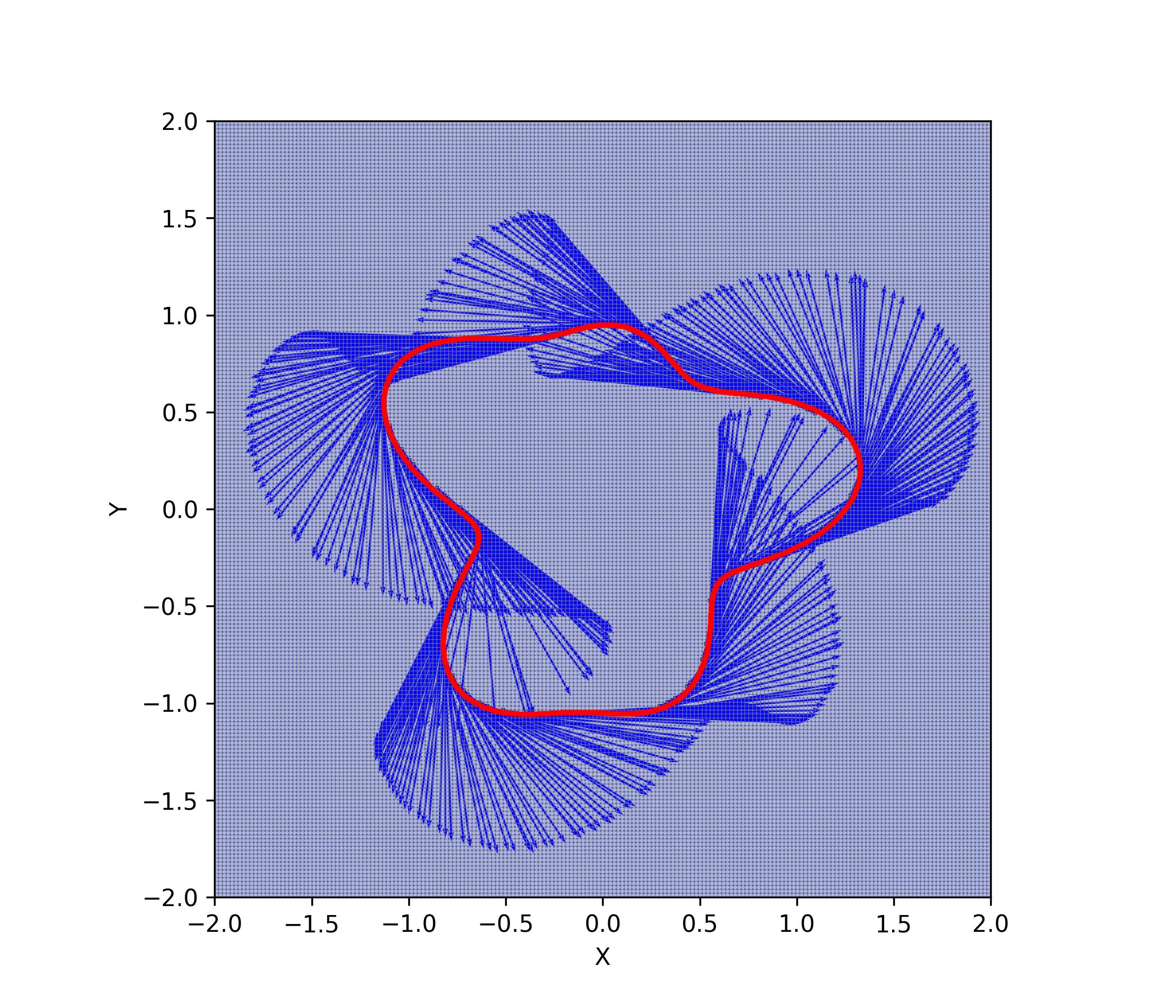}} \\

    \subfigure[]{\includegraphics[width=0.45\textwidth]{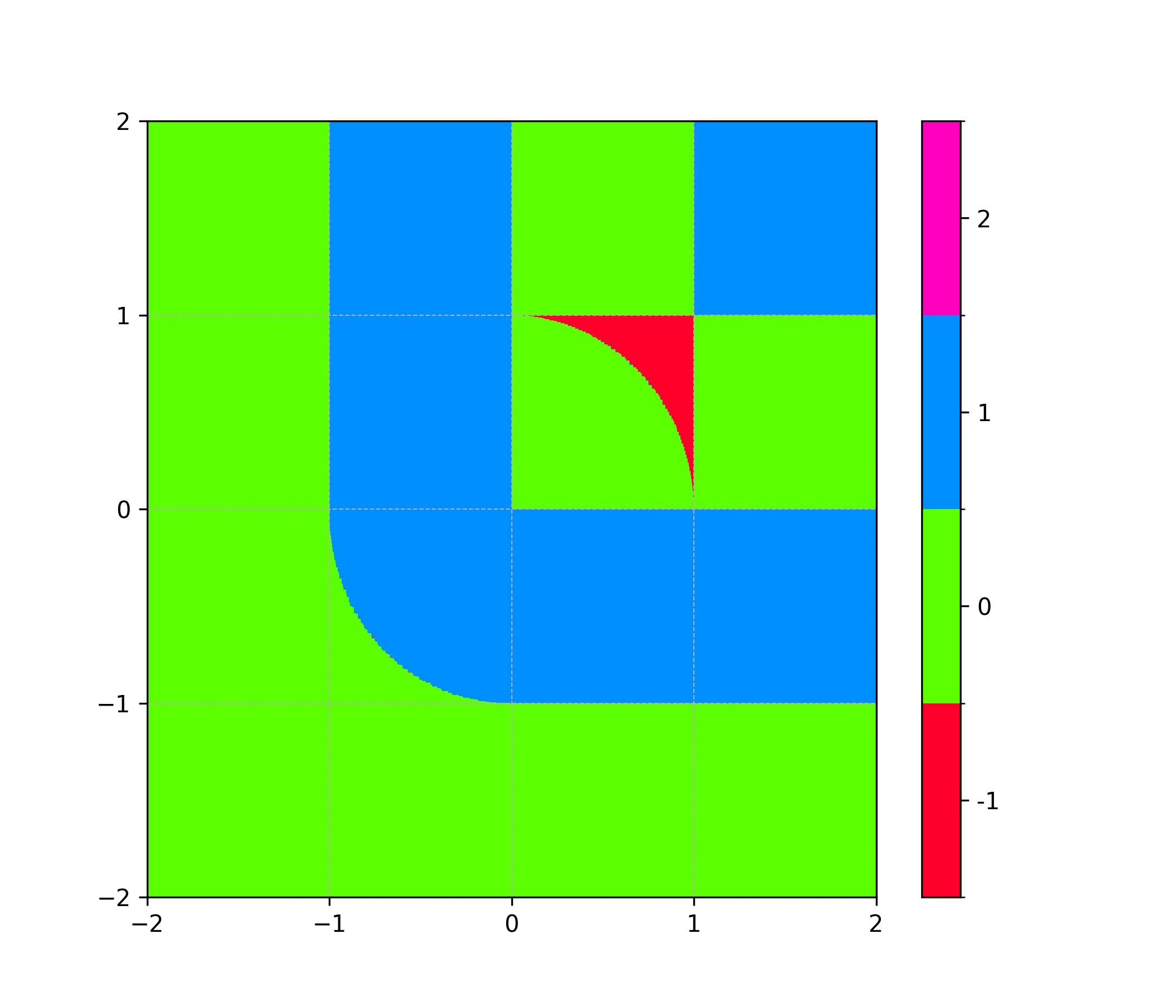}}
    \subfigure[]{\includegraphics[width=0.45\textwidth]{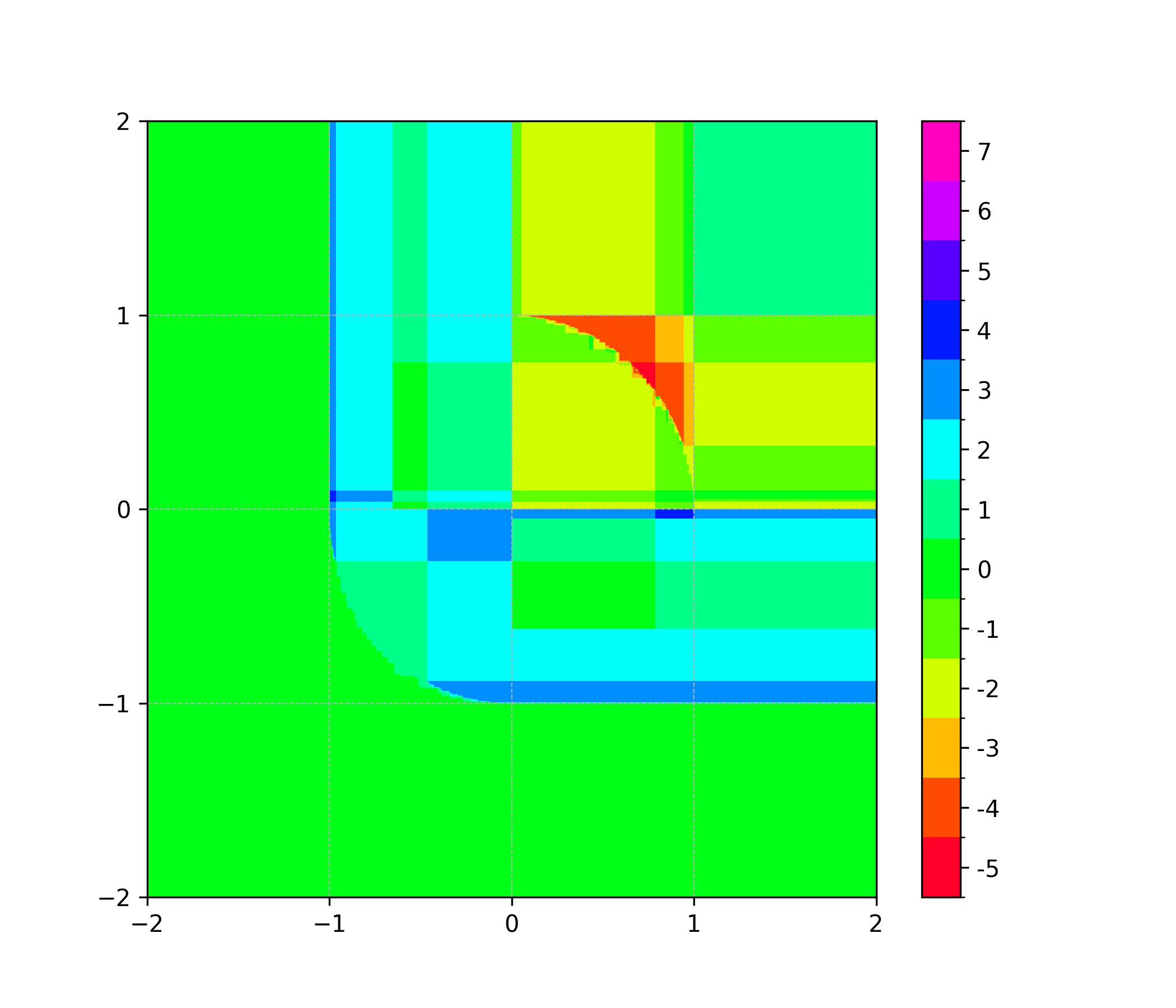}}
    \caption{Representative periodic trajectories used in the preliminary validation of the proposed methodology: (a)-(b) phase portraits of a circle and a blob-shaped closed curve with normalized tangent vector fields, and (c)-(d) the corresponding Euler Characteristic Profiles (ECPs) computed for zero background vectors}
    \label{fig:circle_blob_back_0}
\end{figure}

In this setting, the contribution of an isolated closed periodic orbit of approximately circular shape has a characteristic form in the ECP: it appears as an $L$-shaped corner determined by the minimal and maximal values attained by the vector field along the orbit. 
The values of the ECP within this $L$-shaped region depend strongly on the monotonicity of changes in directions of the orbit. Below this region, the ECP is equal to $0$, whereas above it, the ECP is equal to $1$.

We can note that in the case of an ellipse (for the normalized tangent direction distribution under a monotone reparametrization of direction), we obtain exactly the same ECP as for a circle (Fig. ~\ref{fig:circle_blob_back_0}c), whereas more deformed curves, such as the blob, produce more complex profiles (Fig. ~\ref{fig:circle_blob_back_0}d) which is caused by more rapid changes in direction vectors (and thus more frequent changes of monotonicity) and, consequently, bifiltration. This difference is expected, as more intricate curves generate a greater variety of tangent vectors, although the overall shape of the ECP remains qualitatively similar.  

Next, we introduce a uniform non-zero background vector. Figure~\ref{fig:circle_blob} presents the ECPs for the circle and blob under two background vectors: (a)-(b) $[v_x,v_y] = [10, 10]$ and (c)-(d) $[v_x,v_y] = [-1.5, -1]$, which means that background vectors are taken into account at the very end (Fig.~\ref{fig:circle_blob}a and Fig.~\ref{fig:circle_blob}b) and at the very beginning (Fig.~\ref{fig:circle_blob}c and Fig.~\ref{fig:circle_blob}d) 

\begin{figure}[h]
    \centering
    \subfigure[]{\includegraphics[width=0.45\textwidth]{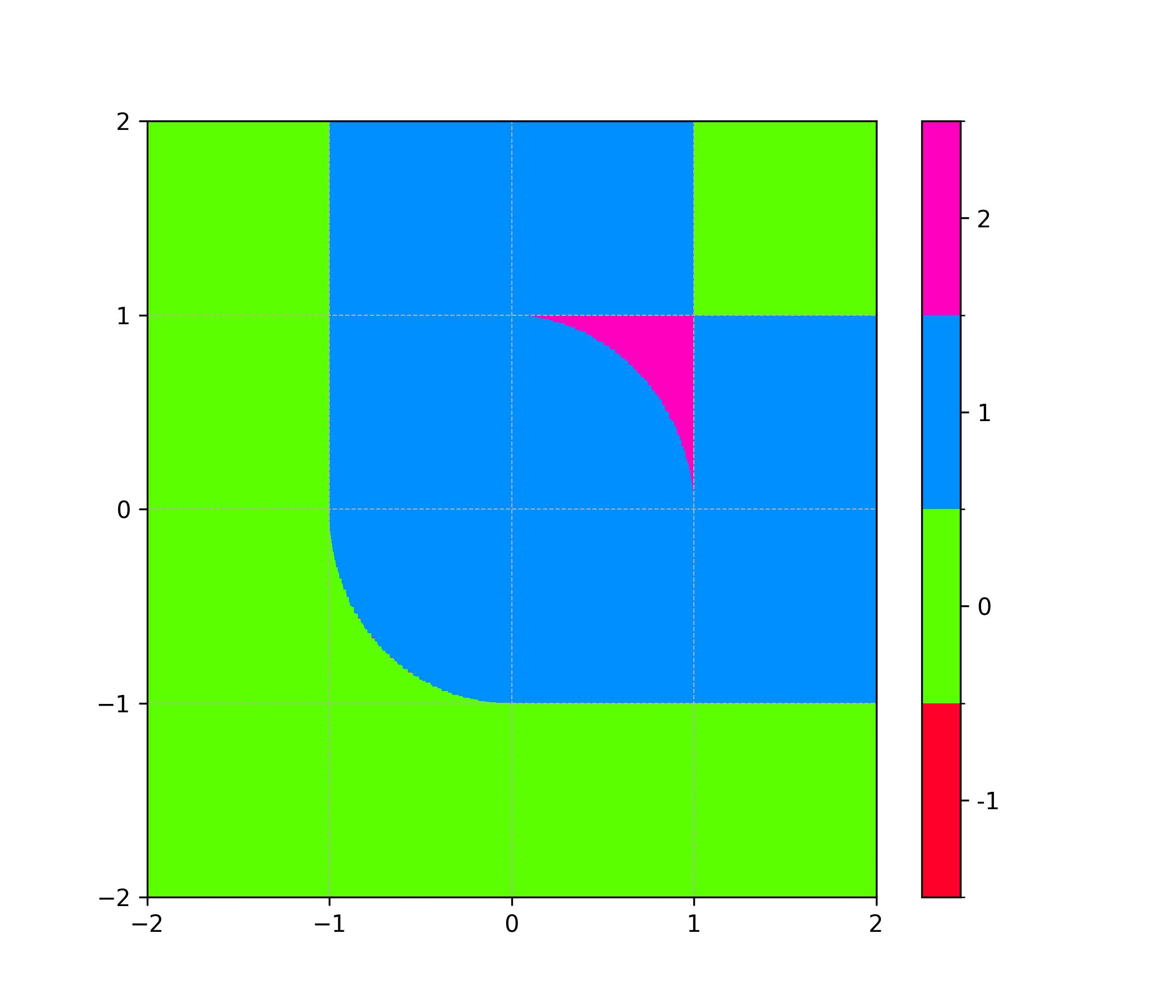}}
    \subfigure[]{\includegraphics[width=0.45\textwidth]{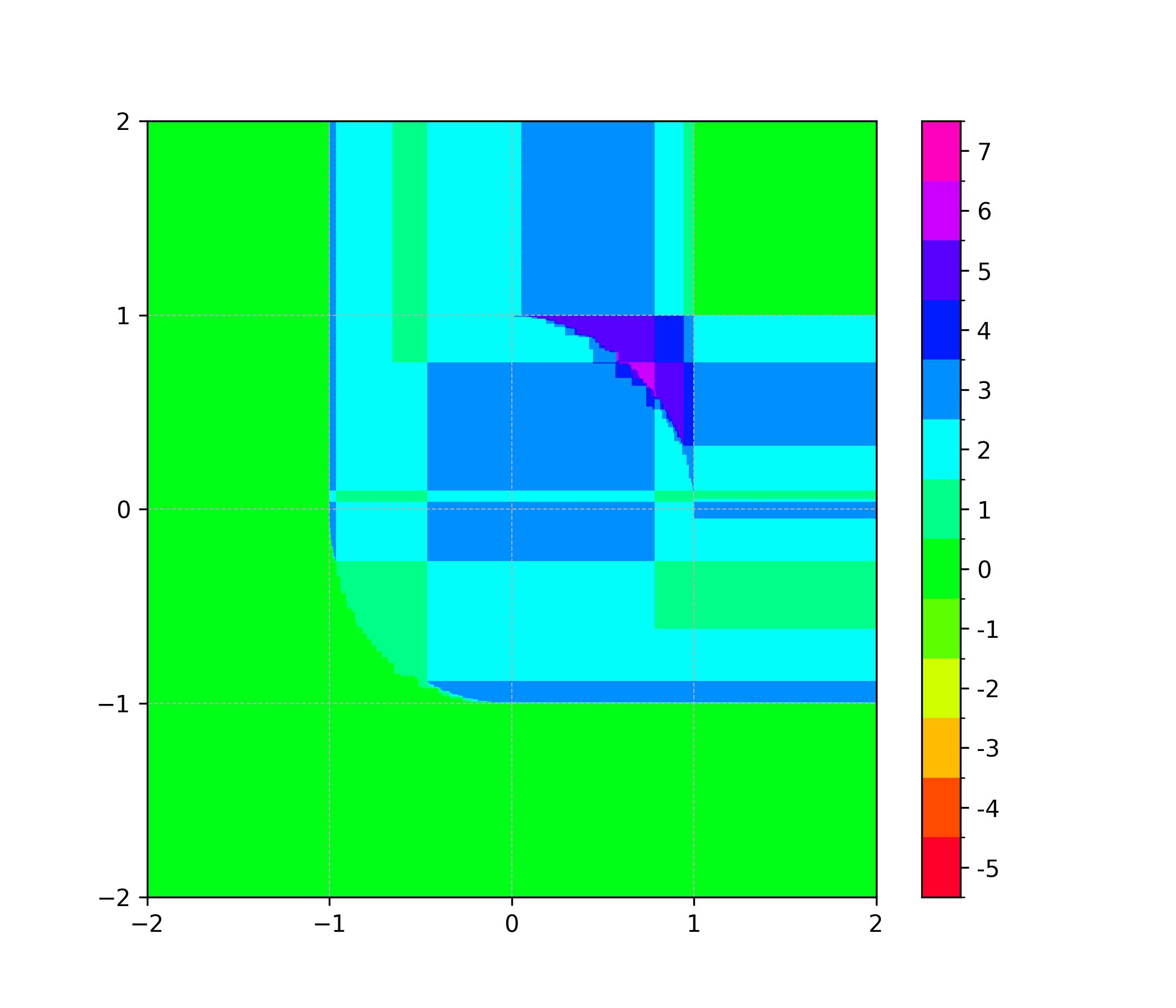}} \\
    \subfigure[]{\includegraphics[width=0.45\textwidth]{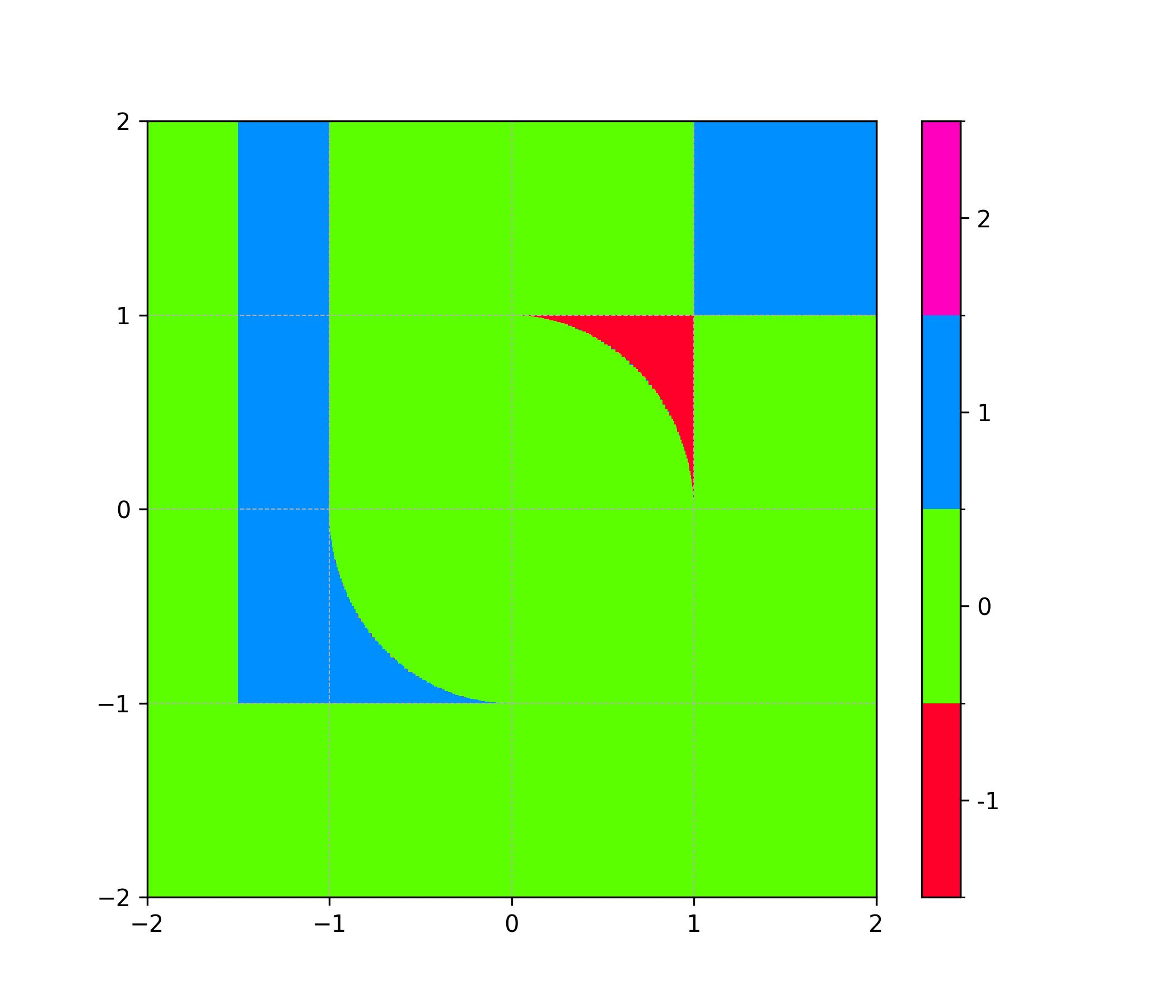}}
    \subfigure[]{\includegraphics[width=0.45\textwidth]{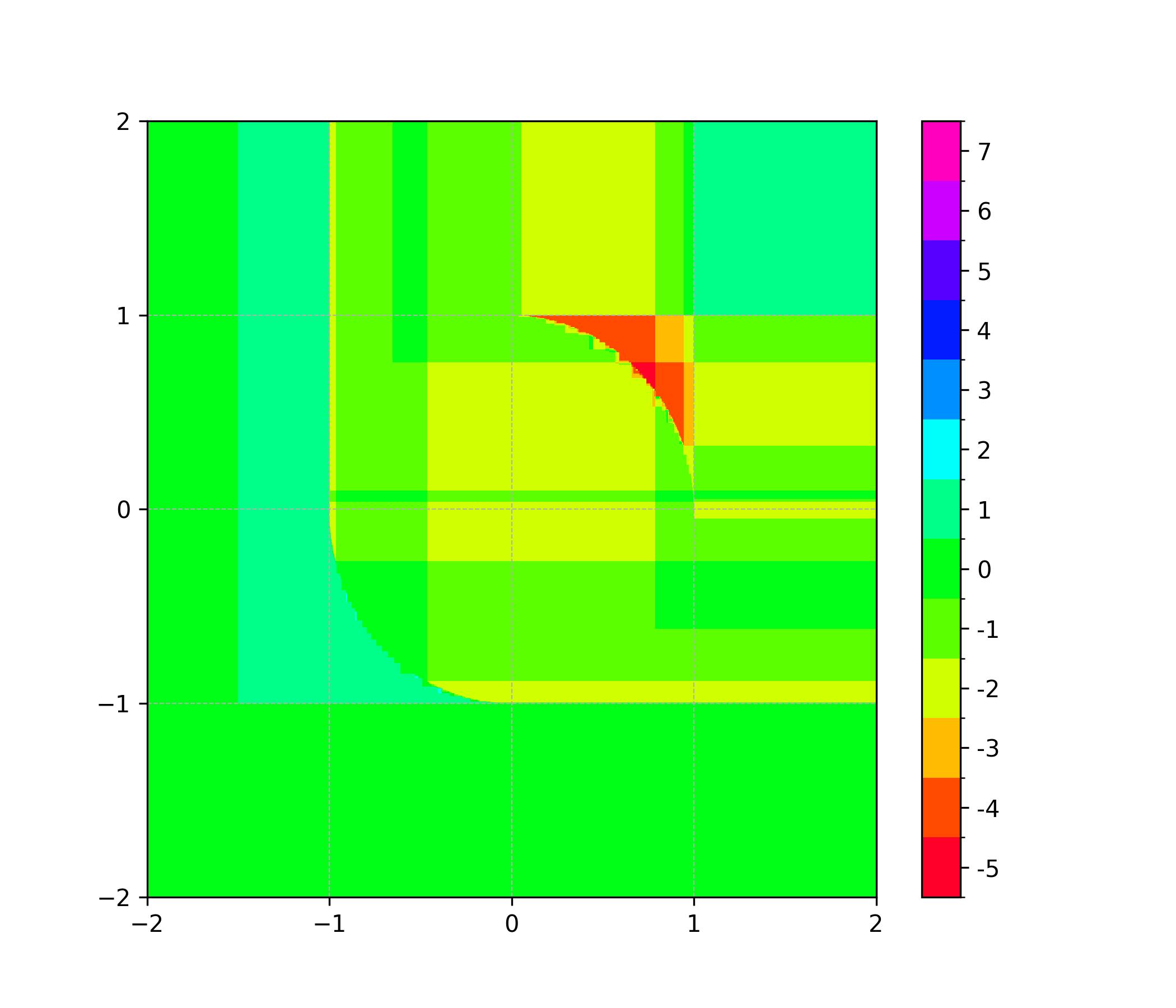}}
    \caption{Euler Characteristic Profiles (ECPs) of the circle and blob-shaped trajectories under constant background vector fields: (a)-(b) background vector $(10,10)$ and (c)-(d) background vector $(-1.5,-1)$. The results illustrate the influence of uniform background flows on the resulting topological descriptors}
    \label{fig:circle_blob}
\end{figure}

The following lemma summarizes the shape of cECC for the points in the unit circle with the tangent vector field defined on it.

\begin{lemma}\label{lem:circle_only_ecp}
Let \(\Gamma=S^1\subset\mathbb{R}^2\) be the unit circle, and let
\[
T:\Gamma\to\mathbb{R}^2,\qquad T(x,y)=(-y,x),
\]
be the normalized tangent vector field on \(\Gamma\).
For \(u=(u_1,u_2)\in\mathbb{R}^2\), define
\[
A_u:=\{p\in \Gamma : T(p)\le u\},
\]
where the inequality is understood coordinatewise.
Then \(A_u\) is one of the following:
\begin{enumerate}
    \item the empty set,
    \item a single closed arc,
    \item a disjoint union of two closed arcs,
    \item the whole circle \(\Gamma\).
\end{enumerate}
Consequently,
\[
\chi(A_u)\in\{0,1,2\}.
\]
\end{lemma}

\begin{proof}
Write \(p=(\cos t,\sin t)\), \(t\in[0,2\pi]\). Then
\[
T(p)=(-\sin t,\cos t).
\]
Hence
\[
A_u=\{t\in[0,2\pi] : -\sin t\le u_1,\ \cos t\le u_2\}.
\]
Each of the two inequalities defines a closed arc on \(S^1\), and therefore
\(A_u\) is the intersection of two closed arcs on the circle.
Such an intersection can only be empty, a single closed arc, a disjoint union of two closed arcs, or the whole circle.
The Euler characteristic of these four possibilities is, respectively, $0, 1, 2, 0$.
This proves the claim.
\qed
\end{proof}

Let us now consider three configurations of circles: concentric, disjoint and intersecting. Their corresponding phase portraits and ECPs are presented in Figure~\ref{fig:circle_vf_ECP}. As can be observed, the concentric and disjoint circles exhibit identical characteristics. This outcome is expected, since, from a topological perspective, a displacement of a geometric object does not alter its Euler characteristics unless it induces additional topological features, such as new connected components or holes. In contrast, intersecting circles introduce topological changes, which are clearly reflected in the corresponding ECP, demonstrating a significant impact on the system's characteristics.

\begin{figure}[h]
    \centering
    \subfigure[]{\includegraphics[width=0.325\textwidth]{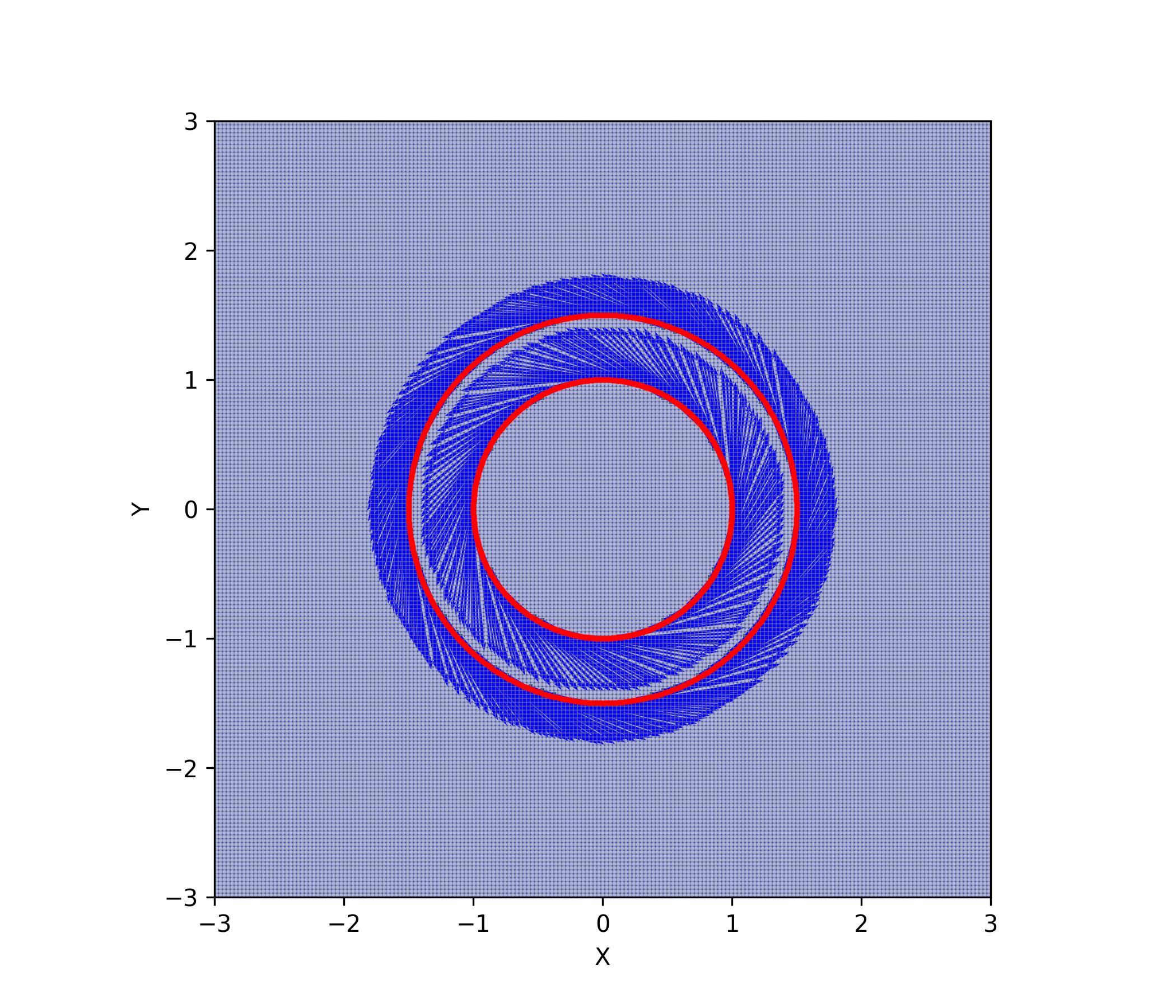}}
    \subfigure[]{\includegraphics[width=0.325\textwidth]{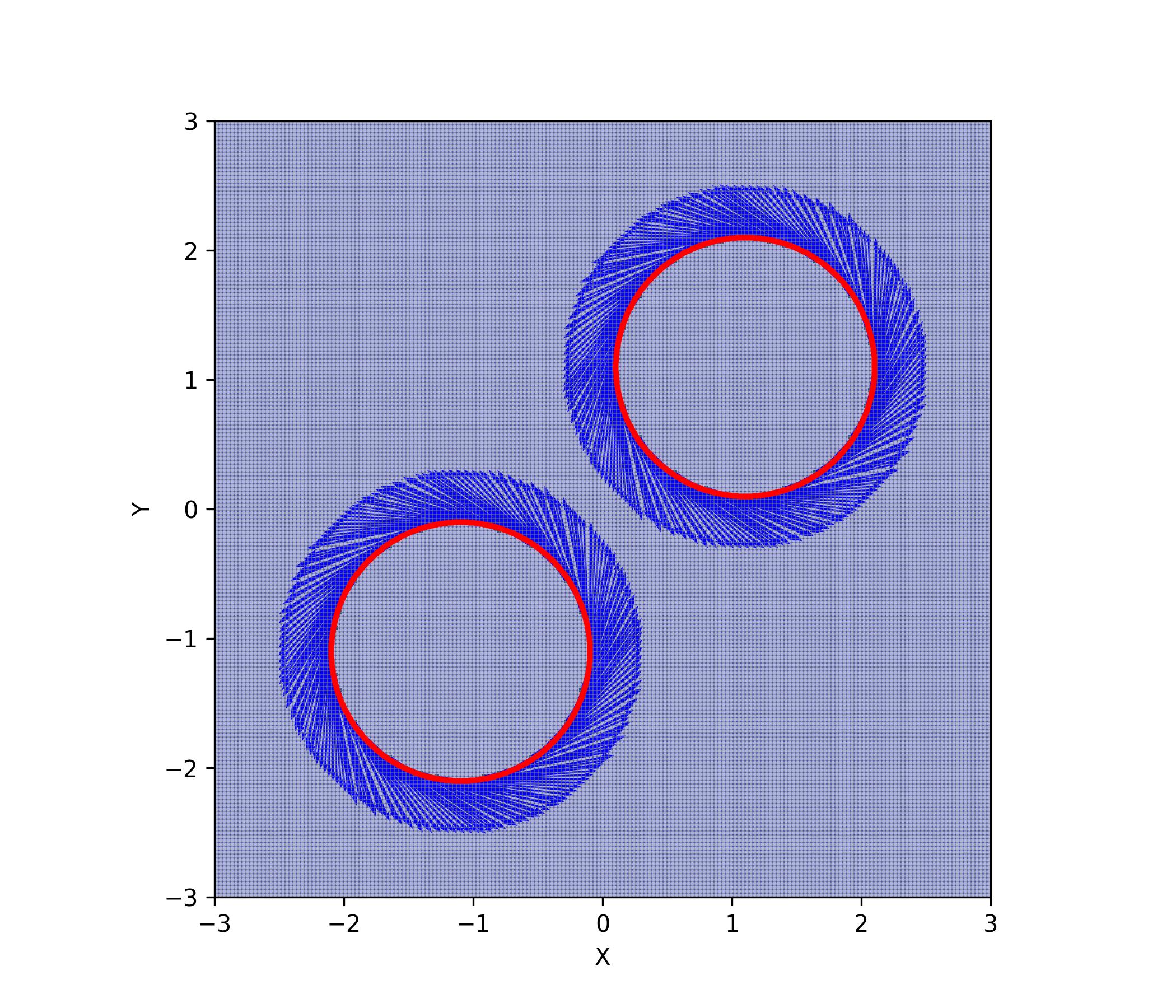}}
    \subfigure[]{\includegraphics[width=0.325\textwidth]{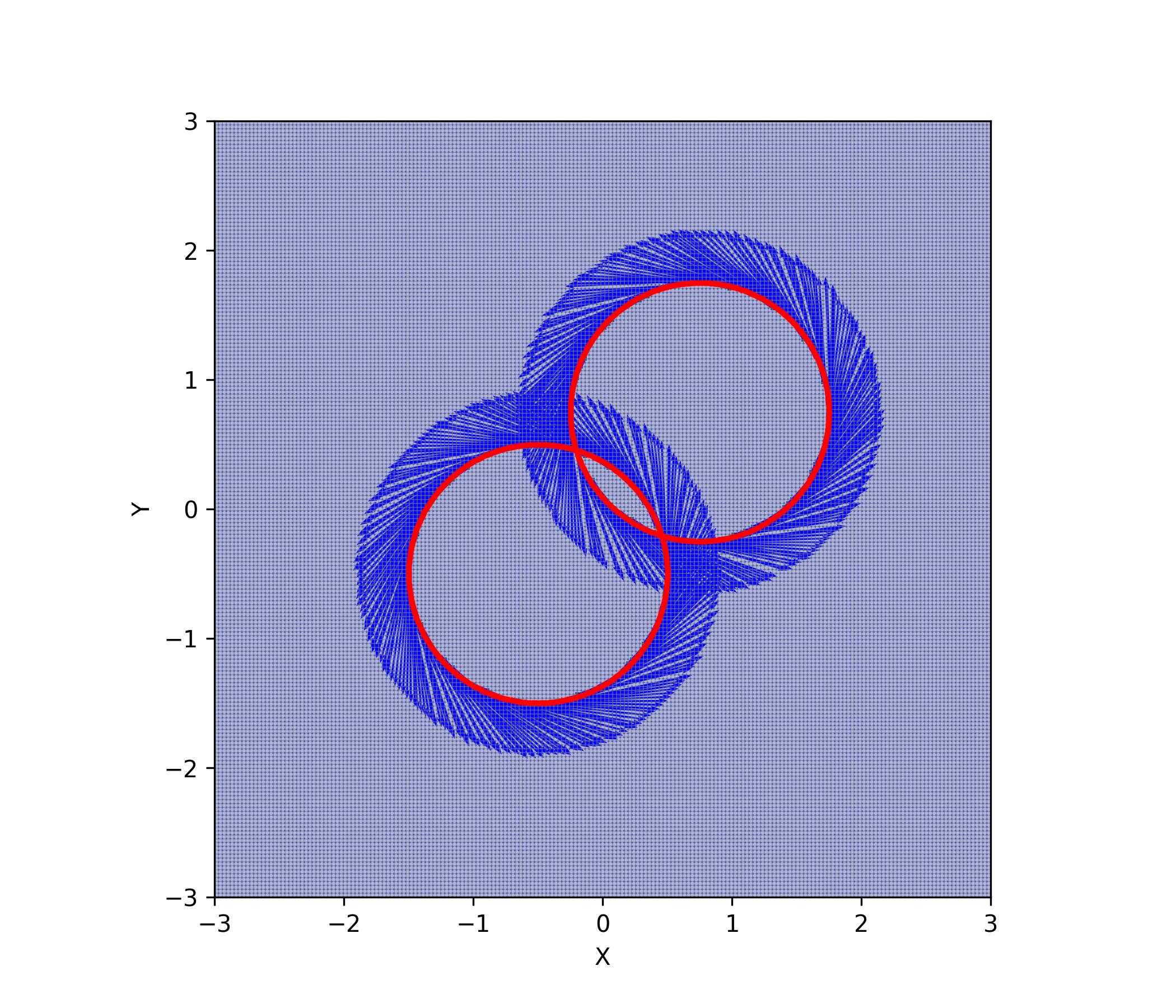}} \\
    \subfigure[]{\includegraphics[width=0.325\textwidth]{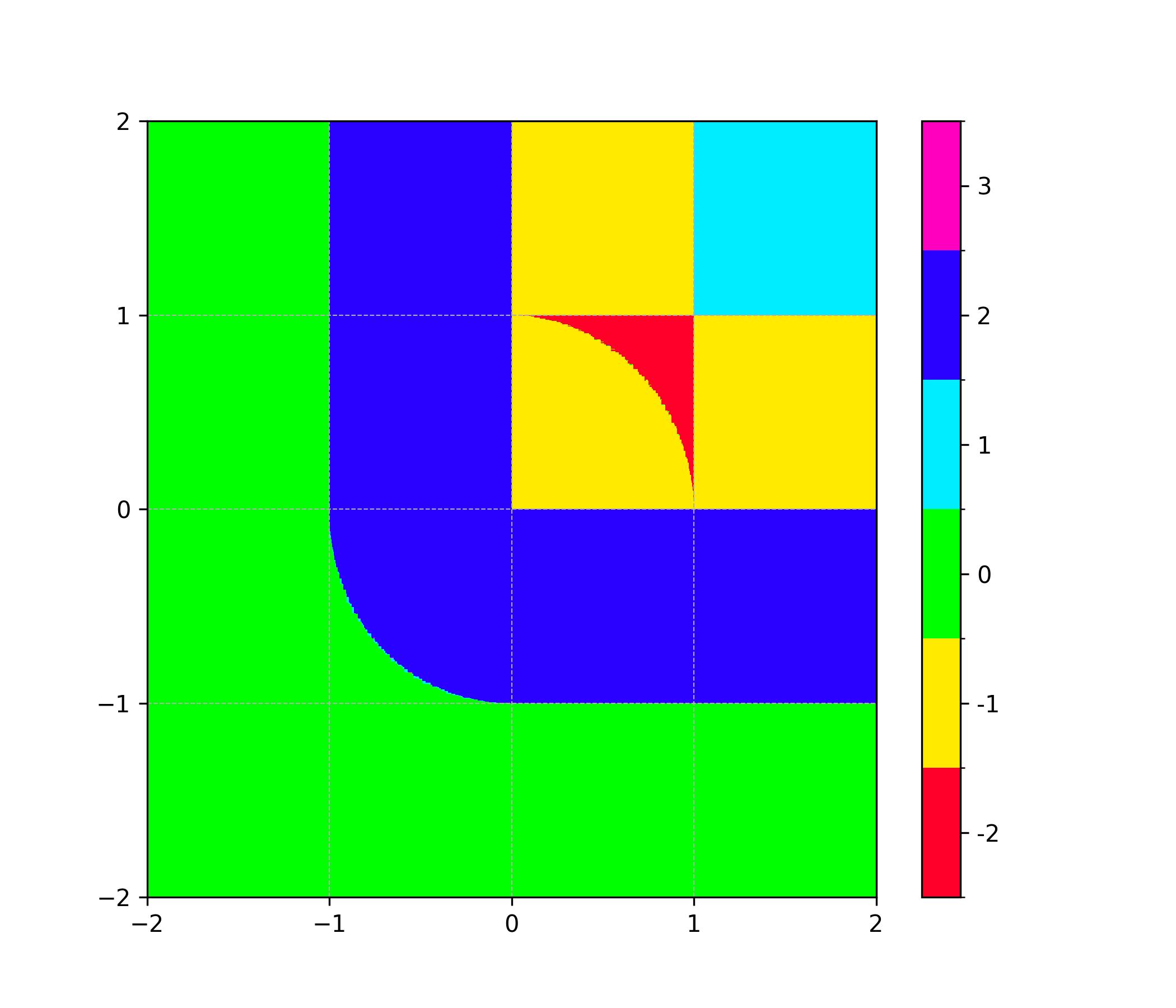}}
    \subfigure[]{\includegraphics[width=0.325\textwidth]{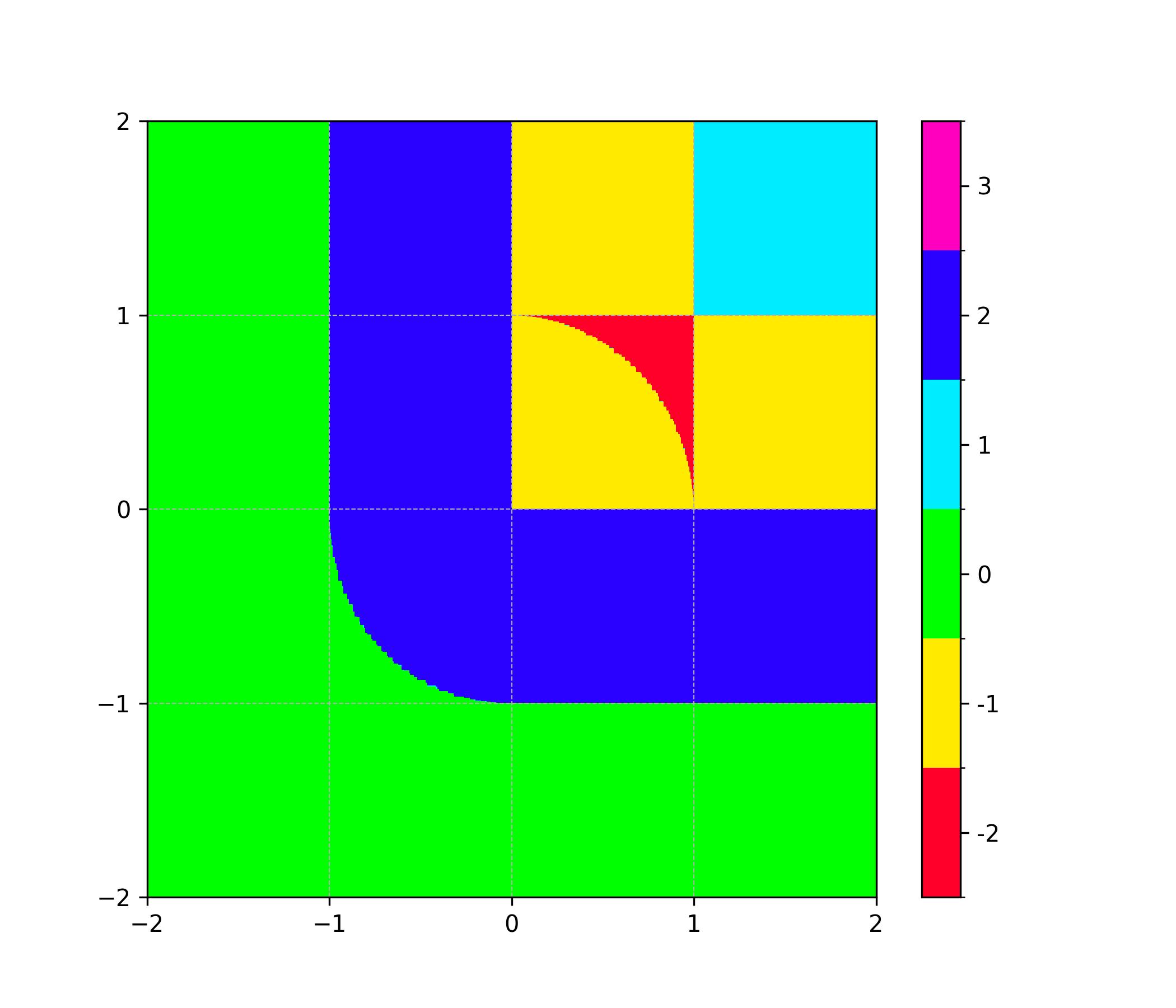}}
    \subfigure[]{\includegraphics[width=0.325\textwidth]{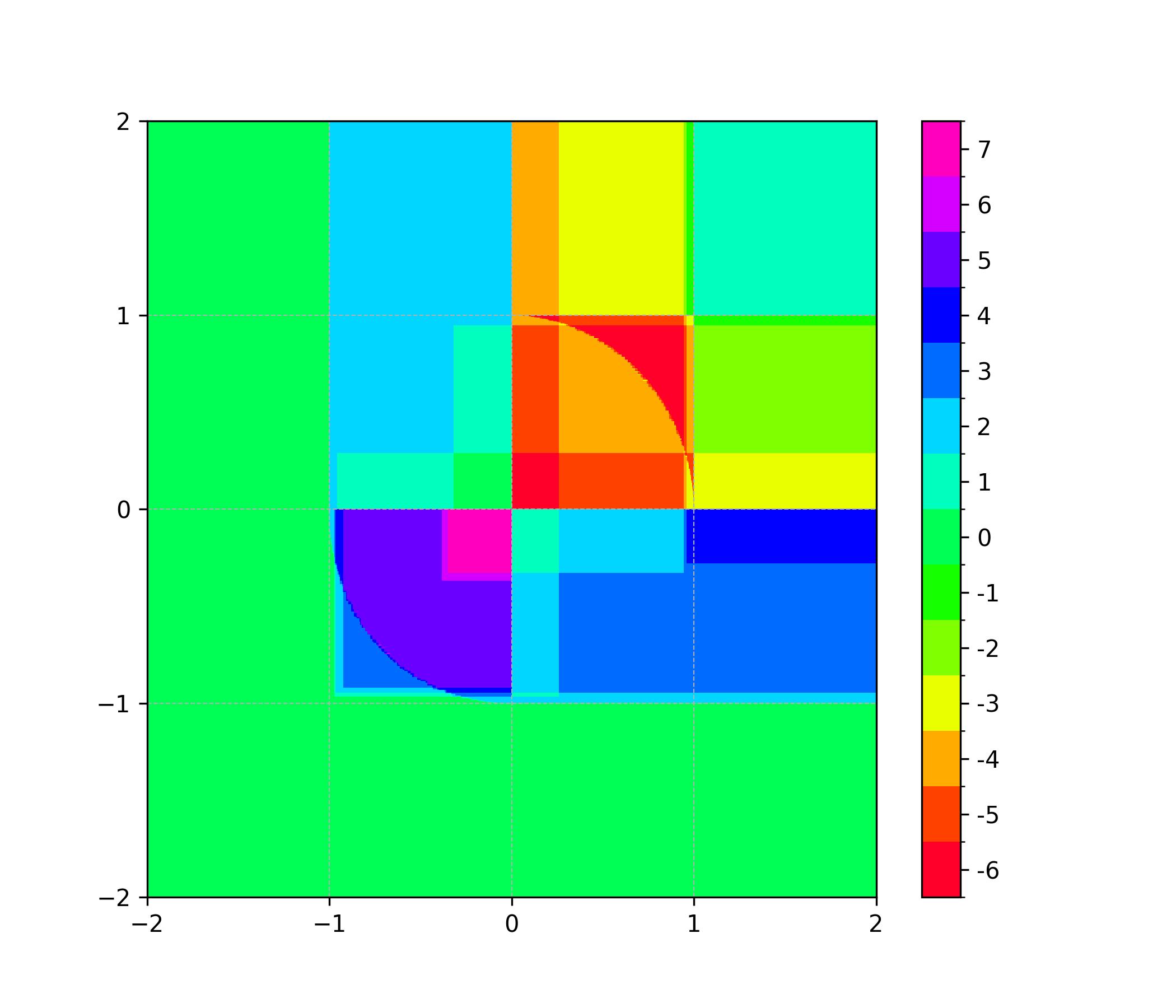}}
    \caption{Representative configurations of multiple periodic trajectories: (a) concentric circles, (b) disjoint circles, and (c) intersecting circles together with their associated tangent vector fields. Panels (d)-(f) present the corresponding Euler Characteristic Profiles (ECPs), demonstrating the sensitivity of the proposed descriptors to changes in the topology of the underlying geometric configuration}
    \label{fig:circle_vf_ECP}
\end{figure}

\subsubsection{DoD and BEPE of closed orbits}
For the DoD representation a circular kernel density estimate based on the von~Mises kernel with the concentration parameter $\kappa = 1$ is applied to the density on a uniform angular grid with $360$ bins. The resulting densities for the circle and blob are shown in Fig.~\ref{fig:circle_dod_bepe}(a) and (b), respectively. We observe that the resulting density is uniform, as all the directions are represented uniformly. 

In parallel, we construct the BEPE of the vector field by associating to each sampled vector a point in $\mathbb{R}^4$ of the form $(x,y,v_x,v_y)$, encoding both the position and direction of the vector. The resulting high-dimensional point cloud is subsequently projected to $\mathbb{R}^2$ using t-SNE with perplexity equal to $30$. The embedding obtained for the circular case is shown in Fig.~\ref{fig:circle_dod_bepe}b.
\begin{figure}[h]
    \centering
    \subfigure[]{\includegraphics[width=0.35\textwidth]{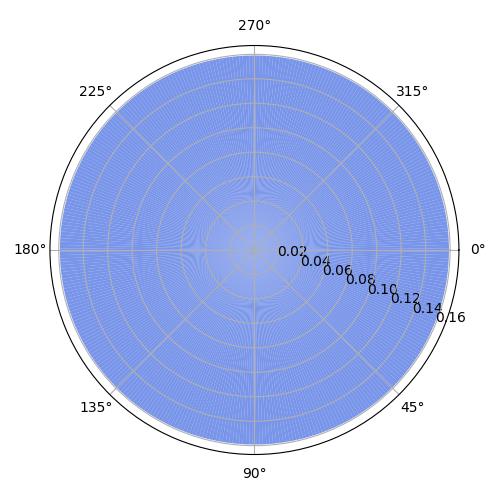}}
    \hspace{2cm}
    \subfigure[]{\includegraphics[width=0.32\textwidth]{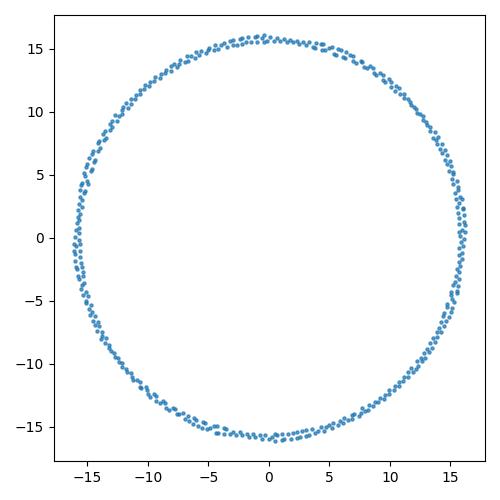}}\\
    \subfigure[]{\includegraphics[width=0.35\textwidth]{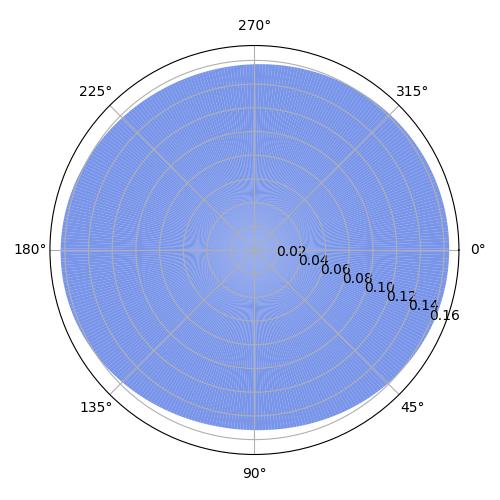}}
    \hspace{2cm}
    \subfigure[]{\includegraphics[width=0.32\textwidth]{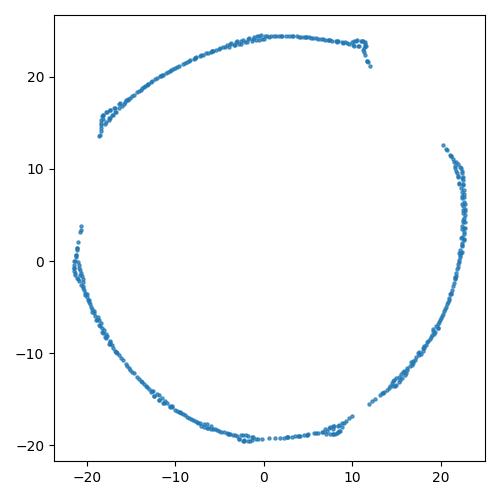}}\\
    \caption{Comparison of the proposed vector-field descriptors for representative periodic trajectories. Panels (a) and (c) show the Density of Directions (DoD) distributions for the circle and blob-shaped trajectory, respectively, while panels (b) and (d) present two-dimensional t-SNE visualizations of the corresponding four-dimensional Begin-End Point Embedding (BEPE)}
    
    \label{fig:circle_dod_bepe}
\end{figure}
A key observation is that, for simple closed curves such as circle/ ellips/ blob, both DoD and BEPE yield very similar representations. As a consequence, these methods do not effectively distinguish between geometrically different but directionally similar curves.
In contrast, the ECP better captures subtle geometric differences through topological changes in the filtration, allowing for a more discriminative description of the underlying structure. Even though the ECP for the circle and ellipse coincides, more complex deformations such as the blob produce noticeably richer profiles, reflecting the increased variability of tangent directions.
These observations highlight a fundamental difference between the methods. The DoD representation captures the global distribution of directions but discards spatial information, making it insensitive to geometric deformations that preserve directional statistics. The BEPE representation incorporates positional information, but for simple periodic structures this information remains highly redundant and does not significantly enhance discriminative power. In contrast, ECP encodes both geometric and topological features through the evolution of sublevel sets, providing a more informative signature.
Importantly, even for this relatively simple benchmark, ECP demonstrates superior discriminative capability compared to DoD and BEPE. This suggests that, for the purpose of analysing closed orbits in vector fields, ECP provides a more robust and informative representation. Consequently, in the remainder of this work, we focus primarily on the ECP-based analysis.

\subsection{Linear 2D Systems with a Stationary Point}
Linear two-dimensional systems with a stationary point form one of the basic classes of dynamical systems. Since they exhibit the principal local types of equilibria, including nodes, saddles and foci, they provide a natural testbed for the ECP. In this subsection, we examine to what extent ECP captures signatures of equilibrium type, encodes stability properties, and remains informative under perturbations.

Let us start with simple linear 2D autonomous systems (without external drive and noise). Precisely, we consider
\begin{equation}\label{eq:linear_2d}
\begin{aligned}
\dot{x} &= ax+by,\\
\dot{y} &= cx+dy,
\end{aligned}
\end{equation}
we will write
\begin{equation}
M:=\begin{bmatrix}
a &  b\\
c & d 
\end{bmatrix}\label{eq:matrix}
\end{equation}
to denote the matrix of the system. The properties of the linear 2D system~\ref{eq:linear_2d} are determined by the eigenvalues $\lambda_1$ and $\lambda_2$ and corresponding eigenvectors $v_1$ and $v_2$ of $M$. 
The systems can be classified into four canonical types:
\begin{itemize}
\item Stable node: $\lambda_1, \lambda_2 < 0$
\item Unstable node: $\lambda_1, \lambda_2 > 0$
\item Saddle point: $\lambda_1 \lambda_2 < 0$
\item Focus (stable/unstable): $\lambda_{1,2} \in \mathbb{C}, \ \mathrm{Im}(\lambda) \neq 0$
\end{itemize}
This classification depends on two invariants:
\[
\mathrm{tr}(A) = \lambda_1 + \lambda_2, \quad
\det(A) = \lambda_1 \lambda_2.
\]
Particular examples are given below.
\[
M_{1,1}= \begin{pmatrix}
-1.5 & 0 \\
0 & -2
\end{pmatrix} \quad \textrm{(node)}, \ \lambda_{1}= - 1.5, \lambda_{2}= - 2.
\]
Here  $(0,0)$ is a stable fixed point of the system. The eigenvectors $v_1=(1 , 0)$ and $v_2=(0, 1)$ corresponding, respectively, to $\lambda_1$ and $\lambda_2$. We sample the vector field restricted to $x \in [-2, 2]$ and $y \in [-2, 2]$ using $51 \times 51$ grid. We will consider two cases of the ECP computations: with and without vector normalization. The phase portrait and ECPs for these examples are shown in the Figure \ref{fig:sn_vf_ECP}. Note that with applying normalization we neglect the information on the magnitude of underlying vectors but without normalization vectors in the vicinity of fixed points (regardless of their stability) have very small vanishing magnitude which implies that both their geometrical coordinates are close to zero and thus fluctuations in vector field coordinates close to fixed points has little effect on the ECP computed which is not the necessarily the case once the normalization has been applied prior to ECP calculation. 

\begin{figure}[h]
    \centering
    \subfigure[]
    {\includegraphics[width=0.3\textwidth]{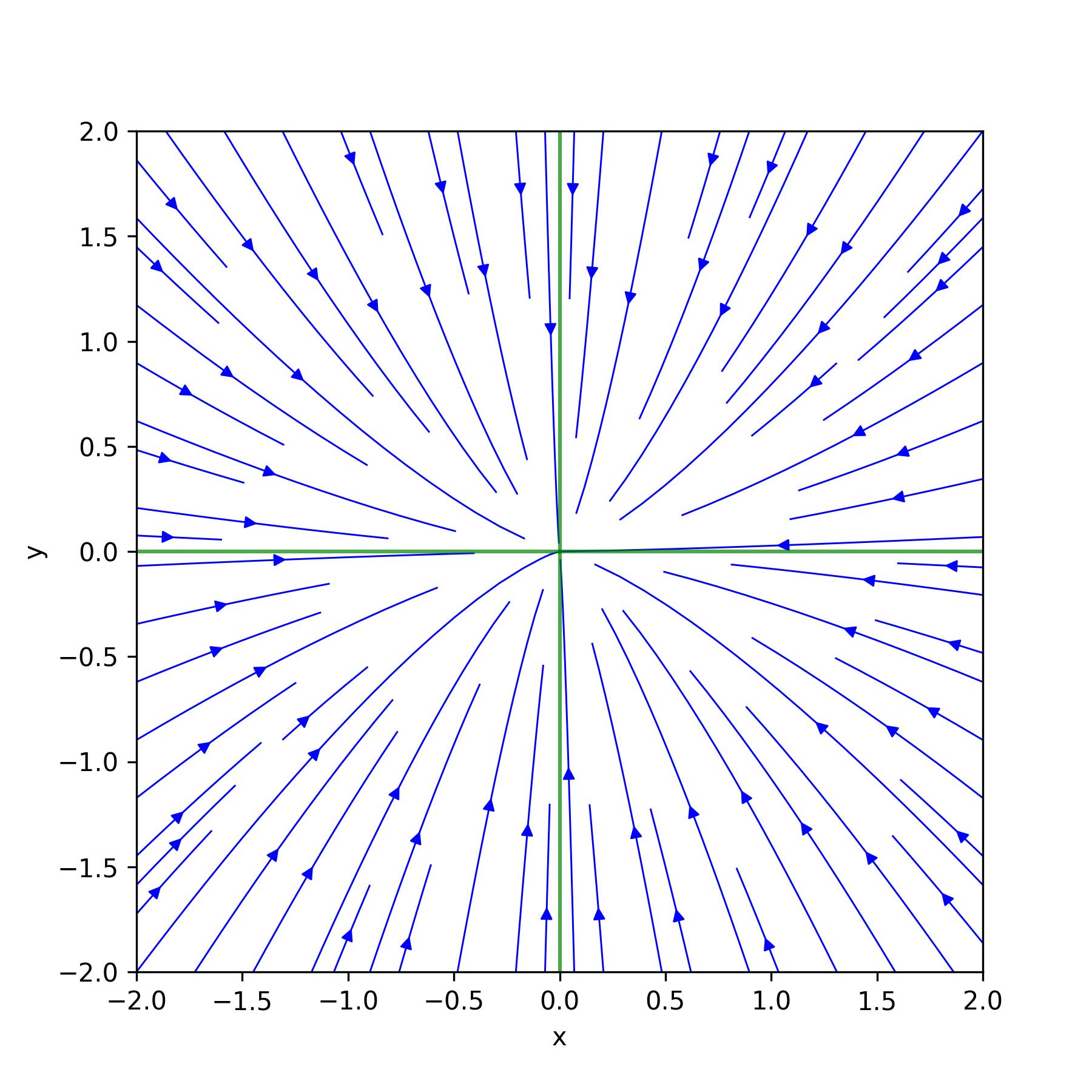}}
    \subfigure[]{\includegraphics[width=0.3\textwidth]{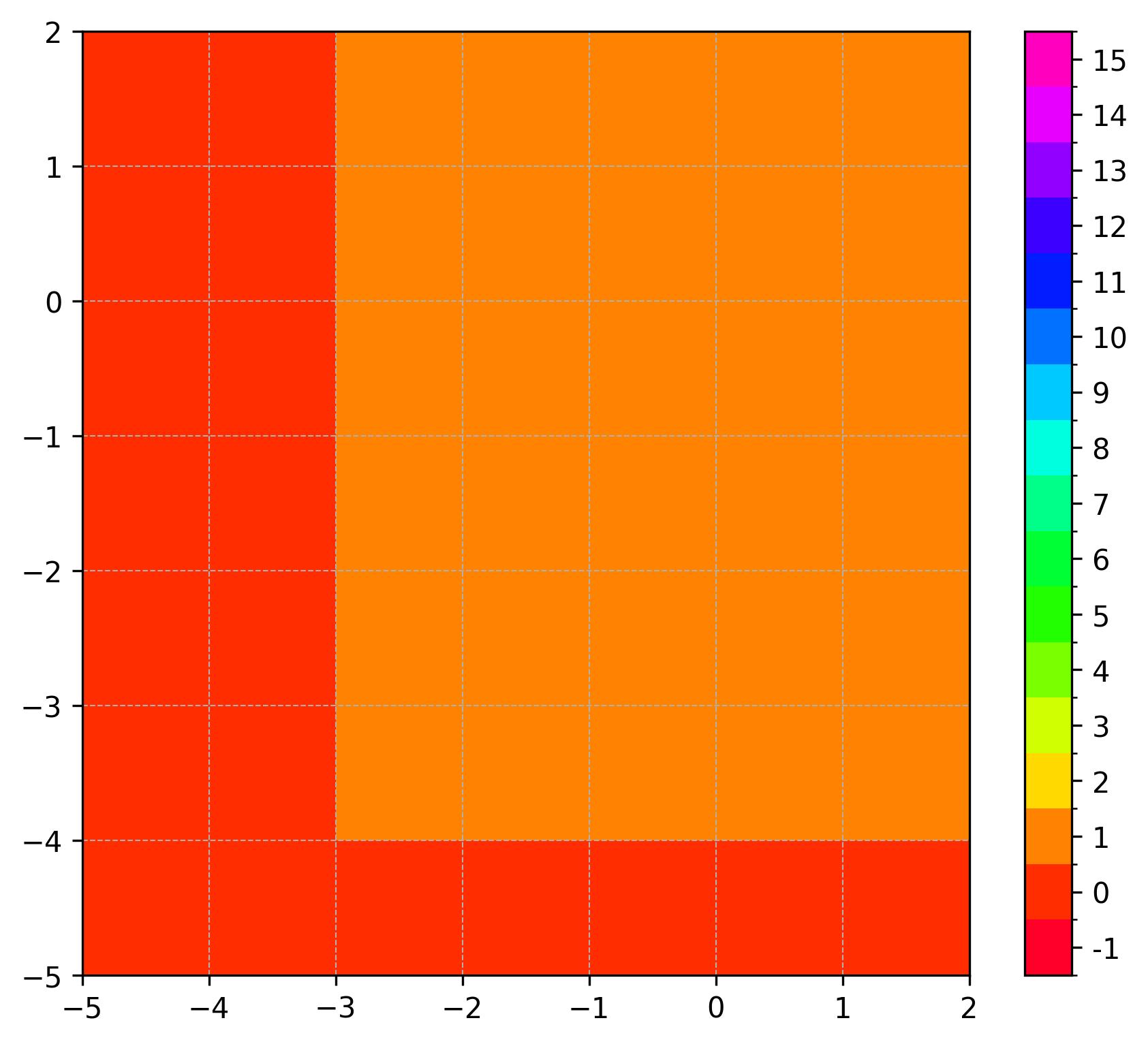}}
    \subfigure[]{\includegraphics[width=0.3\textwidth]{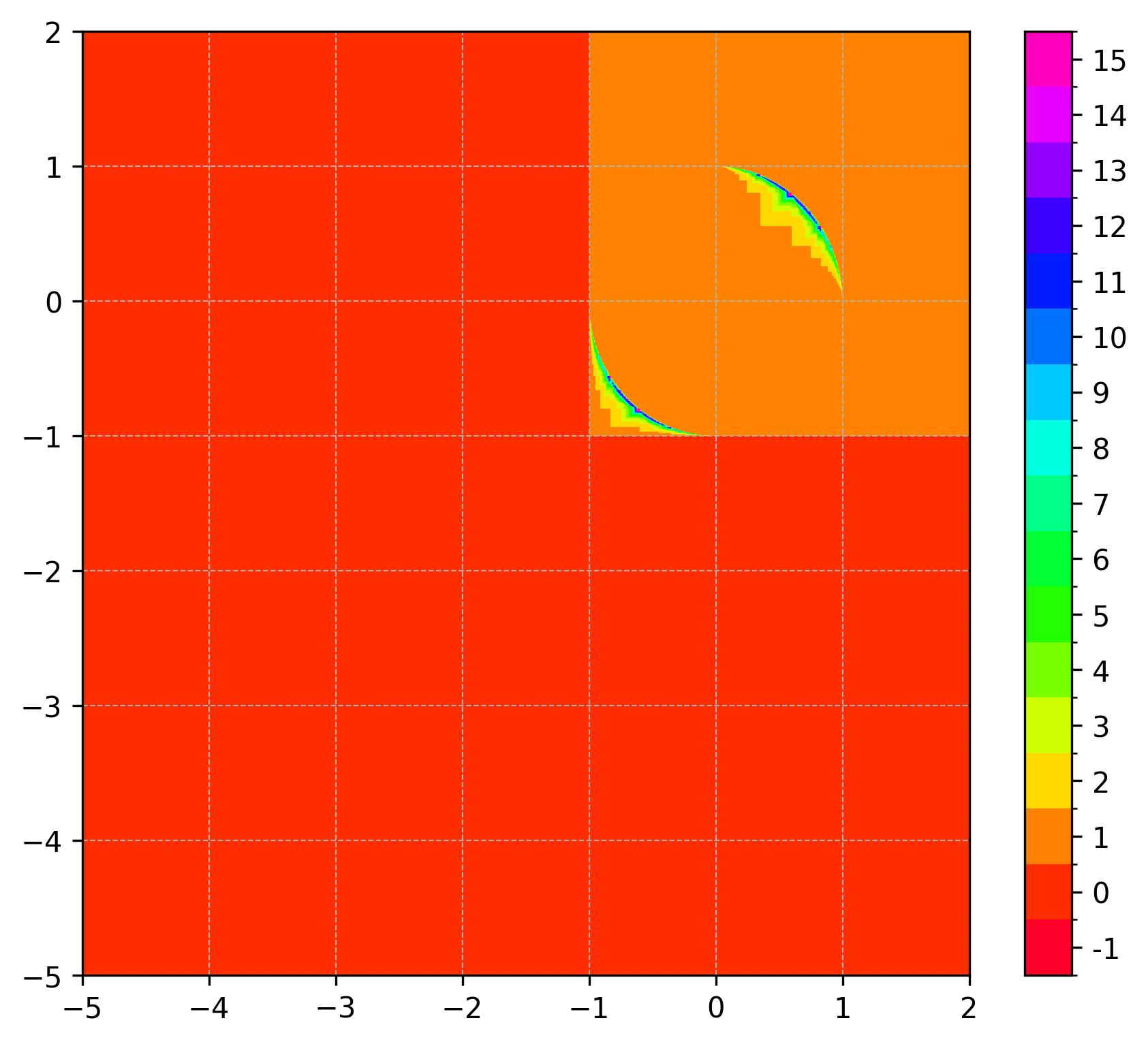}}
    \caption{Representative example of a stable node with eigenvalues $\lambda_1=-1.5$ and $\lambda_2=-2$: (a) phase portrait, (b) Euler Characteristic Profile computed from the original vector field, and (c) Euler Characteristic Profile obtained after vector normalization}
    \label{fig:sn_vf_ECP}
\end{figure}

The example in Figure~\ref{fig:sn_vf_ECP} already shows that the ECP provides a nontrivial signature even for the simplest linear systems with a stationary point. At the same time, it demonstrates that the resulting profile depends on how the vector field is represented in the filtration. Without normalization, the ECP reflects both direction and magnitude of the vectors, while after normalization it emphasizes directional structure and suppresses amplitude effects. This distinction will be important in the examples below, where we use ECP both to identify signatures of equilibrium type and to investigate which aspects of the dynamics remain visible under different choices of filtration.

Let us now focus on Fig. ~\ref{fig:sn_vf_ECP}(b). In the non-normalized case, one can clearly identify a distinguished \textit{corner point} \((x,y)\), equal in this example to \((-3,-4)\), at which the ECP changes from \(0\) to \(1\). Its location is not accidental: it is completely determined by the range of the domain and by the parameters of the system. More precisely, for diagonal linear systems, we obtain the following description.

\begin{lemma}
Consider the system~\eqref{eq:linear_2d} with \(b=c=0\), sampled on the rectangular domain
\[
D=[x_{\min},x_{\max}] \times [y_{\min},y_{\max}],
\]
and let \(a\) and \(d\) denote the nonzero diagonal entries of the matrix.
Then the associated ECP satisfies
\[
ECP(x,y)=
\begin{cases}
1, & x \ge x_{\mathrm{ecp}}, y \ge y_{\mathrm{ecp}},\\
0, & \text{otherwise},
\end{cases}
\]
where \(x_{\mathrm{ecp}}\) and \(y_{\mathrm{ecp}}\) are given by
\begin{align*}
    \text{if } a > 0 \text{ and } d > 0, & \quad x_{\mathrm{ecp}} = a\,x_{\min}, \quad y_{\mathrm{ecp}} = d\,y_{\min}, \\
    \text{if } a < 0 \text{ and } d < 0, & \quad x_{\mathrm{ecp}} = a\,x_{\max}, \quad y_{\mathrm{ecp}} = d\,y_{\max}, \\
    \text{if } a > 0 \text{ and } d < 0, & \quad x_{\mathrm{ecp}} = a\,x_{\min}, \quad y_{\mathrm{ecp}} = d\,y_{\max}, \\
    \text{if } a < 0 \text{ and } d > 0, & \quad x_{\mathrm{ecp}} = a\,x_{\max}, \quad y_{\mathrm{ecp}} = d\,y_{\min}.
\end{align*}
\end{lemma}
\begin{remark}
The four sign configurations of \(a\) and \(d\) correspond to the basic diagonal linear systems: sources, sinks and saddles. In particular, the location of the corner point \((x_{\mathrm{ecp}},y_{\mathrm{ecp}})\) is determined by selecting, in each coordinate, the extreme value of the sampled domain from which the corresponding component of the vector field attains its minimal value. Thus, the ECP of such a system is completely determined by the signs of \(a\) and \(d\) together with the bounding box of the sampled domain.
\end{remark}

Let us now consider additional examples of two-dimensional linear systems. The corresponding system matrices together with their phase portraits are presented in Fig.~\ref{fig:linear_examples} (note that in each phase portrait, the nullclines corresponding to $\dot{x}=0$ and $\dot{y}=0$ are indicated in green). For those examples we compare several feature representations for multifiltration:

\begin{enumerate}
\item raw vector field samples,
\item vector field augmented with orientation angle,
\item vector field augmented with curl,
\item vector field augmented with divergence,
\item eigenvalues of the Jacobian matrix.
\end{enumerate}
Note that the eigenvalues are used as filtration coordinates by converting to real values by considering a four-valued filtration: $(\operatorname{Re}\lambda_1,\operatorname{Im}\lambda_1,\operatorname{Re}\lambda_2,\operatorname{Im}\lambda_2)$.

For all feature representations, we assume access only to discrete samples of the vector field on a regular grid. The required first-order spatial derivatives are approximated numerically from the sampled field values, yielding a pointwise estimate of the Jacobian matrix. The divergence, curl and Jacobian eigenvalues are subsequently computed at each grid point from these numerical derivatives, while the orientation angle is obtained directly from the local vector components. Therefore, the procedure does not require analytical derivatives nor prior knowledge of the governing equations, making it readily applicable to general sampled vector field data.


\newpage
\begin{figure*}[h!]
\centering
\begin{center}
\begin{minipage}[t]{0.5\linewidth}
\centering
\includegraphics[width=0.78\linewidth]{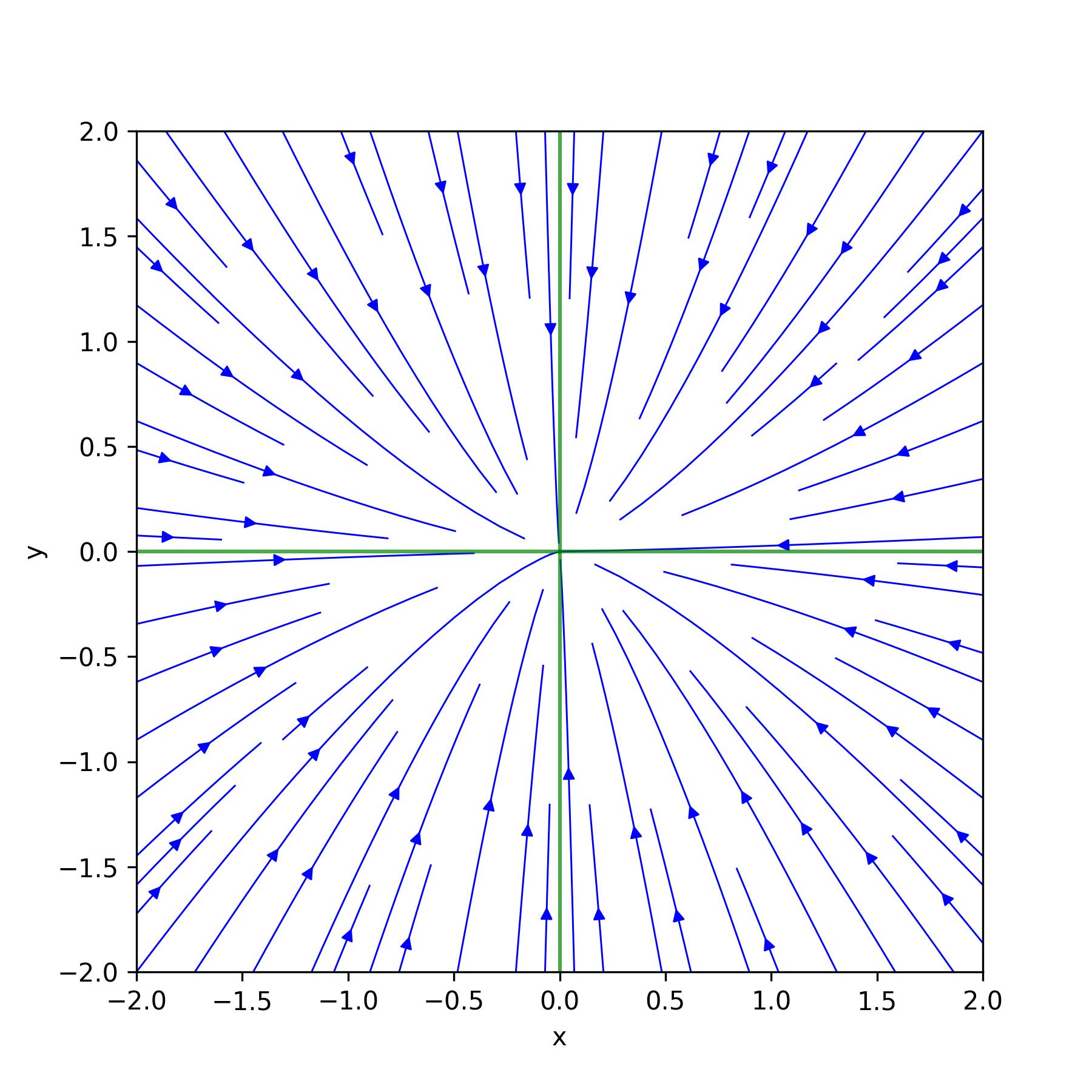}
\\
\(
M_{1,1}= 
\begin{pmatrix}-1.5 & 0.0 \\ 0.0 & -2.0\end{pmatrix} 
\quad \textrm{(stable node 1)}
\)
\end{minipage}%
\hfill
\begin{minipage}[t]{0.5\linewidth}
\centering
\includegraphics[width=0.78\linewidth]{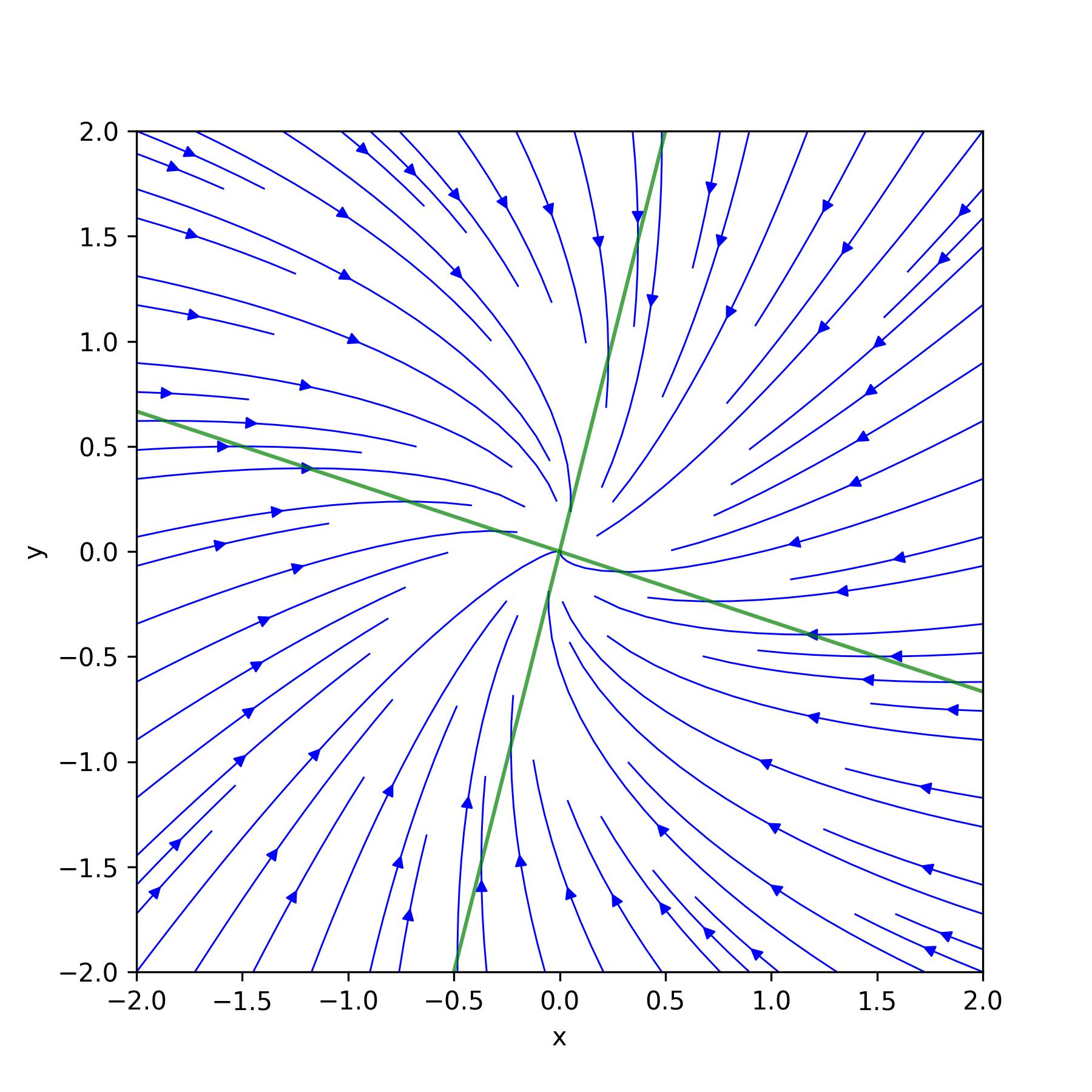}
\\
\(
M_{1,2}:=
\begin{pmatrix}-2.0 & 0.5 \\ -0.5 & -1.5\end{pmatrix} 
\quad \textrm{(stable node 2)}
\)
\end{minipage}


\begin{minipage}[t]{0.5\linewidth}
\centering
\includegraphics[width=0.78\linewidth]{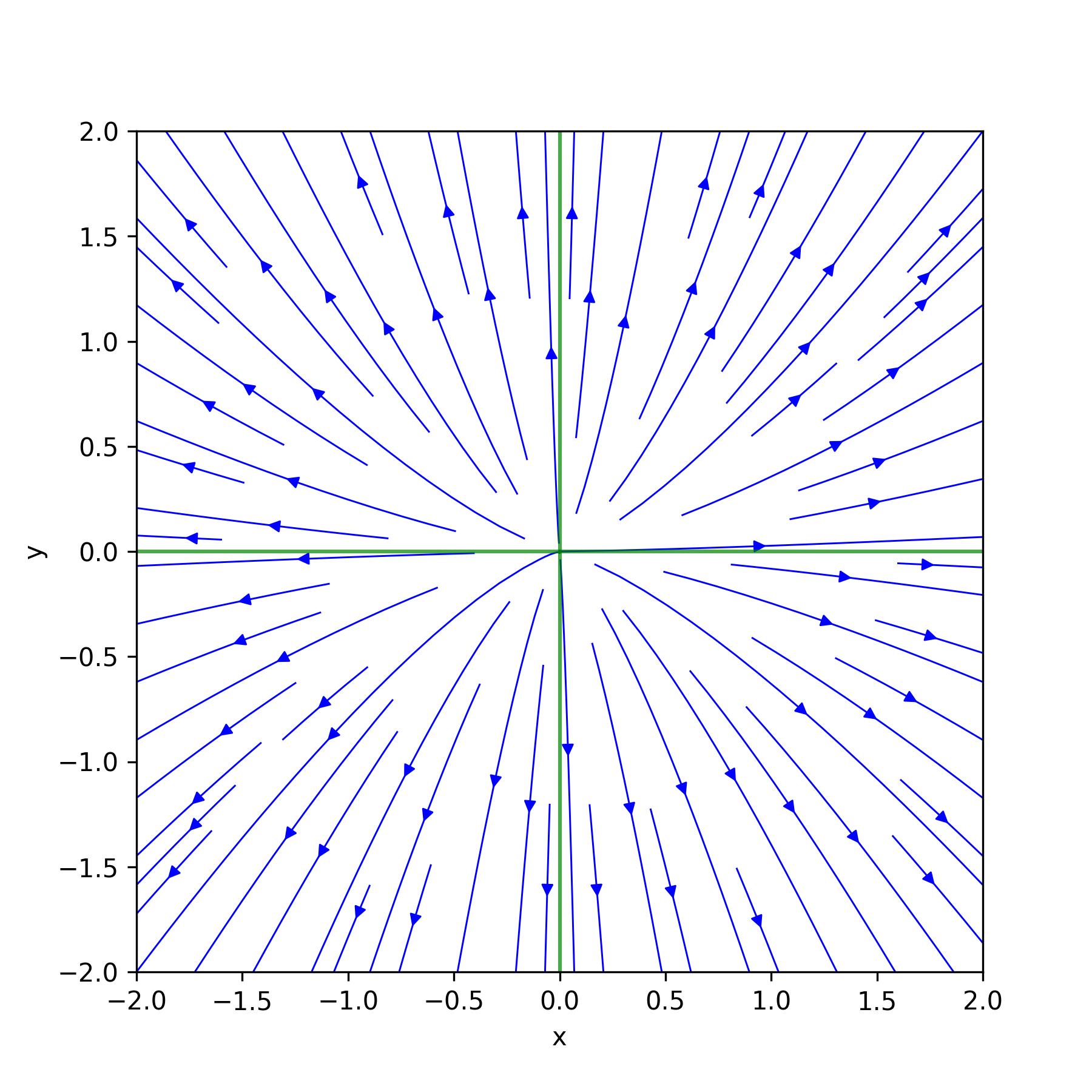}
\\
\(
M_{1,3}:=
\begin{pmatrix}1.5 & 0.0 \\ 0.0 & 2.0\end{pmatrix} 
\quad \textrm{(unstable node 1)}
\)
\end{minipage}%
\hfill
\begin{minipage}[t]{0.5\linewidth}
\centering
\includegraphics[width=0.78\linewidth]{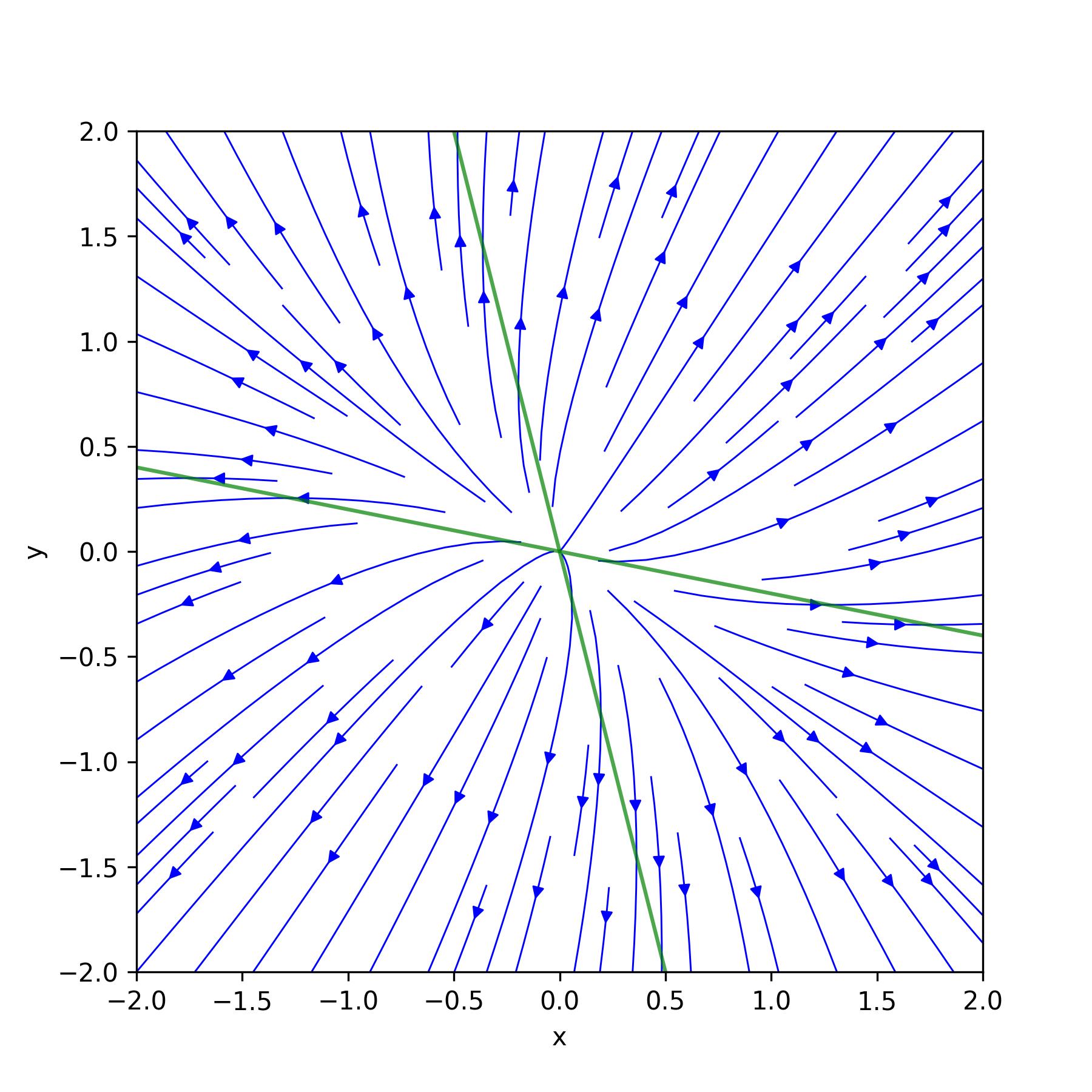}
\\
\(
M_{1,4}:=
\begin{pmatrix}2.0 & 0.5 \\ 0.5 & 2.5\end{pmatrix} 
\quad \textrm{(unstable node 2)}
\)
\end{minipage}


\begin{minipage}[t]{0.5\linewidth}
\centering
\includegraphics[width=0.78\linewidth]{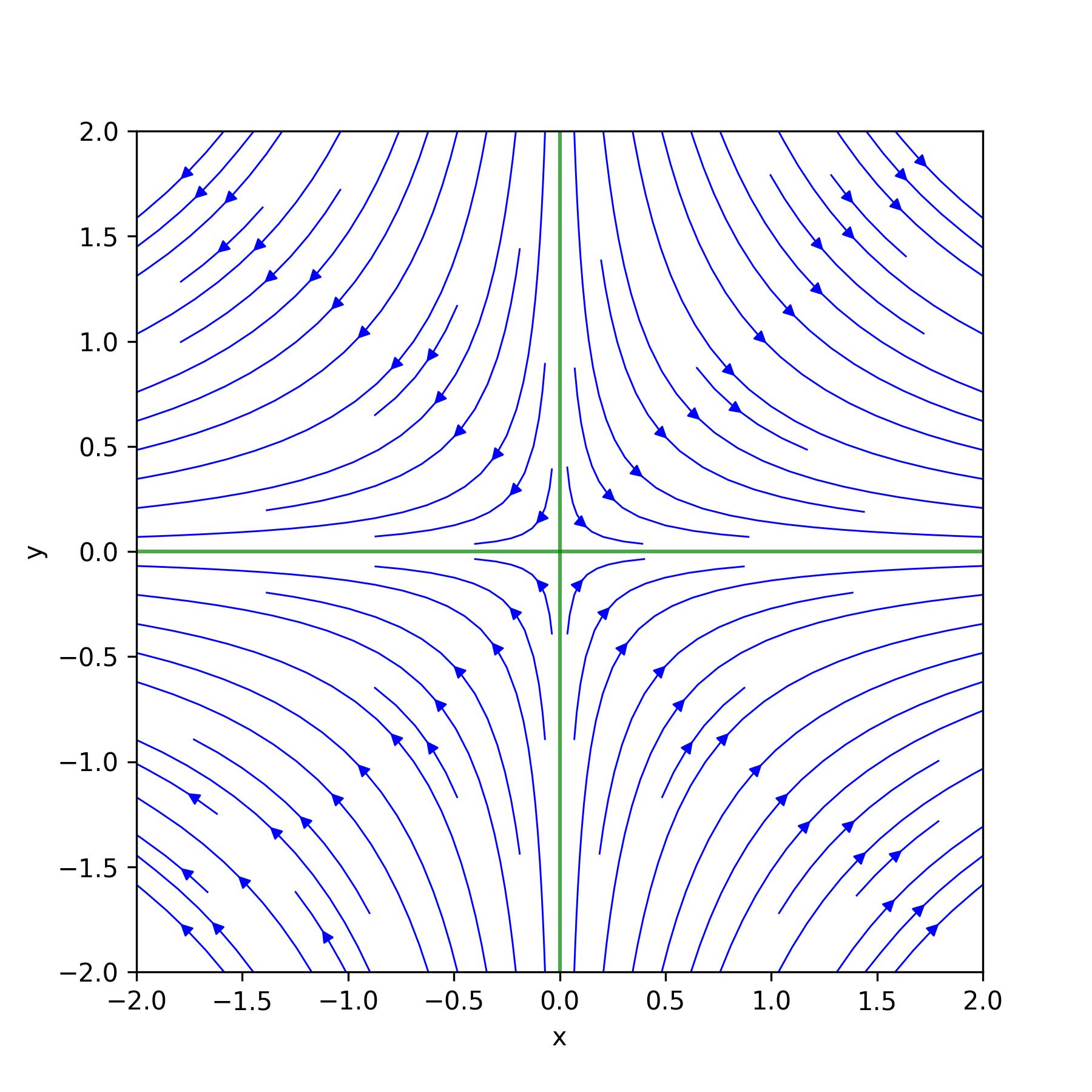}
\\
\(
M_{1,5}:=
\begin{pmatrix}1.0 & 0.0 \\ 0.0 & -1.0\end{pmatrix} 
\quad \textrm{(saddle 1)}
\)
\end{minipage}%
\hfill
\begin{minipage}[t]{0.5\linewidth}
\centering
\includegraphics[width=0.78\linewidth]{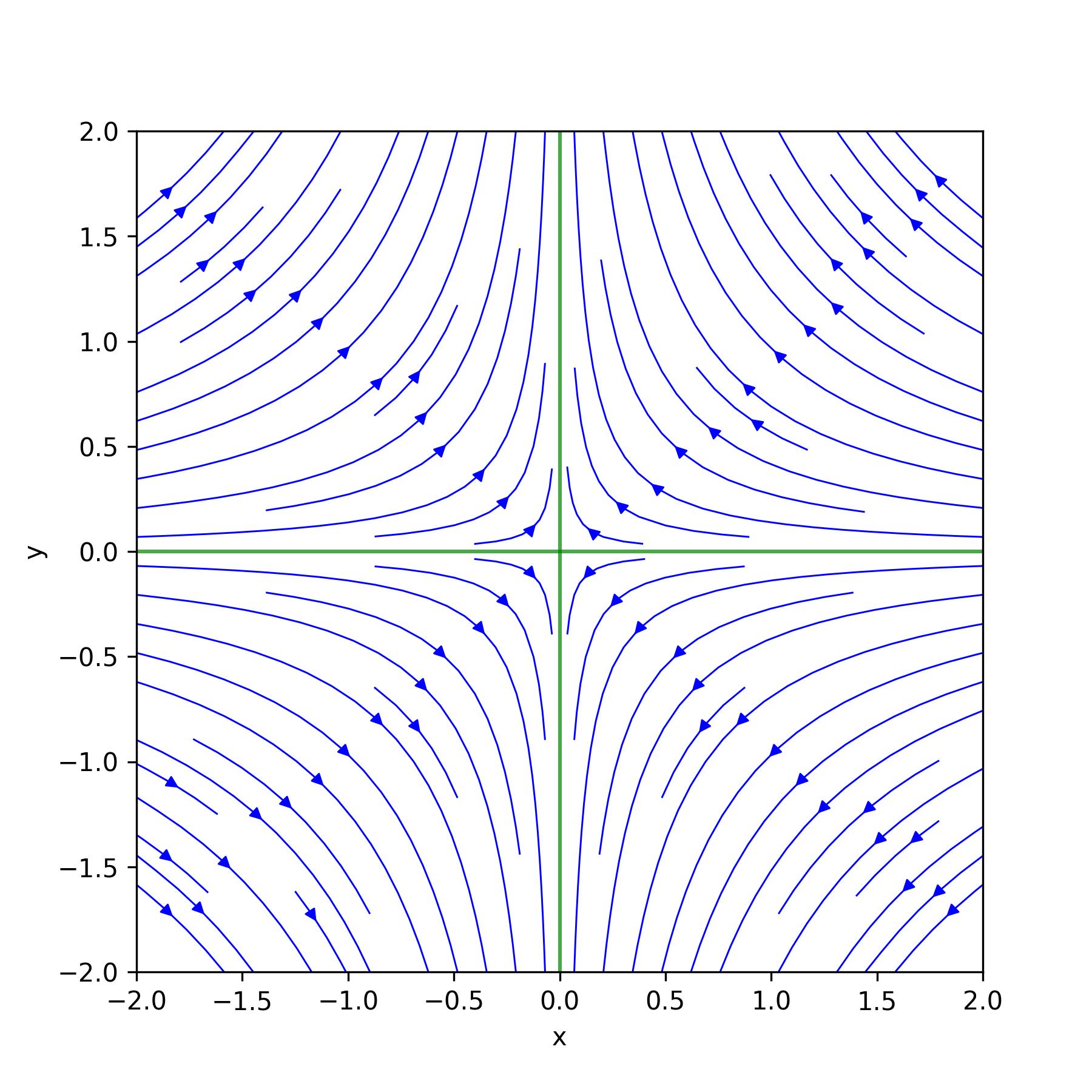}
\\
\(
M_{1,6}:=
\begin{pmatrix}-1.0 & 0.0 \\ 0.0 & 1.0\end{pmatrix} 
\quad \textrm{(saddle 2)}
\)
\vspace{5em}
\end{minipage}
\end{center}

\newpage

\end{figure*}
\begin{figure*}[h!]
\centering
\begin{center}
\begin{minipage}[t]{0.5\linewidth}
\centering
\includegraphics[width=0.78\linewidth]{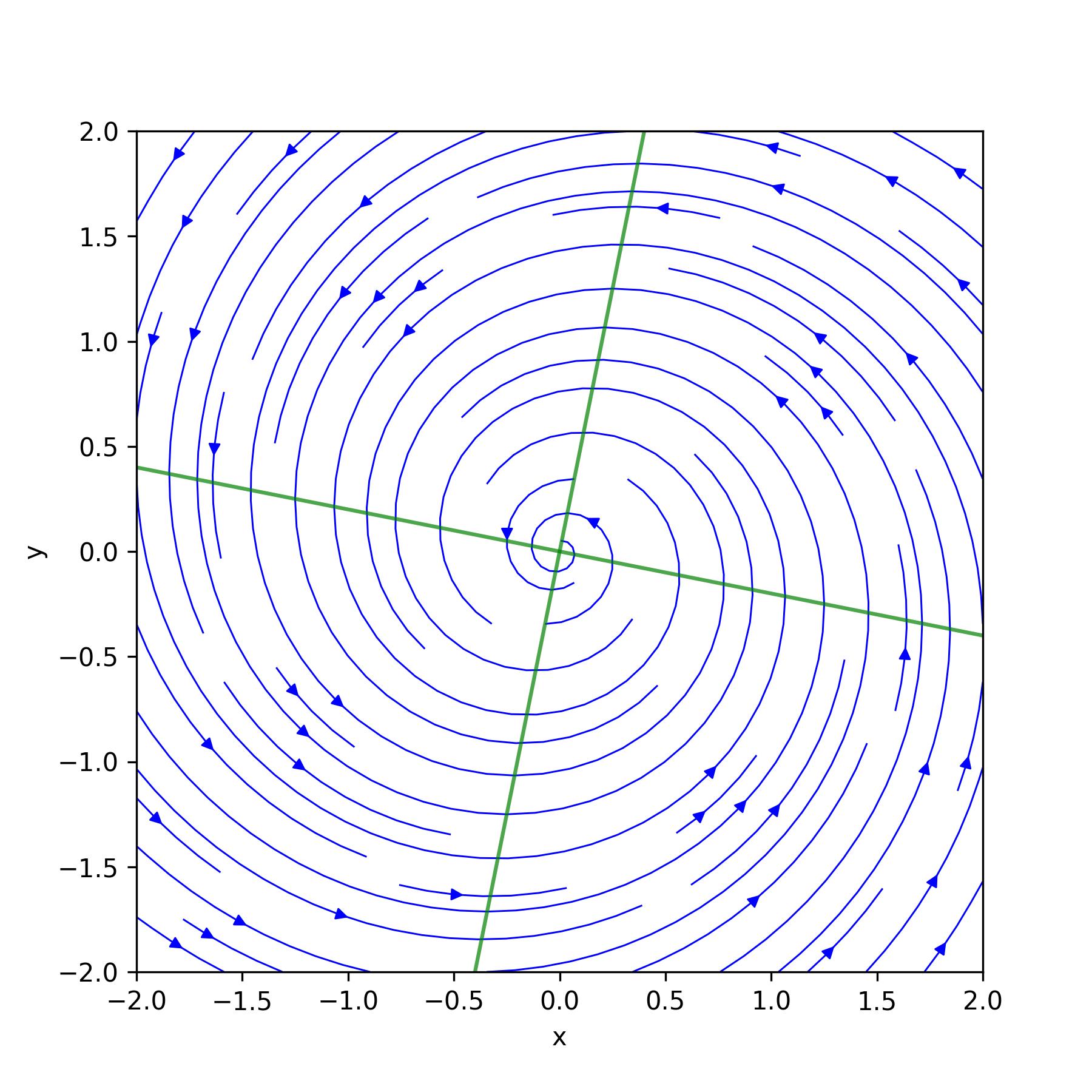}
\\
\(
M_{1,7}:=
\begin{pmatrix}-1.0 & -5.0 \\ 5.0 & -1.0\end{pmatrix} 
\quad \textrm{(stable focus 1)}
\)
\end{minipage}%
\hfill
\begin{minipage}[t]{0.5\linewidth}
\centering
\includegraphics[width=0.78\linewidth]{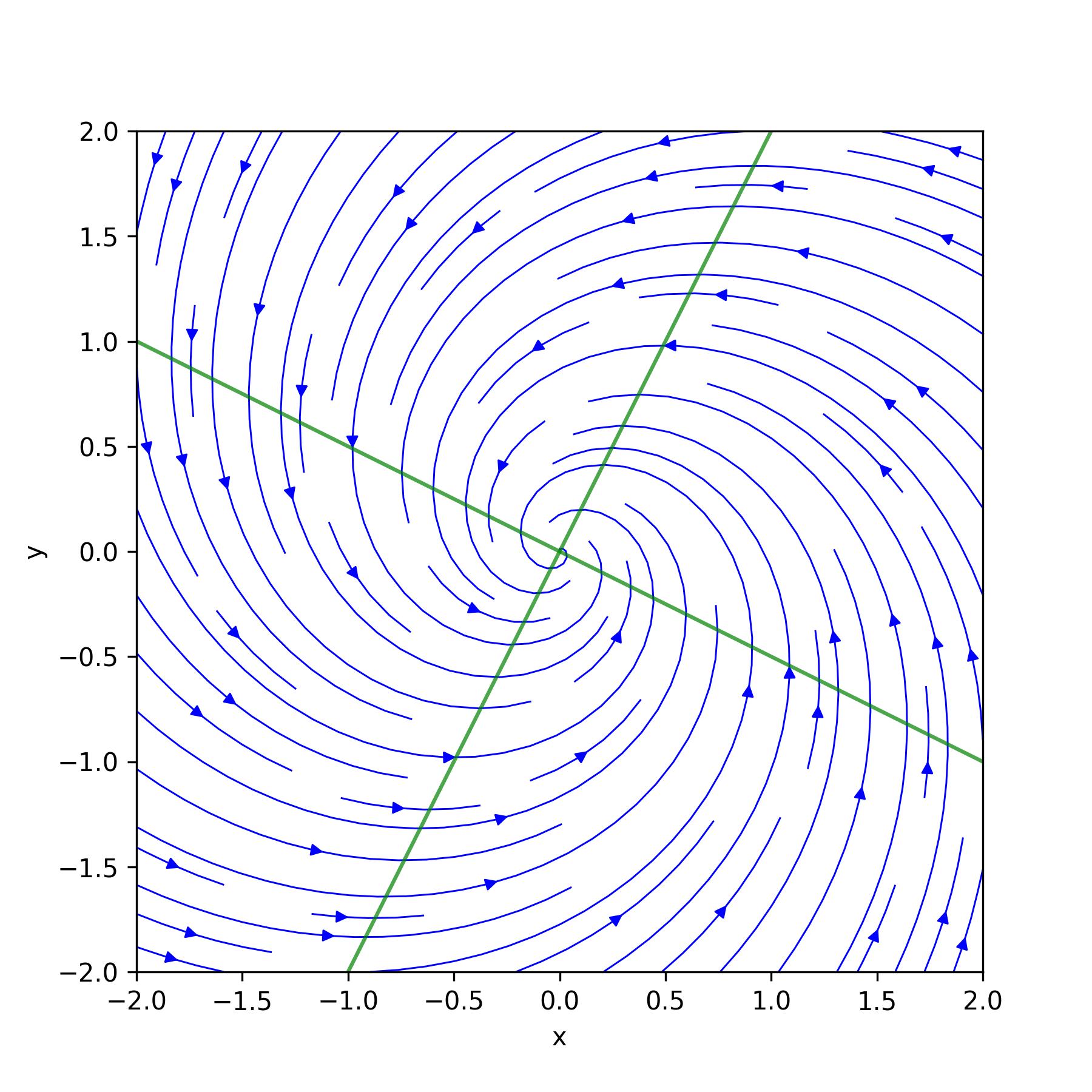}
\\
\(
M_{1,8}:=
\begin{pmatrix}-2.0 & -4.0 \\ 4.0 & -2.0\end{pmatrix} 
\quad \textrm{(stable focus 2)}
\)
\end{minipage}


\begin{minipage}[t]{0.5\linewidth}
\centering
\includegraphics[width=0.78\linewidth]{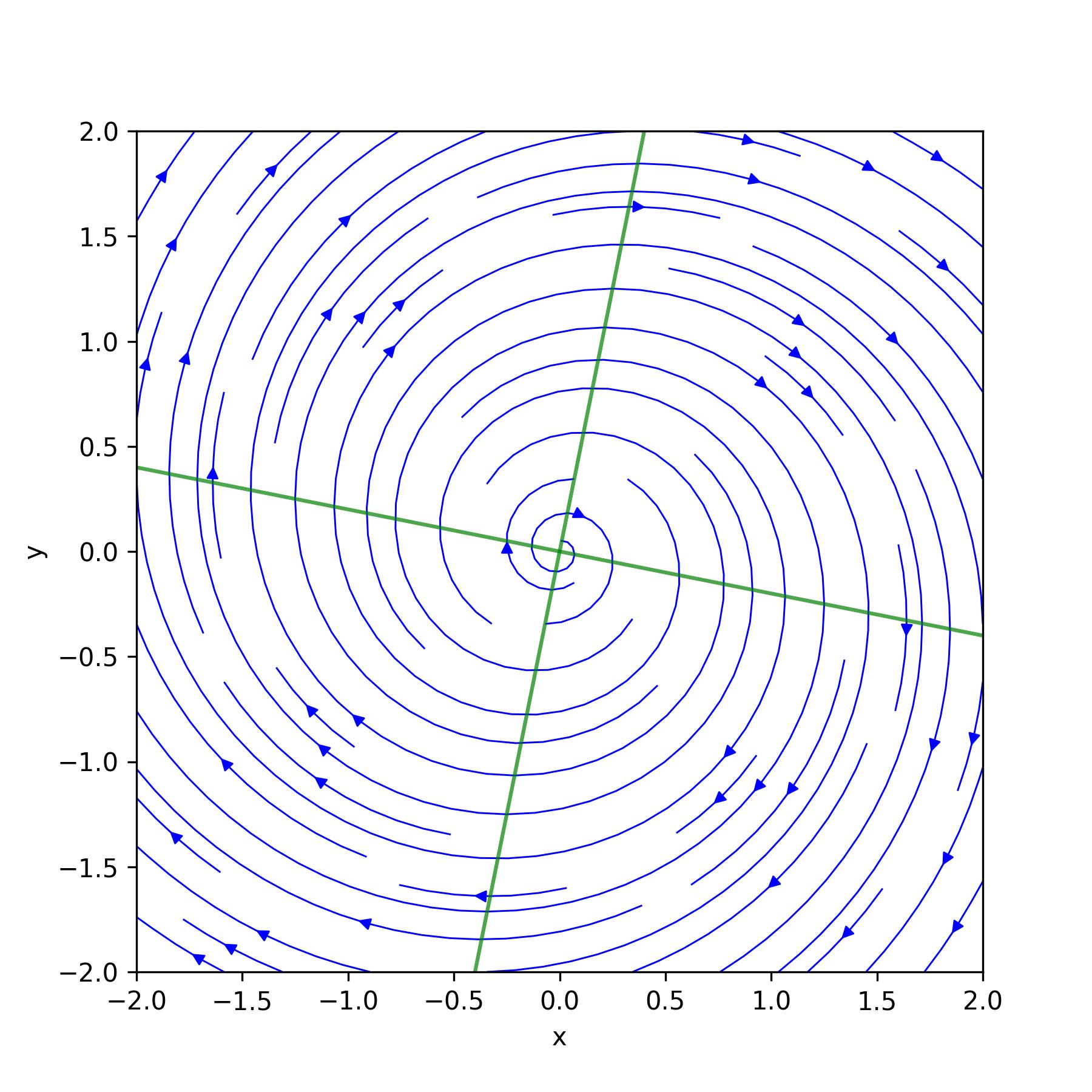}
\\
\(
M_{1,9}:=
\begin{pmatrix}1.0 & 5.0 \\ -5.0 & 1.0\end{pmatrix} 
\quad \textrm{(unstable focus 1)}
\)
\end{minipage}%
\hfill
\begin{minipage}[t]{0.5\linewidth}
\centering
\includegraphics[width=0.78\linewidth]{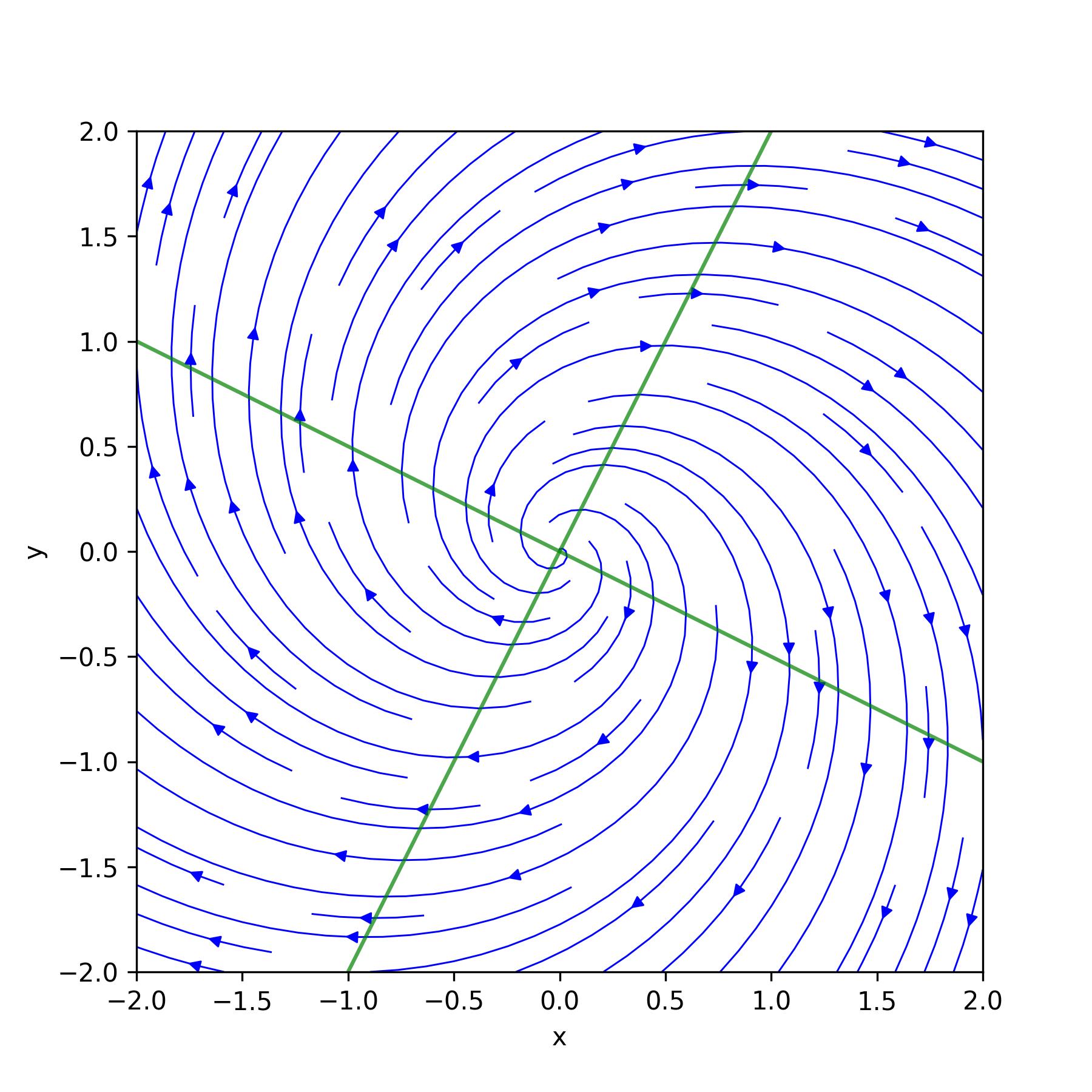}
\\
\(
M_{1,10}:=
\begin{pmatrix}2.0 & 4.0 \\ -4.0 & 2.0\end{pmatrix} 
\quad \textrm{(unstable focus 2)}
\)
\end{minipage}


\begin{minipage}[t]{0.5\linewidth}
\centering
\includegraphics[width=0.78\linewidth]{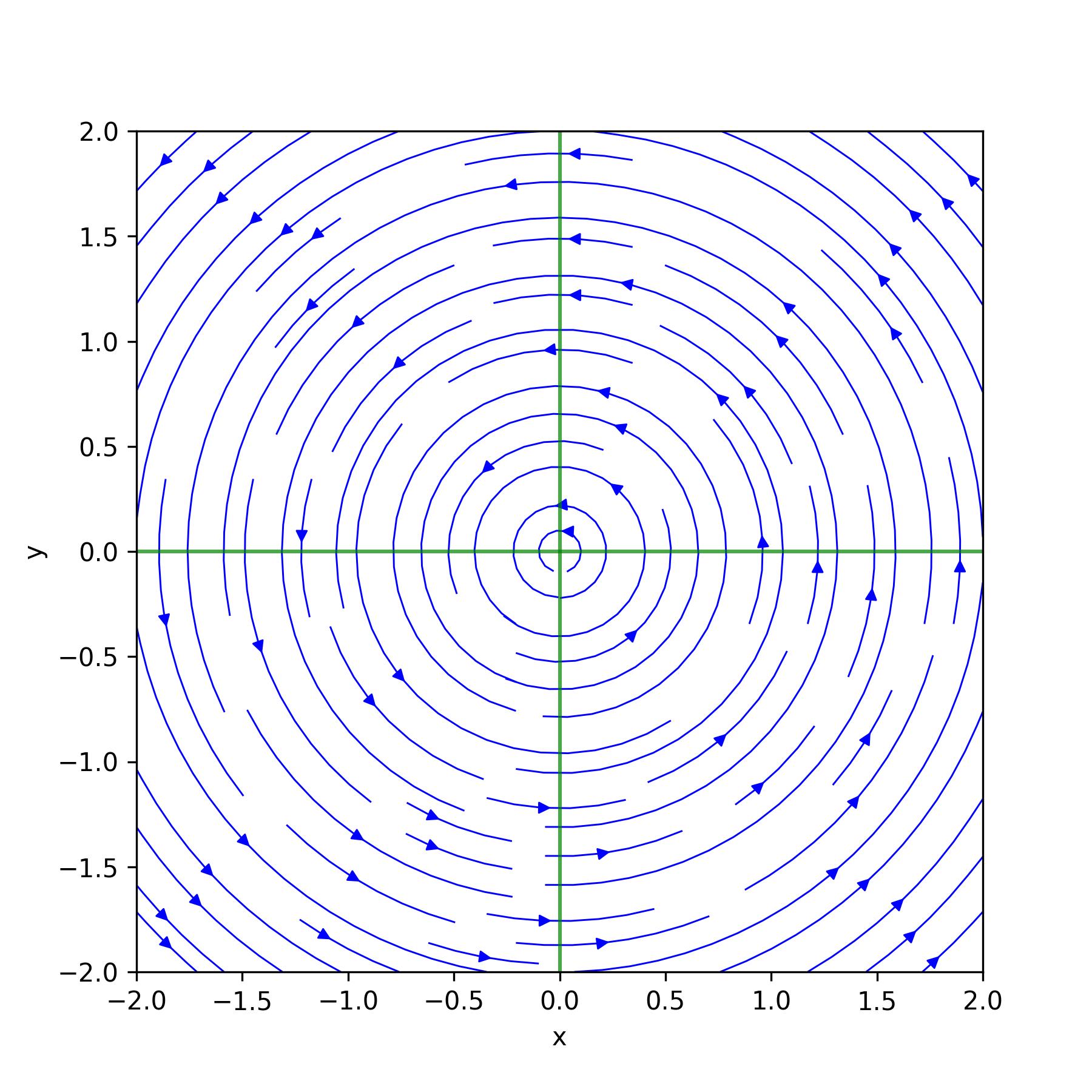}
\\
\(
M_{1,11}:=
\begin{pmatrix}0.0 & -1.0 \\ 1.0 & 0.0\end{pmatrix} 
\quad \textrm{(center 1)}
\)
\end{minipage}%
\hfill
\begin{minipage}[t]{0.5\linewidth}
\centering
\includegraphics[width=0.78\linewidth]{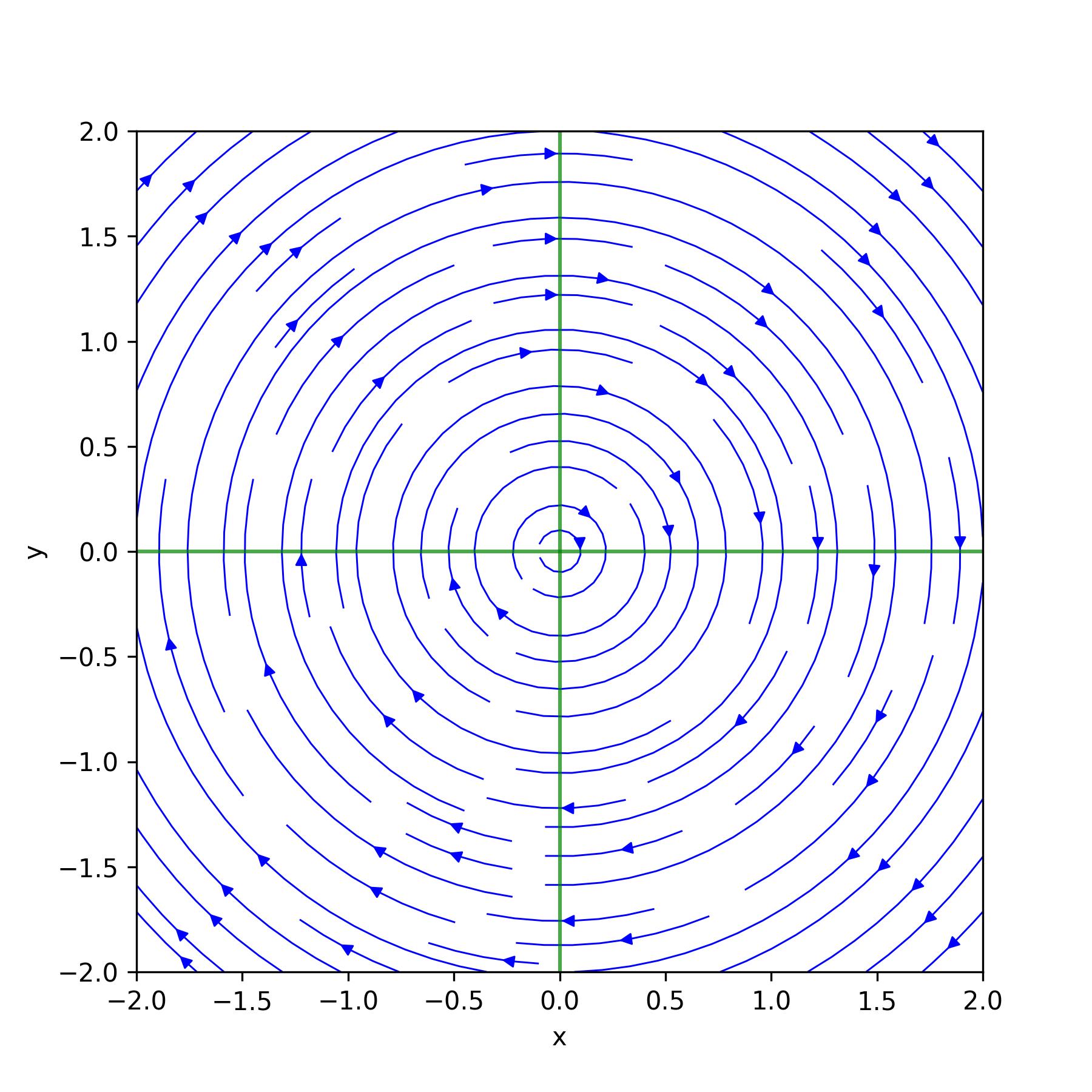}
\\
\(
M_{1,12}:=
\begin{pmatrix}0.0 & 1.0 \\ -1.0 & 0.0\end{pmatrix} 
\quad \textrm{(center 2)}
\)
\end{minipage}
\end{center}
\caption{
Representative examples of two-dimensional linear dynamical systems used throughout the numerical experiments. The figure includes phase portraits together with the corresponding system matrices. Green lines indicate the nullclines defined by $\dot{x}=0$ and $\dot{y}=0$. These benchmark systems constitute the reference dataset used to evaluate the ability of the proposed multifiltration strategies and Euler Characteristic Profiles to distinguish different qualitative types of equilibria}
\label{fig:linear_examples}
\end{figure*}

\begin{figure}[h]
    \centering
    \subfigure[]{\includegraphics[width=0.45\textwidth]{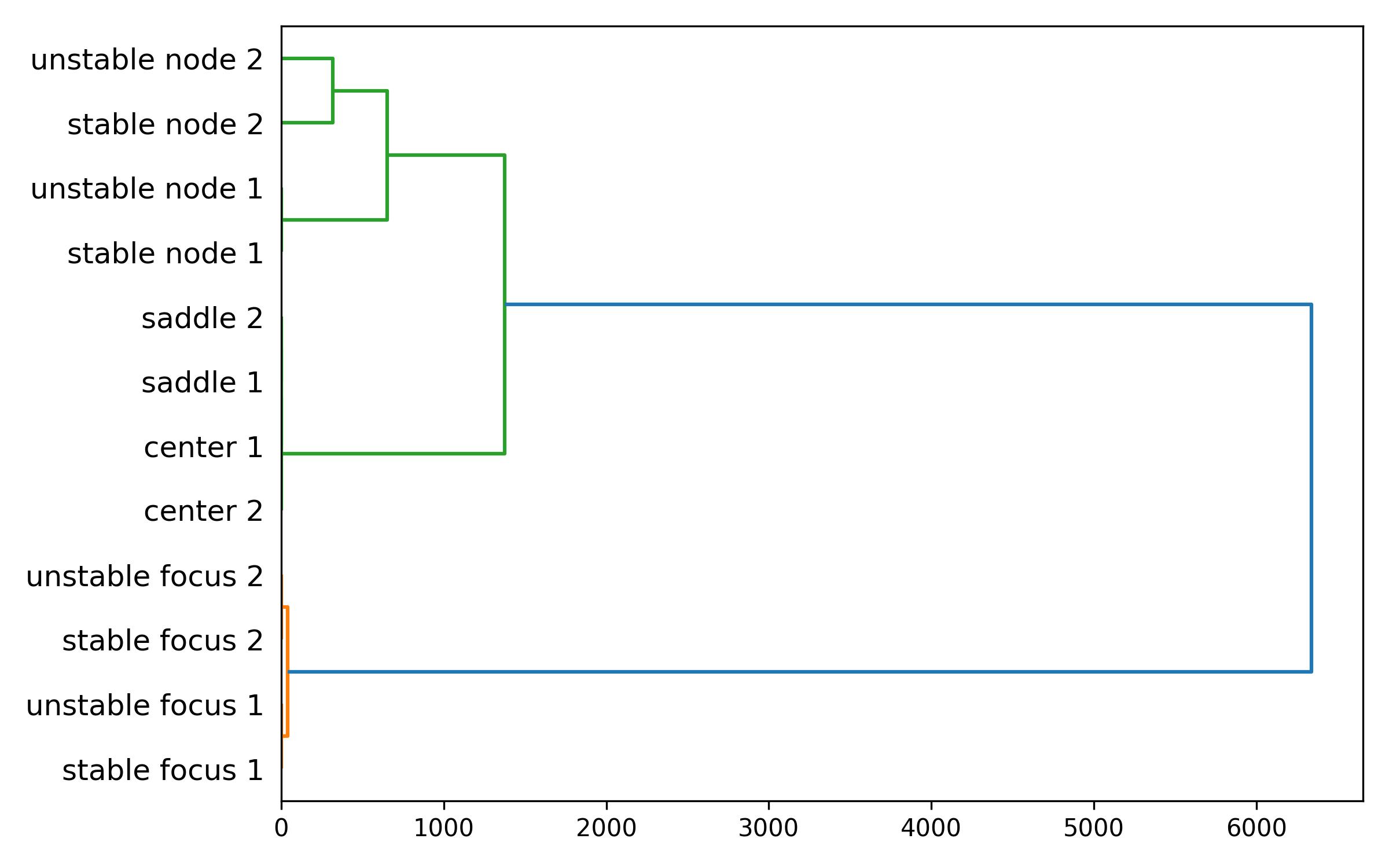}}
    \subfigure[]{\includegraphics[width=0.45\textwidth]{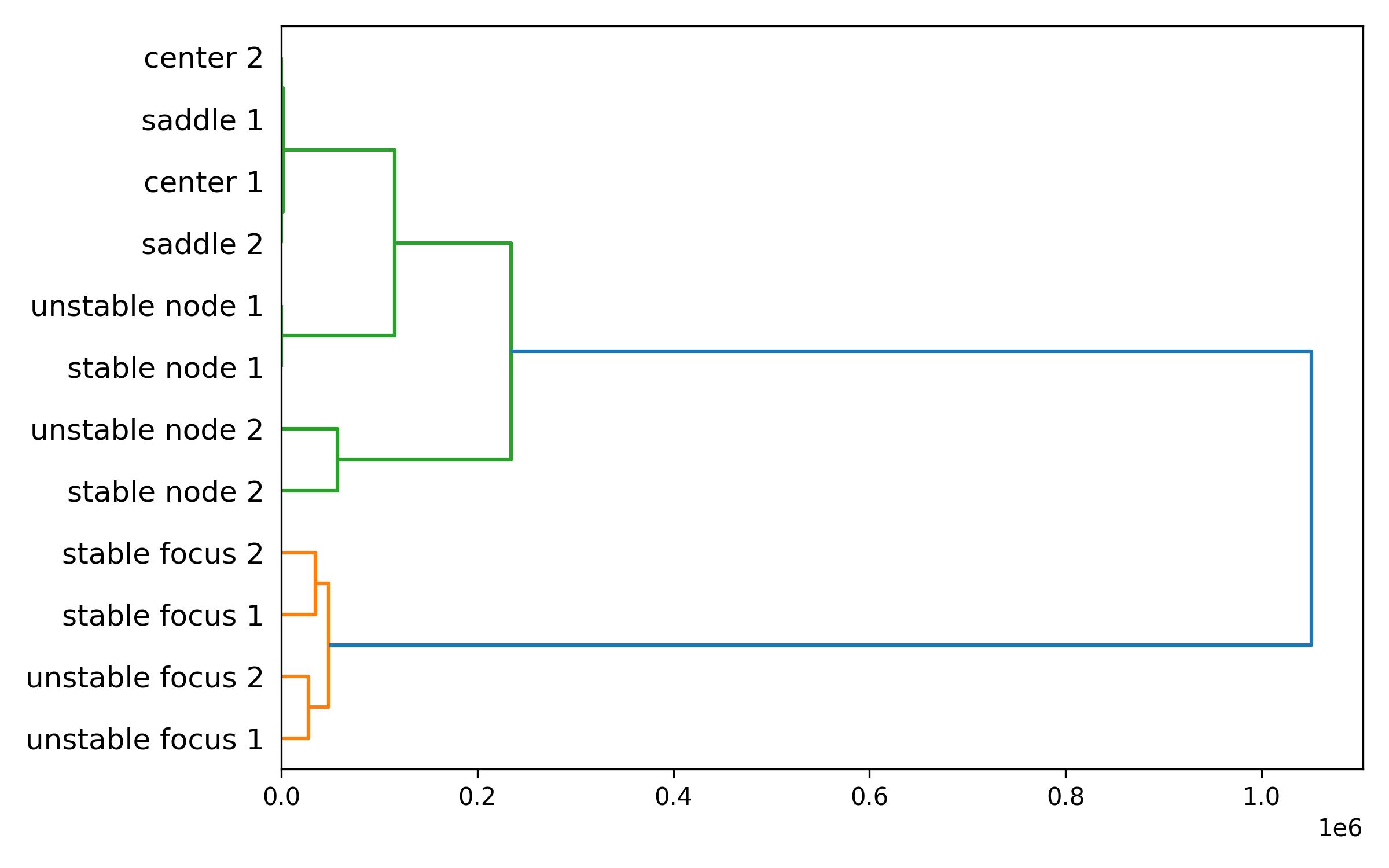}} \\
    \subfigure[]{\includegraphics[width=0.45\textwidth]{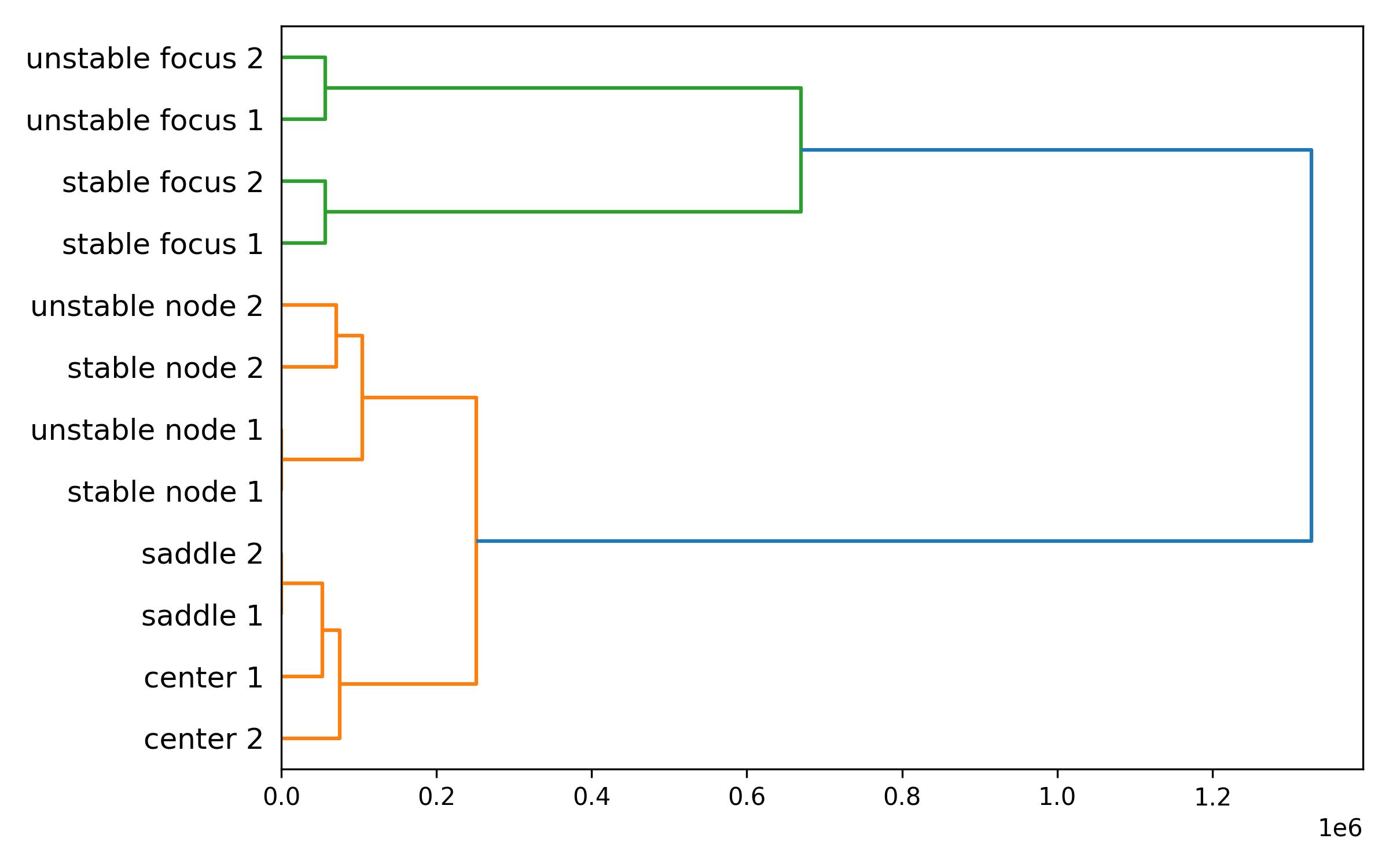}}
    \subfigure[]{\includegraphics[width=0.45\textwidth]{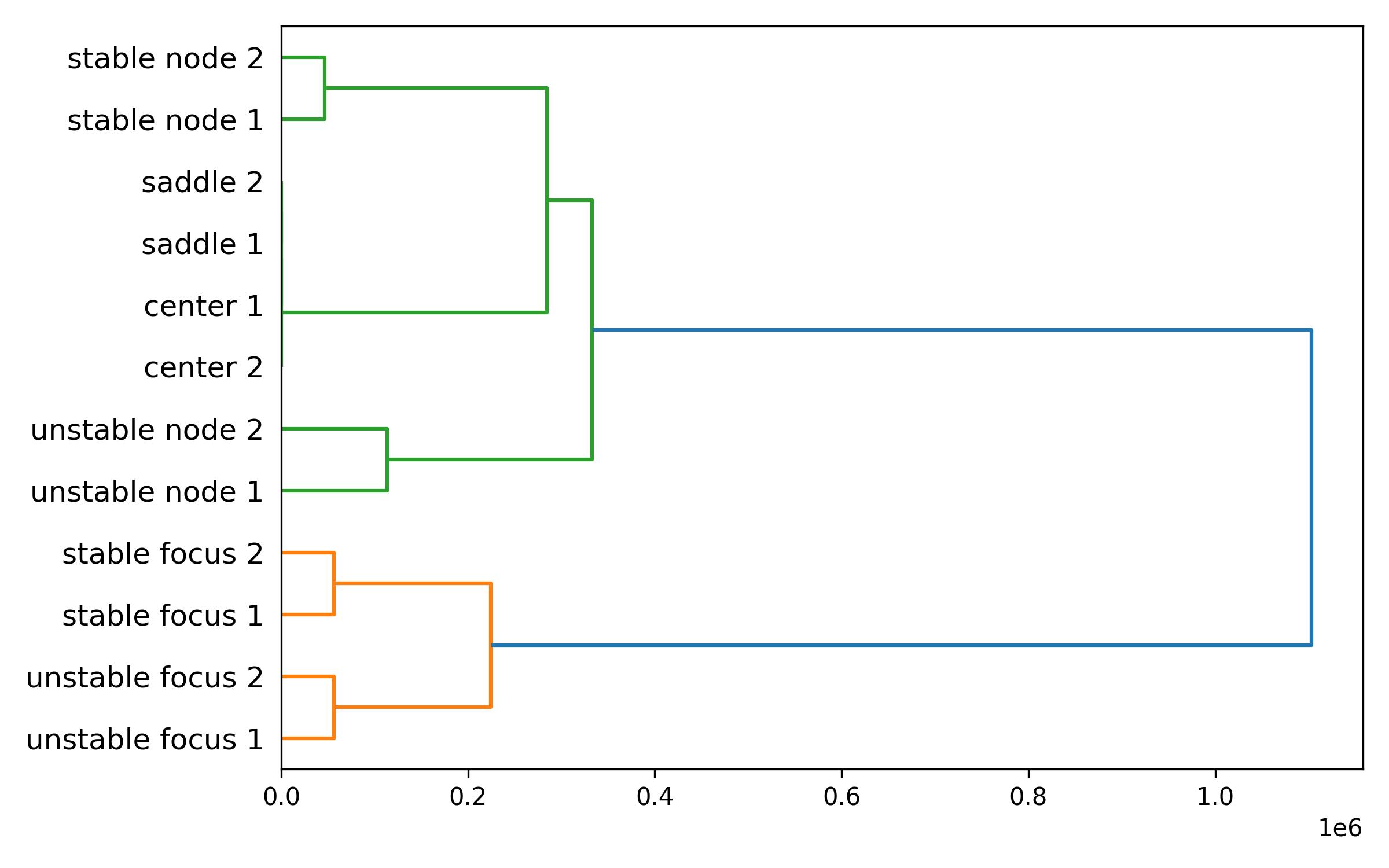}} \\
    \subfigure[]{\includegraphics[width=0.45\textwidth]{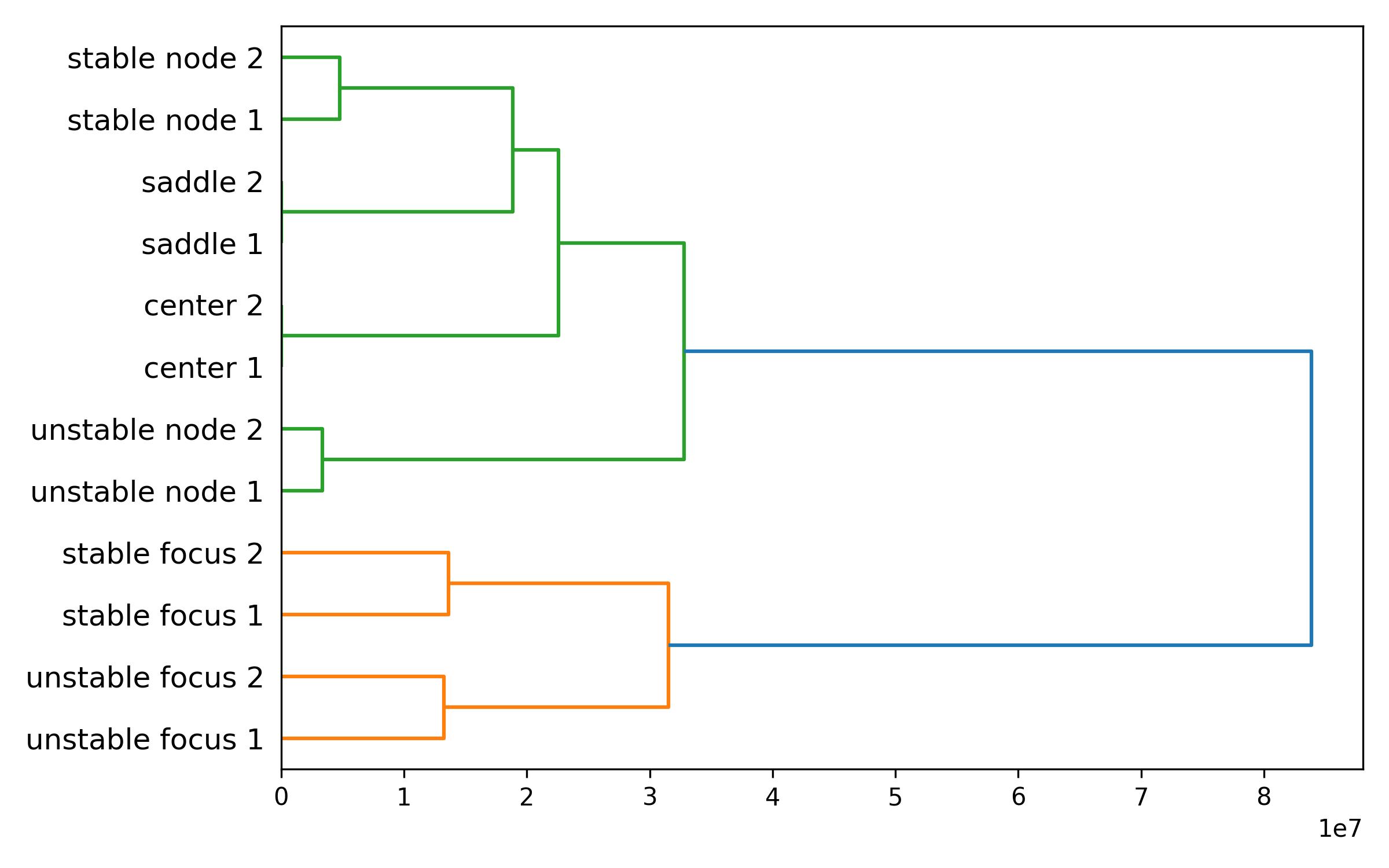}}
    \caption{Hierarchical clustering of two-dimensional linear autonomous systems obtained using Euler Characteristic Profiles (ECPs) computed from different multifiltration strategies: (a) vector field components, (b) vector field components with orientation angle, (c) vector field components with curl, (d) vector field components with divergence, and (e) eigenvalue-based multifiltration}
    \label{fig:2D_dendrograms}
\end{figure} 

\newpage
It should be noted that for multidimensional data with a regular grid, the ECP representation can be calculated more efficiently using the method described in Appendix~B.

For increased numerical accuracy, the computational grid has been refined to $101\times101$ points. 
The ECPs were computed using various multifiltration functions constructed from vector field features and the resulting dendrograms Fig. \ref{fig:2D_dendrograms}  were analyzed to assess the separability of the dynamical categories.

When the bifiltration was constructed directly from the vector field components, the clustering of ECP (shown in the dendrogram Fig. \ref{fig:2D_dendrograms} (a)) provides only a coarse separation of the systems. The primary distinction is between foci and all remaining equilibrium types, as foci are the only class characterized by complex eigenvalues. Beyond this, the clustering is mainly driven by the global magnitude and geometric similarity of the vector fields rather than by invariant dynamical properties. Consequently, stable and unstable systems of the same qualitative type remain mixed.

Using the orientation angle and the local rotation (curl) of the vectors as an auxiliary filtration produces the dendrograms shown in Fig.~\ref{fig:2D_dendrograms}(b) and Fig.~\ref{fig:2D_dendrograms}(c), respectively. Introducing an additional filtration already improves the discrimination of stable and unstable foci, indicating that these features provide complementary information beyond the raw vector values. However, the orientation angle alone does not further improve the separation of the remaining equilibrium types, since it reflects only the local direction of the flow. In contrast, the curl captures the local rotational structure, allowing saddle systems to be distinguished from centers, which are otherwise clustered together.

When the divergence (div) was added to the vector-based bifiltration, the resulting ECP dendrogram Fig. \ref{fig:2D_dendrograms}(d) revealed clustering according to the stability of equilibria. Divergence correlates with the sign of the trace of the Jacobian $\nabla \cdot \mathbf{F} = \mathrm{tr}(A)$, which enables separation between stable and unstable systems. Nevertheless, divergence alone does not encode the rotational structure of the flow and therefore cannot fully recover the qualitative classification.

Finally, when the multifiltration was defined using the eigenvalues of the Jacobian matrices, the ECP dendrogram Fig. \ref{fig:2D_dendrograms}(e) achieved a complete  separation of dynamical regimes. Stable and unstable systems, as well as the different equilibrium types, form distinct clusters. This result is consistent with the theoretical role of eigenvalues as local invariants of linear systems, directly encoding both the qualitative type of the equilibrium (real or complex eigenvalues) and its stability (sign of the real part). Hence, the eigenvalue-based filtration provides the most discriminative and physically meaningful representation among the tested variants.

Overall, the comparative analysis demonstrates that the choice of filtration determines which dynamical features are emphasized. No single local differential feature (orientation angle, curl or divergence) fully captures the topological classification of linear systems. Each proposed filtration answers a different question about the underlying dynamics and extracts complementary information. By combining multifiltrations, one can obtain a richer and more discriminative topological representation, which is particularly useful when systems exhibit subtle differences in flow direction or local stability characteristics. This approach enables a more nuanced comparison between linear systems and allows for the identification of features that would remain hidden if only a single scalar field were considered.

Our objective, however, is not to point out the best choice for (multi--)filtration when characterizing vector fields but indicate and examine various possibilities so that a choice for filtration can be made based on the important features from a point of view a particular classification problem.  

Note, that the preceding examples reveal a fundamental distinction between linear and nonlinear systems from the perspective of ECP. For a linear map is an intersection of affine half-spaces and is therefore convex whenever nonempty. Consequently, all nonempty sublevel sets are contractible, implying
\[
\chi(D_y)=1,
\]
while
\[
\chi(D_y)=0
\]
whenever $D_y=\varnothing$. The same conclusion remains valid after restricting the system to any convex bounded domain. Hence, for linear systems the ECP is necessarily binary and can attain only the values $\{0,1\}$.

This observation highlights that linear systems possess only trivial sublevel-set topology. Any richer behaviour of the ECP must therefore originate from nonlinear effects. Indeed, introducing a nonlinear perturbation may destroy convexity of the sublevel sets, leading to the creation of multiple connected components, holes, or more complicated topological structures. As a consequence, the Euler characteristic may attain a much broader range of values, including positive integers greater than one, negative values, and in certain unbounded settings even infinite values.

From this perspective, the ECP can be interpreted as a quantitative measure of topological complexity. Linear dynamics produce the simplest possible profile, whereas deviations from the binary pattern $\{0,1\}$ provide a direct topological signature of nonlinear phenomena. Therefore, the emergence of additional profile values may be viewed as an indicator of increasing nonlinear complexity in the underlying system.

\section{Properties of Euler Characteristic Profiles}
\label{sec:prop_ECP}
At the beginning, let us see if and how ECP changes when the vector field is transformed by rotation, rescaling, translation by a vector and perhaps other general transformations. In particular, with the knowledge on the impact of rotation on ECP, we can define \textbf{equivariant ECP}. i.e. a counterpart of Euler Characteristic Profile which is invariant with respect to the translation. 
\begin{lemma}
Let $ K $ be a finite cell complex equipped with a  multifiltration $f : K \to \mathbb{R}^n$.
Let $ \mathrm{ECP}(K, f) \subset \mathbb{R}^n \times \mathbb{Z} $ denote the Euler Characteristic Profile associated to $ f $, expressed as a collection of anchor points with contributions. Let $ p : \mathbb{R}^n \to \mathbb{R}^n $ be an transformation, i.e., a composition of:
\begin{enumerate}
    \item a scaling by a constant factor $ S > 0 $,
    \item a translation by a vector $ w \in \mathbb{R}^n $.
\end{enumerate}
Then, the ECP of the composition $ p \circ f : K \to \mathbb{R}^n $ is given by:
\[
\mathrm{ECP}(K, p \circ f) = \left\{ \left(p(a^i), \mathrm{val}^i\right) \right\}_{i=1}^k,
\]
where $ \{(a^i, \mathrm{val}^i)\}_{i=1}^k $ are the anchor points and corresponding contributions of $ \mathrm{ECP}(K, f) $.

\end{lemma}

\begin{proof}
Recall that for a multifiltration, the $\mathrm{ECP}(K,f)$ can be expressed as a finite sum of generalized step functions. Each such function is constant (with integer value) on the region
\[
\{x \in \mathbb{R}^n \mid x \geq a\},
\]
for some anchor point $a = (a_1,\dots,a_n) \in \mathbb{R}^n$ and vanishes below $a$.

The anchor points correspond precisely to the filtration values at which cells of $K$ enter the filtration induced by $f$. Now consider the transformed filtration $p \circ f$. Since $p$ is affine and order-preserving (as $S>0$), the relative ordering of filtration values is unchanged; only their locations in $\mathbb{R}^n$ are transformed by $p$.

Consequently, each cell of $K$ enters the filtration at the transformed value
$p(a)$, where $a$ was its original anchor point under $f$.
Thus, the generalized step function associated to $a$ in $\mathrm{ECP}(K,f)$
is replaced by a step function anchored at $p(a)$ in $\mathrm{ECP}(K,p \circ f)$, with the same integer coefficient.

Therefore, the collection of anchor points of $\mathrm{ECP}(K,p \circ f)$
is exactly the image under $p$ of the anchor points of $\mathrm{ECP}(K,f)$,
which proves the claim. \qed
\end{proof}

\begin{remark}[Multi--critical multifiltrations]
For a multi--critical multifiltration, that is a filtration for which a cell appear on a number of non--compatible points, the generalized step functions appearing in the ECP are no longer supported above a single anchor point, but rather above a finite set of anchor points
\[
a^{(1)}, \dots, a^{(k)} \in \mathbb{R}^n,
\]
meaning the function takes value $1$ on points dominating at least one of these anchors. Applying an affine transformation $p$ to the parameter space simultaneously transforms all anchor points $a^{(i)}$ to $p(a^{(i)})$. The same argument as above shows that
\[
\mathrm{ECP}(K, p \circ f)
\]
is obtained by applying $p$ to all anchor points, and the proof carries over  without modification.
\end{remark}

\begin{remark}[Rotations]
In contrast to scalings and translations, ECP is not equivariant with respect to rotations of the parameter space. This failure stems from the fact that rotations do not preserve the coordinate-wise partial order  on $\mathbb{R}^n$ that underlies the definition of multifiltrations.

Let us illustrate the ECP on a simple linear dynamical system. Consider the planar vector field
\begin{equation}
\begin{aligned}
\dot{x} &= ax+by,\\
\dot{y} &= cx+dy,
\end{aligned}
\end{equation}
with parameters $a=-1.5, b=0, c=0, d=-2$ so that the origin is a stable node. The vector field is sampled on a regular grid $[x_{\min},x_{\max}] \times [y_{\min},y_{\max}]$ with domain $[-2,2]\times[-2,2]$.  At each grid point $(x_i,y_j)$ we evaluate the vector $f(x_i,y_j) = (f(x_i,y_j), g(x_i,y_j))$. This vector values are treated as filtration parameters, which defines a bifiltration on the set of two–dimensional cells. Following the standard construction the filtration value of lower-dimensional cells (edges and vertices) is defined as the componentwise minimum of the values of incident higher–dimensional cells. This procedure produces a 1-critical bifiltration of the cubical complex associated with the grid. The resulting filtered complex is illustrated schematically in Fig. \ref{fig:rotation}(a). The ECP computed for this filtration is shown in Fig. \ref{fig:rotation}(c).

Next we consider a rotated version of the same vector field.  For a rotation angle $\theta$ we define the rotation matrix
\[
R =
\begin{pmatrix}
\cos\theta & -\sin\theta \\
\sin\theta & \cos\theta
\end{pmatrix}.
\]
In our example we take $\theta=\pi/4$. Applying this transformation to the vector values produces a new field $\tilde f(x,y) = R f(x,y)$. The rotated vectors again define a bifiltration on the same grid by assigning the values
$(\tilde f(x_i,y_j),\tilde g(x_i,y_j))$ to the two–dimensional cells. As before, filtration values of edges and vertices are defined using the componentwise minimum. The resulting filtered complex is illustrated in Fig. \ref{fig:rotation}(b) and the corresponding ECP is shown in Fig. \ref{fig:rotation}(d).

A natural question is whether the ECP behaves covariantly under rotations of the vector field.  To test this, we rotate the original profile by the same angle $\theta=\pi/4$ in the parameter space.

However, the rotated profile does not coincide with the profile computed from the rotated vector field. In other words,
\[
\mathrm{ECP}_{R f} \neq R\big(\mathrm{ECP}_f\big).
\]
This discrepancy arises from the fact that the construction of the multifiltration depends on the coordinate axes of the filtration space. While the vector field itself is rotated linearly, the componentwise ordering used to define the bifiltration is not invariant under such transformations.

Consequently, ECPs derived from componentwise filtrations of vector fields are generally not rotationally invariant.
\begin{figure}[h]
    \centering
    \subfigure[]{\includegraphics[width=0.3\textwidth]{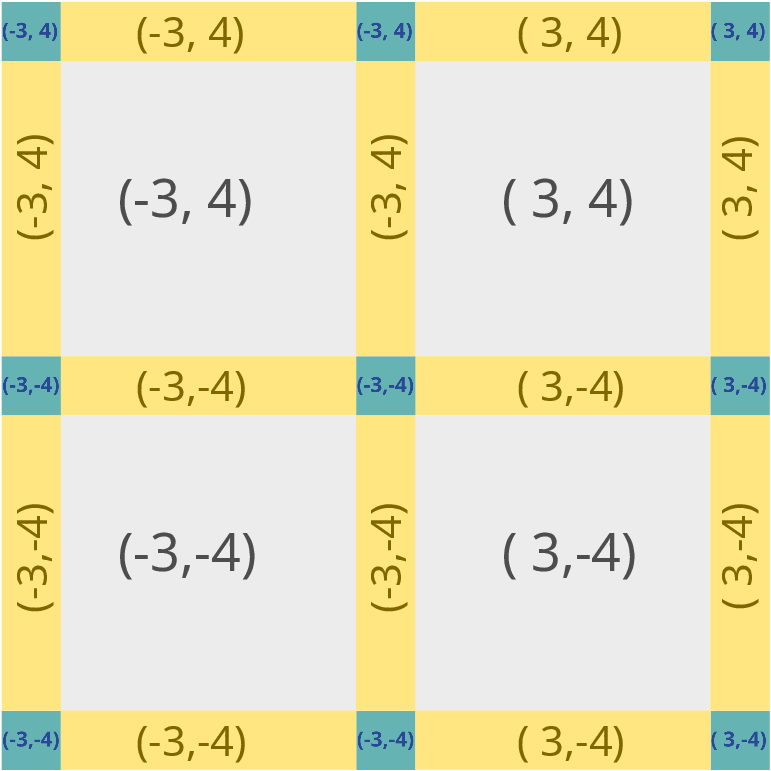}}
    \hspace{2.5cm}
    \subfigure[]{\includegraphics[width=0.3\textwidth]{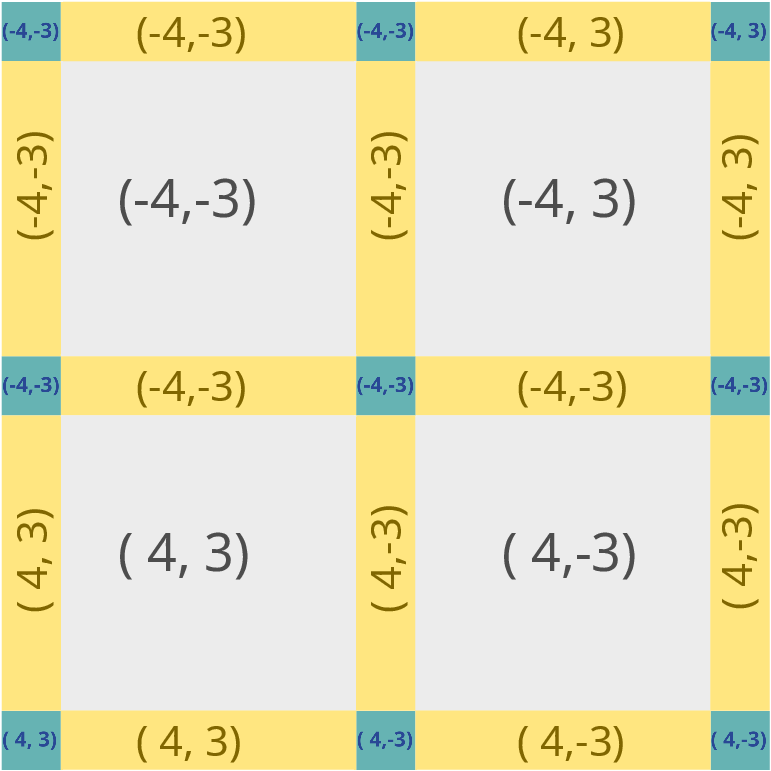}} \\
    \subfigure[]{\includegraphics[width=0.4\textwidth]{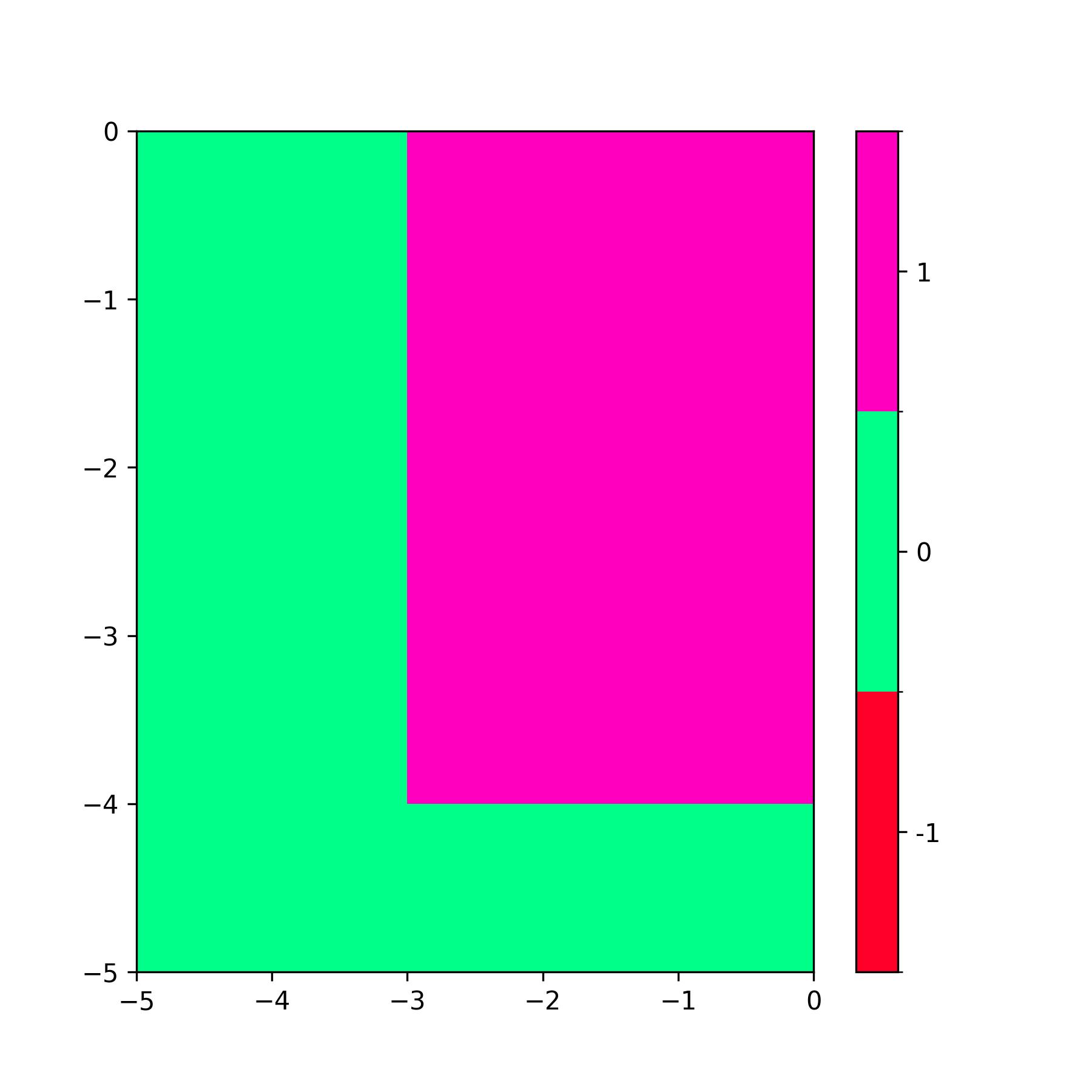}}
    \hspace{0.8cm}
    \subfigure[]{\includegraphics[width=0.4\textwidth]{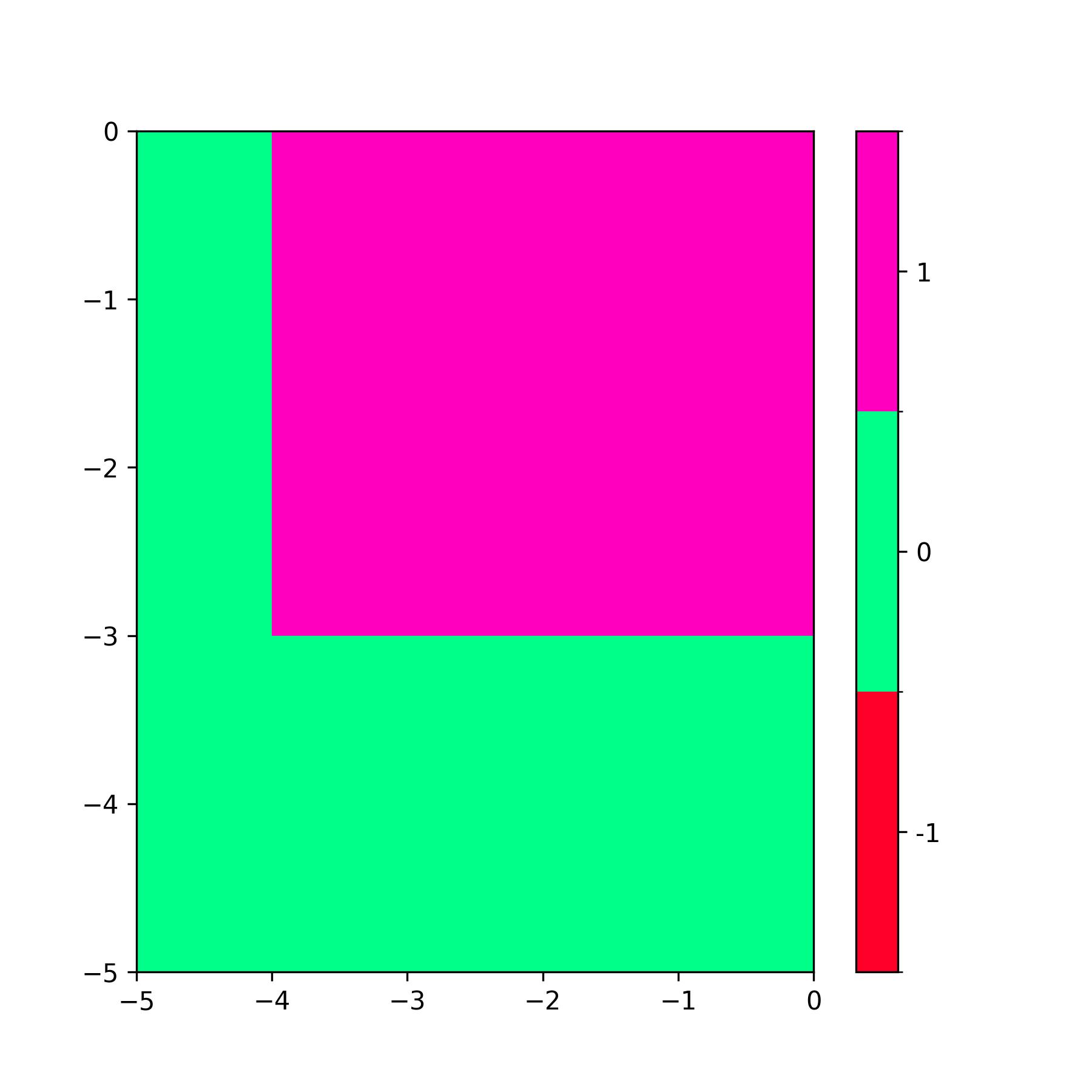}}
    \caption{Illustration of rotational invariance for Euler Characteristic Profiles (ECPs). Panels (a) - (b) present a cubical complex equipped with the original bifiltration and the corresponding bifiltration after a $\pi/4$ rotation, respectively. Panels (c) - (d) show the associated ECPs, illustrating the effect of geometric transformations on the proposed descriptor}
    \label{fig:rotation}
\end{figure}
\end{remark}

\subsection{ECP under Dynamical Equivalence and Related Notions}

There are various notions of classifying dynamical similarity of continuous dynamical systems. Let us below recall a couple of definitions (see e.g. \cite{kuznetsov}).

In below we assume that $X\subset \mathbb{R}^n$ and $Y\subset \mathbb{R}^n$ are topological spaces and $f: X\to \mathbb{R}^n$ and $g: Y\to \mathbb{R}^n$ are vector fields which induce flows $\Phi$ on $X$ and $\Psi$ on $Y$, respectively, i.e. $\dot{x}=f(x)$ and $\dot{y}=g(y)$.  

\begin{definition}(Topological equivalence of flows) We say that two flows $\Phi: X\times\mathbb{R}\to X$ and $\Psi: Y\times \mathbb{R}\to Y$ are \emph{topologically equivalent} if there exists a homeomorphism $h: X\to Y$ such that the following two conditions are satisfied: 
\begin{enumerate}
\item for every $x\in X$  we have \[ h(\mathcal{O}(x,\Phi))=\mathcal{O}(h(x),\Psi),\] 
where $\mathcal{O}(x,\Phi):=\{\Phi(x,t): \ t\in\mathbb{R}\}$ and $\mathcal{O}(h(x),\Psi):=\{\Psi(h(x),t): \ t\in \mathbb{R}\}$ denote the orbits of $x\in X$ and $h(x)\in Y$, respectively, 
\item for every $x\in X$, there exists
$\delta_x>0$ such that, if
\[
0<\vert s \vert < t<\delta_x,
\qquad\text{and}\qquad
\Psi(h(x),s)=h(\Phi(x,t)),
\]
then $s>0$. 
\end{enumerate}

\end{definition}
Thus topological equivalence means that the orbits of $\Phi$ are mapped homeomorphically onto the corresponding orbits of $\Psi$ with preserving the direction of time. 

Next two definitions refer more directly to the vector fields inducing the flows and are special cases of topological equivalence which can be described analytically:

\begin{definition}(Smooth equivalence) Flows $\Phi$ and $\Psi$ defined, respectively, by the autonomous ODEs $\dot{x}=f(x)$ and $\dot{y}=g(y)$,
are \emph{smoothly equivalent} if there exists a diffeomorphism  $h:X\to Y$ 
such that
\begin{equation}\label{eq:smoothequiv}
f(x)=M^{-1}(x)g(h(x)), 
\end{equation}
where $M:=D_h(x)$ is the Jacobian matrix of $h$ evaluated at $x$. 
\end{definition}
Dynamical systems  which are smooth equivalent are related by the coordinate transformation $y=h(x)$ and are also called \emph{diffeomorphic}. 

\begin{definition}(Orbital equivalence)\label{def:orbitalequiv}
Suppose that $\mu = \mu(x)>0$ is a smooth scalar positive function and that
\begin{equation}\label{eq:orbitalequiv}
f(x) = \mu(x)g(x).
\end{equation}

Then the systems $\dot{x}=f(x)$ and $\dot{x}=g(x)$ are called \emph{orbitally equivalent}.
\end{definition}

Systems that are smoothly equivalent or orbitally equivalent are also topologically equivalent but reverse is not true (for example, two planar linear systems $\dot{x}=Ax$ and $\dot{y}=By$, one featuring a stable focus and the other one a stable node at $(0,0)$ are neither orbitally nor smoothly equivalent but are topologically equivalent in any closed disc centered at $(0,0)$). On the other hand, systems which are orbitally equivalent only differ by velocity along the orbits, which is altered by multiplying by the factor $\mu(x)$ (note that in particular, both such systems need to act on the same phase space).  In this case the homeomorphism $h$ in Definition \ref{def:orbitalequiv} is the identity ($h(x)=x$). 

The next simple result shows the relation between the continuous ECP and smoothly equivalent systems in 1D (in which case we denote continuous ECP as $cECC$):

\begin{lemma}\label{lem:ECPforsmoothequiv}
Let $f:I_1\to \mathbb{R}$ and $g:I_2\to\mathbb{R}$, where $I_1\subset\mathbb{R}$  and $I_2\subset\mathbb{R}$ are closed intervals. Suppose that the systems $\dot{x}=f(x)$, $x\in I_1$ and $\dot{y}=g(y)$, $y\in I_2$, are smoothly equivalent and $h: I_1\to I_2$ is the connecting diffeomorphism as in \eqref{eq:smoothequiv}.

Suppose that $h: I_1\to I_2$ is increasing.  Then
\begin{equation}\label{eq:ECPforsmoothequiv1}
cECC_{g}(0) = cECC_{f}(0). 
\end{equation}

On the other hand, if $h: I_1\to I_2$ is decreasing and the set $\mathrm{Zero}_f$ of zeros of $f$ is finite, then:
\begin{equation}\label{eq:ECPforsmoothequiv2}
cECC_{g}(0)=1+\vert \mathrm{Zero}_f \vert - cECC_{f}(0),
\end{equation}
where $\vert \mathrm{Zero}_f  \vert$ denotes the cardinality of $\mathrm{Zero}_f$.
  
\end{lemma}
\begin{remark}
    Note that, as $f$ and $g$ are connected by the diffeomorphism $h$ according to \eqref{eq:smoothequiv}, finite set of zeros of $f$ implies that the set of zeros of $g$ is also finite. In the increasing case, the assumption on the finiteness of $\mathrm{Zero}_f$ is not needed whereas in the decreasing case it can replaced by more general assumption that the sets $\mathrm{Zero}_f$  and $Z_1:=f^{-1}((-\infty,0])$  have well-defined finite Euler characteristic, in which case \eqref{eq:ECPforsmoothequiv2} takes the form
 \begin{equation}\label{eq:ECPforsmoothequiv3}
cECC_{g}(0)=1+\chi(\mathrm{Zero}_f)  - cECC_{f}(0).
\end{equation}   
\end{remark}
 
\begin{proof}
Firstly, note that if $h: I_1 \to I_2$ is the diffeomorphism providing smooth equivalence between the two systems, then  $h$ is either strictly increasing (in which case $M(x)>0$ for every $x\in I_1$) or strictly decreasing (and $M(x)<0$ for every $x\in I_1$). Define $Z_1:=f^{-1}((-\infty, 0 ])$ and $Z_2:=g^{-1}((-\infty, 0 ])$.

Suppose that $h$ is strictly increasing.  Then $f(x)\leq 0$ is equivalent to $M(x)f(x)\leq 0$ and from \eqref{eq:smoothequiv} we conclude that 
\[
Z_1:=\{ x\in I_1: \ g(h(x))\leq 0 \}=\{x= h^{-1}(y)\in I_1: \ g(y)\leq 0\}= h^{-1}(Z_2).
\]
Consequently, $\chi(Z_1)=\chi(h^{-1}(Z_2))=\chi(Z_2)$, as $h$ preserves Euler characteristic of sets which proves \eqref{eq:ECPforsmoothequiv1}. 

Now, suppose that $h: I_1\to I_2$ is decreasing. Note that $y\in Z_2$  if and only if there exists a unique $x\in I_1$ such that $y=h(x)$ and $f(x)\geq 0$. Thus
\[
Z_2=h((I_1\setminus Z_1)\cup \mathrm{Zero}_f).
\]
Consequently,
\[
cECC_{g}(0)=\chi(Z_2)=\chi(h((I_1\setminus Z_1)\cup \mathrm{Zero}_f))=\chi((I_1\setminus Z_1)\cup \mathrm{Zero}_f).
\]
Note that $(I_1\setminus Z_1)\cup \mathrm{Zero}_f=\{x\in I_1: \ f(x)\geq 0\}$. Since $\{f(x)\geq 0\}\cup Z_1=I_1$ and $\{f(x)\geq 0\}\cap Z_1=\mathrm{Zero}_f$, by inclusion-exclusion property of the Euler characteristic we have:
\[
\chi(I_1)=\chi(\{x\in I_1: \ f(x)\geq 0\})+\chi(Z_1)-\chi(\mathrm{Zero}_f).
\]
This proves \eqref{eq:ECPforsmoothequiv2}, as $\chi(I_1)=1$ and $\chi(\mathrm{Zero}_f)=\vert \mathrm{Zero}_f \vert$. 
\qed
\end{proof}

\section{Numerical Examples: ECP for Nonlinear and Higher Dimensional Systems}
\label{sec:ECP_nonlinear}
Nonlinear systems are more complex and diverse in their behavior compared to linear systems. These systems can exhibit a wide range of phenomena, including bifurcations, chaos and multi--stable changes ~\cite{ott2002}. In this section, we delve into more intricate examples, where the application of ECP is crucial for understanding the system's behavior under various conditions.

\subsection{Hopf Bifurcation}
The Hopf bifurcation is a well-known nonlinear phenomenon that occurs when a system undergoes a transition from a stable equilibrium to a limit cycle. This bifurcation is central to the study of nonlinear dynamics and is often observed in biological, chemical and mechanical systems. Let us start with a simple experiment on the system exhibiting Hopf bifurcation given by the equation: 
\begin{equation}\label{eq:Hopf}
\begin{aligned}
  \dot{x} &=\beta x - y - x (x^2+y^2)  \\
  \dot{y} &=x +\beta y - y (x^2+y^2)
\end{aligned}
\end{equation}
where for the following values of $\beta$ parameters we have:

\begin{itemize}[label=\textbullet]
    \item$\beta<0$: $(0,0)$ a stable focus; 
    \item$\beta=0$: supercritical Hopf bifurcation (stable fixed point $(0,0)$ loses its stability and stable limit cycle is born);
    \item $\beta>0$: $(0,0)$ unstable focus and there is also unique and attracting limit cycle around $(0,0)$ with radius $\sqrt{\beta}$. 
\end{itemize}

We examined the systems for $\beta$ parameters in the range $[-1.0,1.0]$. An example phase portrait for $\beta=1.0$ and the corresponding ECP are illustrated in Fig. \ref{fig:HB_vf_ECP}. The calculations were performed for ECP and various $Lp$ metrics, and the resulting dendrograms can be seen in Fig. \ref{fig:dendrograms}. The performance of ECP to detect changes in dynamics were compared with methods of Conley Index where we check for which value of the parameter $\beta$ the Conley index change. Sample Morse sets for $\beta=-1.0$ and $\beta =1.0$ are shown in Fig. \ref{fig:morse_sets}.

\begin{figure}[h]
    \centering
    \subfigure[]{\includegraphics[width=0.47\textwidth]{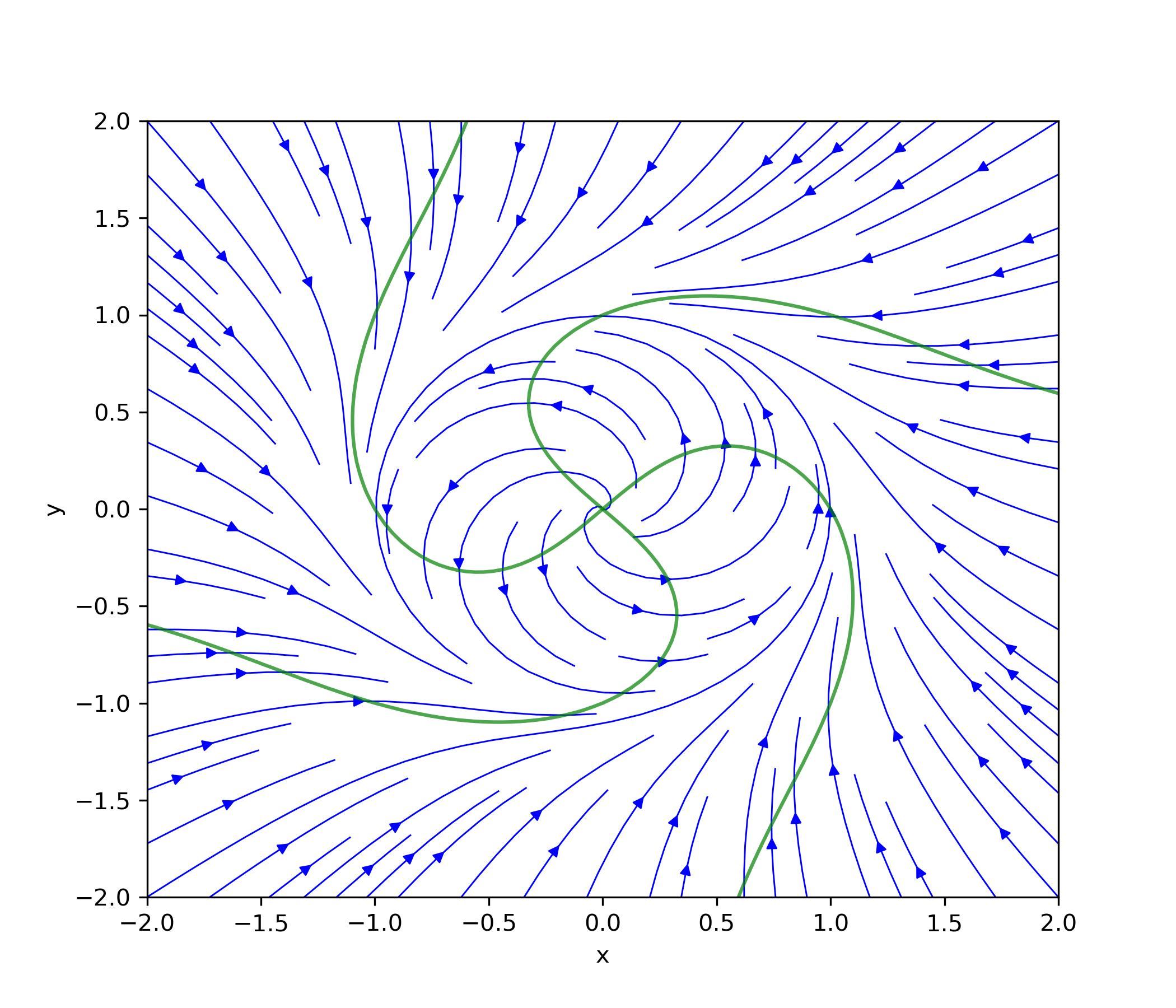}}
    \subfigure[]{\includegraphics[width=0.47\textwidth]{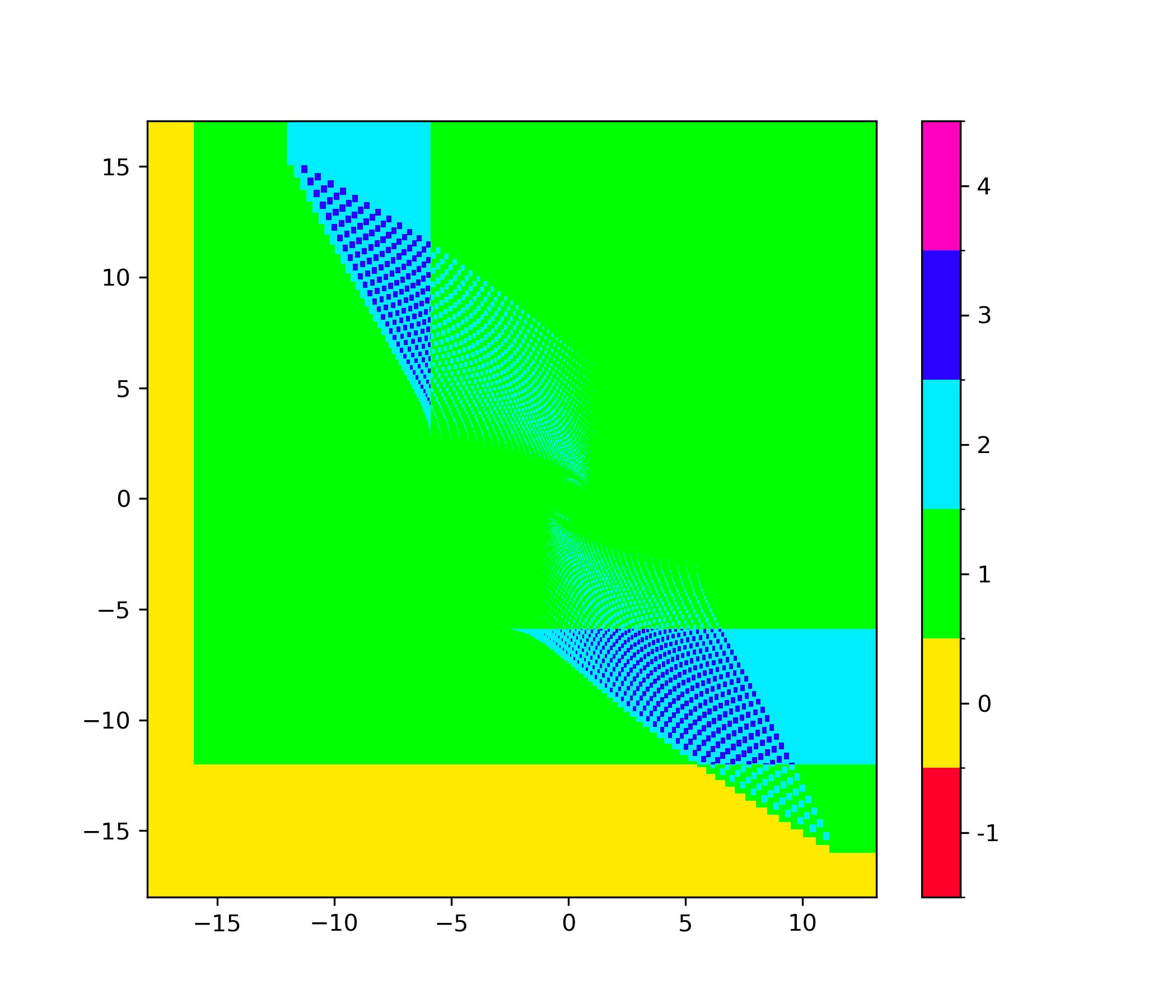}}
    \caption{Representative example of the Hopf bifurcation for $\beta=1$: (a) phase portrait illustrating the stable limit cycle and (b) the corresponding Euler Characteristic Profile (ECP) computed from the associated vector field}
    \label{fig:HB_vf_ECP}
\end{figure}

\begin{figure}[h]
    \centering
    \subfigure[]{\includegraphics[width=0.35\textwidth]{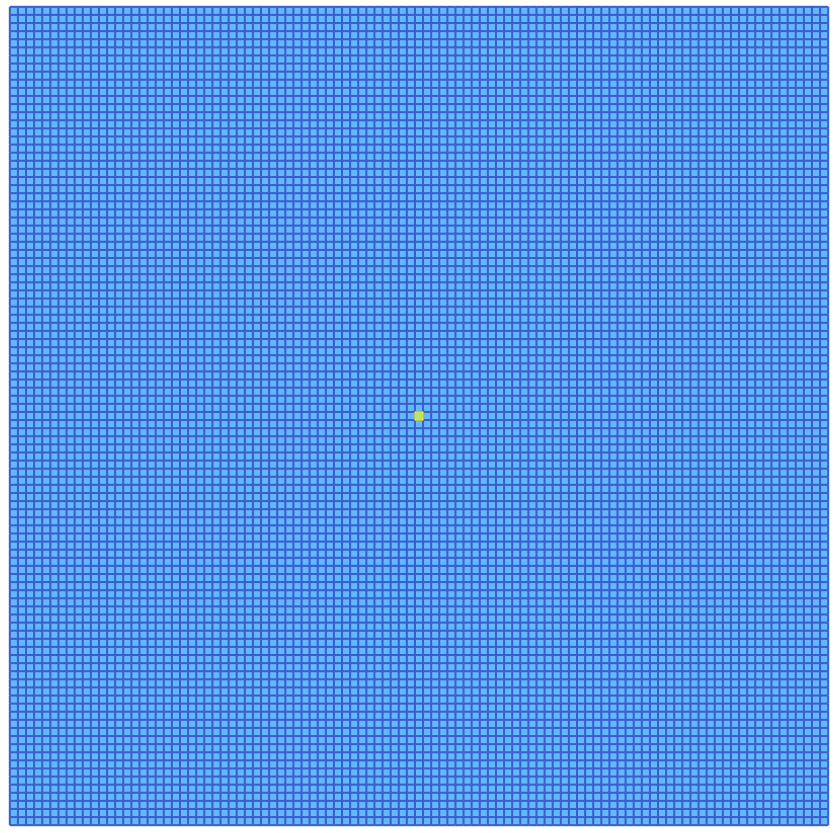}}
    \hspace{1cm}
    \subfigure[]{\includegraphics[width=0.35\textwidth]{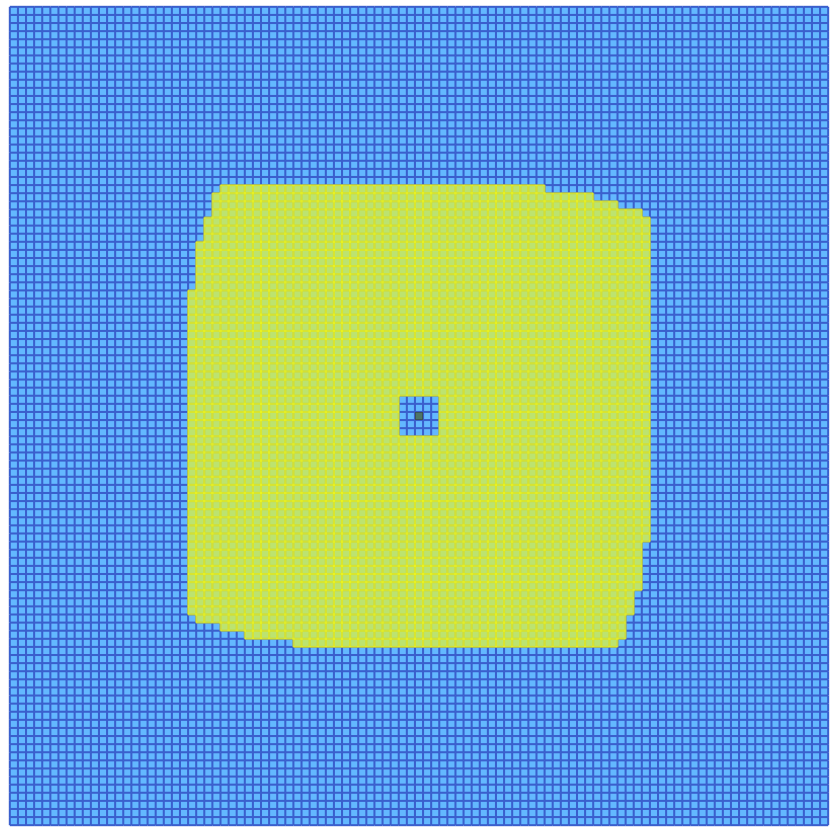}}
    \caption{Morse decompositions computed for the Hopf bifurcation system: (a) $\beta = -1$, corresponding to a stable equilibrium, and (b) $\beta = 1$, illustrating the emergence of a stable limit cycle after the bifurcation}
    \label{fig:morse_sets}
\end{figure}

\begin{figure}[h]
    \centering
    \subfigure[]{\includegraphics[width=0.46\textwidth]{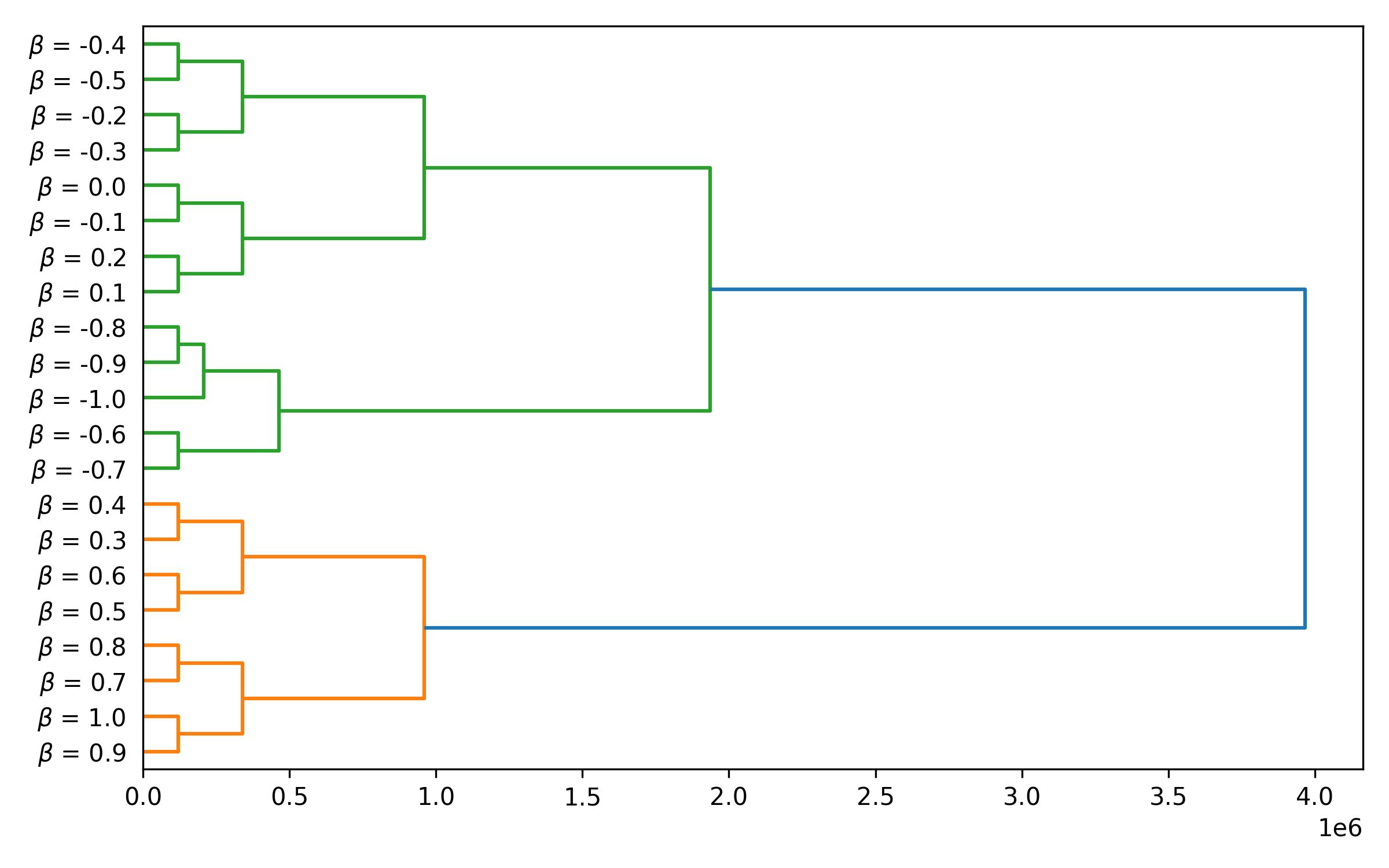}}
    \subfigure[]{\includegraphics[width=0.46\textwidth]{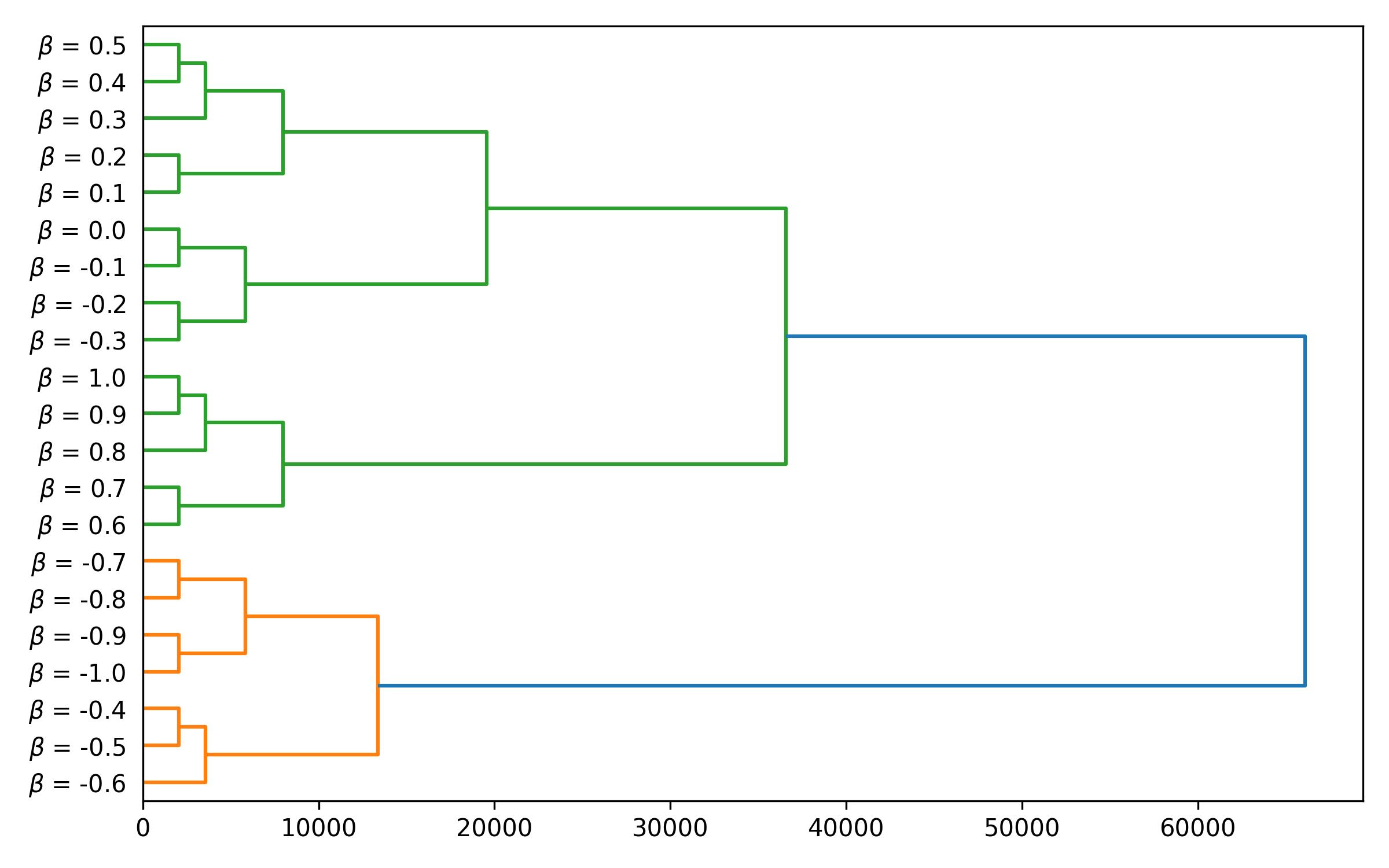}}\\
    \subfigure[]{\includegraphics[width=0.46\textwidth]{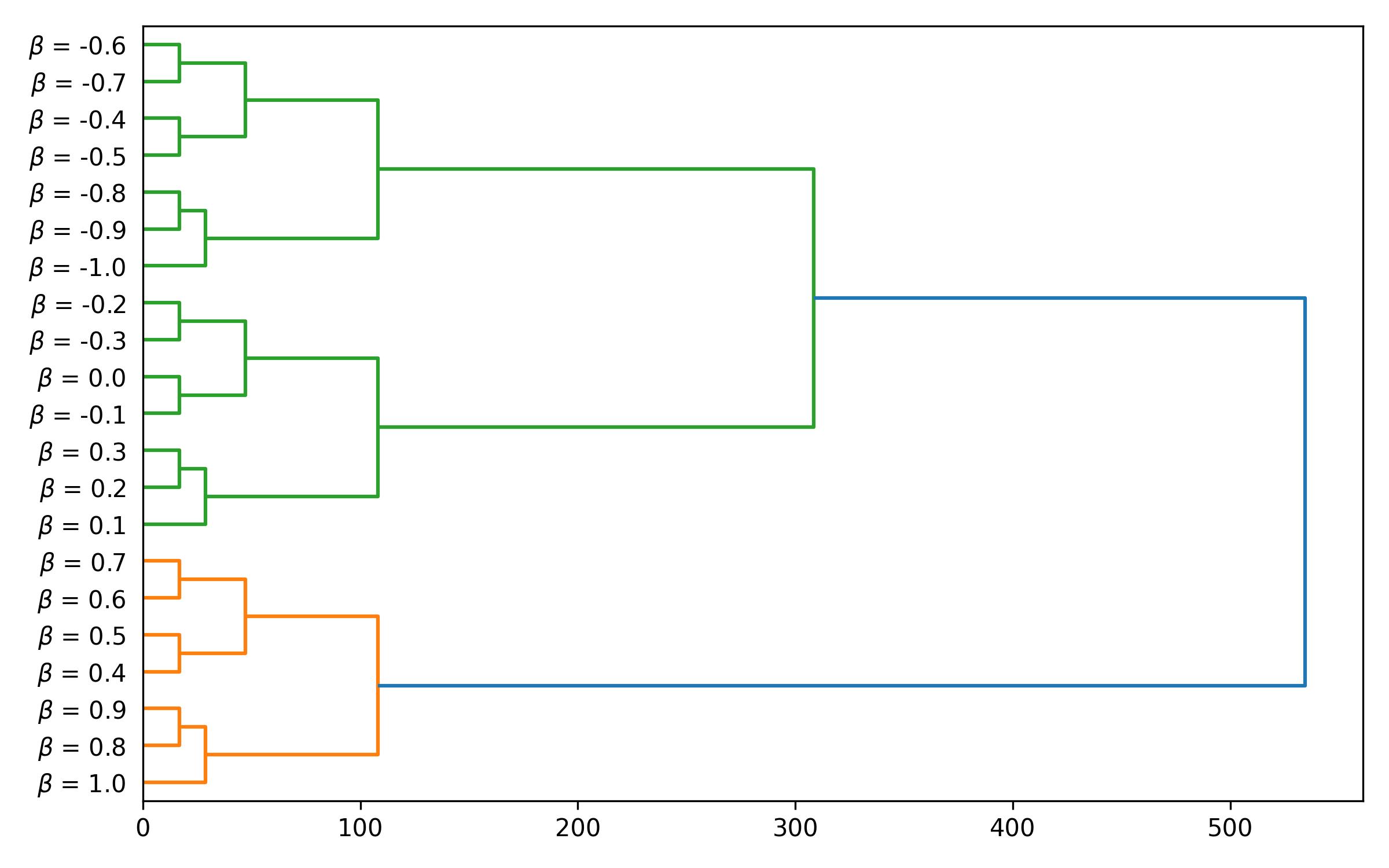}} 
    \subfigure[]{\includegraphics[width=0.46\textwidth]{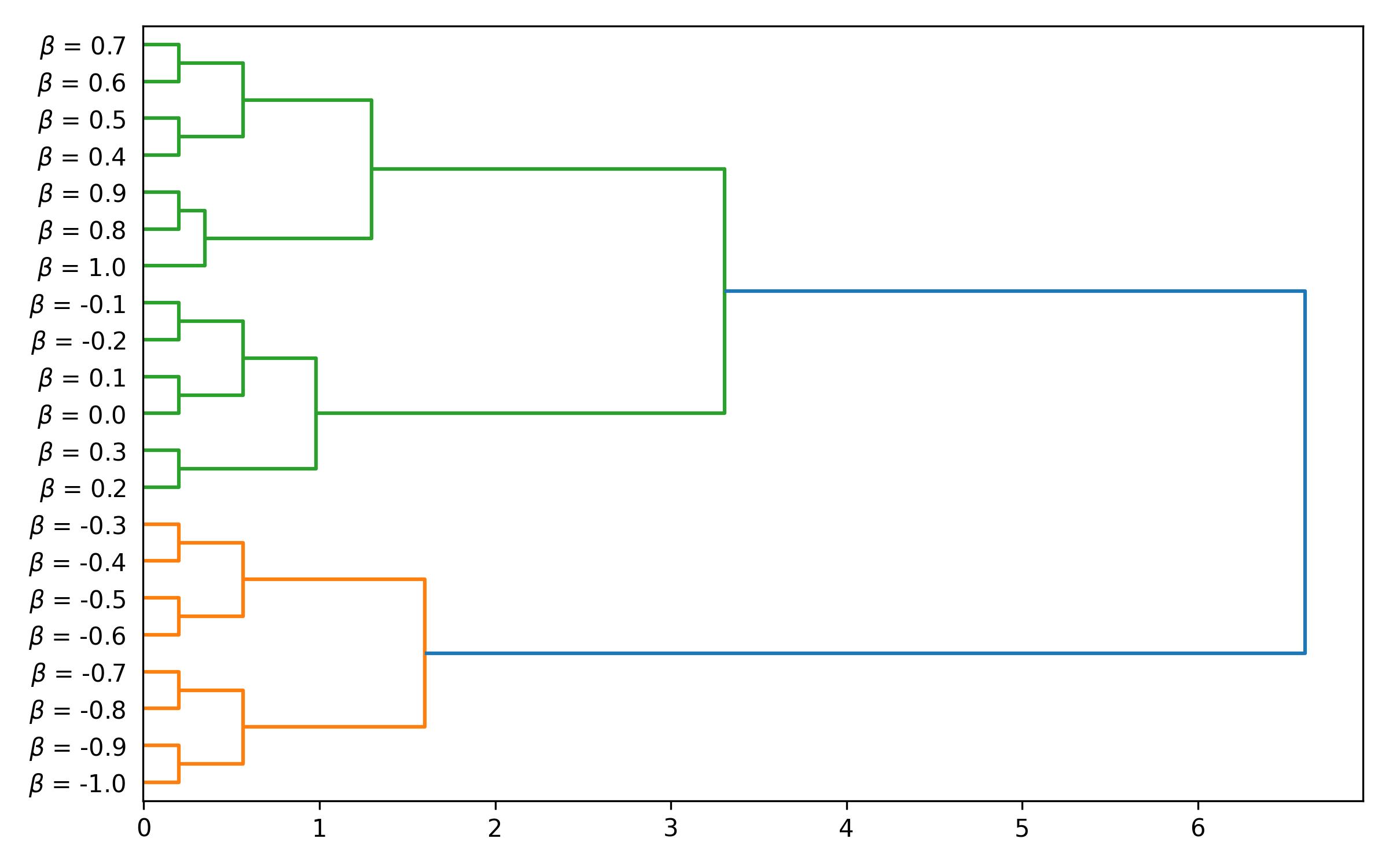}}
    \caption{Hierarchical clustering of Hopf bifurcation systems obtained using different similarity measures: (a) Euler Characteristic Profile (ECP) distance, (b) \(L_1\) distance, (c) \(L_2\) distance, and (d) \(L_{sup}\) distance. The dendrograms illustrate the ability of the proposed topological descriptors to distinguish different dynamical regimes}
    \label{fig:dendrograms}
\end{figure}

For this vector fields, the Conley index \cite{julia_code} remained constant with values $[1,0,0]$ for all parameters except $\beta=1.0$ (Fig. \ref{fig:morse_sets} (b)), where it shifted to a composite structure $[[1,1,0],[0,0,1]]$. This reflects the emergence of a limit cycle and the coexistence of multiple invariant sets at the bifurcation point. However, this index change occurs only when the system undergoes a distinct topological reorganization; it does not capture gradual or quantitative transitions in the vector field.

To obtain a finer resolution of the parameter space, the $\beta$ parameter values were further sampled densely in the range from $-1$ to $1$ with an increment of $0.1$. Even under this refined sampling, the Conley index remained constant for almost all values and exhibited the doubled structure only for $\beta=0.9$ and $\beta=1.0$. This delayed index change is notable, since the Hopf bifurcation theoretically occurs near $\beta=0$. 

In this particular discretization and computational setup, the Conley-index pipeline changed only for larger positive values of $\beta$, whereas the ECP-based clustering separated the sampled vector fields closer to the parameter range where the Hopf transition is expected.
Therefore, these results demonstrate that the Conley index, while robust in detecting qualitative topological transitions such as the appearance of a limit cycle, lacks sensitivity to early-stage or quantitative bifurcation phenomena.

For comparison, we examined the simplest data comparison, which is $L_p$ distances between ECPs and the corresponding dendrograms are shown in Fig. \ref{fig:dendrograms}. For ECP two distinct clusters are identified: the first spanning the range $b \in [-1, 0.2]$ and the second covering $b \in [0.3, 1]$. $L_1$ metric produces slightly shifted partitions, with clusters observed in the intervals $b \in [-1, -0.4]$ and $b \in [-0.3, 1]$. For $L_2$ metric, the dendrogram indicates clustering over $b \in [-1, 0.3]$ and $b \in [0.4, 1]$. Finally, $L_{sup}$ metric delineates two clusters in the ranges $b \in [-1, -0.3]$ and $b \in [0.2, 1]$. All dendrograms demonstrates that the system exhibits two main regimes of dynamical behavior, with the boundaries between clusters shifting slightly depending on the specific clustering method employed. This variation suggests that while the general bifurcation structure remains stable, the sensitivity of the clustering outcome reflects subtle differences in the underlying data representation and linkage criteria. 

Among the tested methods, the ECP approach yields the most consistent and physically meaningful clustering. The separation between $b \in [-1, 0.2]$ and $b \in [0.4, 1]$ corresponds closely to the expected bifurcation threshold and captures the transition between qualitatively different dynamical regimes. Furthermore, this method provides clear and reproducible cluster boundaries, minimizing overlap and ambiguity between the regions. 

The ECP, the Conley Index and the \(L_p\) norms represent distinct yet complementary approaches to the quantitative and qualitative characterization of dynamical systems. While the first two originate from algebraic topology and focus on invariant structures and qualitative changes, the \(L_p\) measure geometric deviations in a purely metric sense.

The ECP provides a low-dimensional yet informative topological summary of data by tracking changes in the Euler characteristic as a function of a filtration parameter. It efficiently captures structural transitions within scalar or vector fields and is computationally inexpensive. Moreover, the ECP exhibits moderate robustness to stochastic perturbations: while local fluctuations may distort the curve, global topological trends remain preserved. Its main limitation lies in the aggregation of homological information into a single scalar descriptor, resulting in loss of fine structural distinctions and the absence of explicit dynamical directionality.

The Conley Index, in contrast, offers a rigorous algebraic invariant associated with isolated invariant sets. It detects bifurcations and qualitative changes in the topology of attractors or repellers and remains invariant under continuous perturbations of the vector field. This topological stability makes it particularly valuable in theoretical studies of nonlinear dynamics. However, the method is computationally demanding, requiring the explicit construction of isolating neighborhoods and Morse decompositions and is often challenging to interpret for high-dimensional or noisy systems.

The $L_p$ norms ($L_1, L_2, L_{inf}$) provide a purely quantitative metric of similarity or deviation between functions, signals, or vector fields. They measure the magnitude of difference in an integral sense and are straightforward to compute and interpret. Unlike the ECP or Conley Index, they are entirely insensitive to topological or qualitative changes, reflecting only geometric deviations in amplitude or energy. Furthermore, \(L_p\) norms are sensitive to noise and scaling, particularly for low values of \(p\).

In summary, the Conley Index provides the highest degree of topological rigor but at substantial computational cost; the ECP offers an interpretable and noise-tolerant topological descriptor suitable for comparative analysis; and the \(L_p\) norms serve as efficient geometric measures for quantitative comparison. The combined use of ECP and \(L_p\) norms can bridge the gap between topological insight and numerical similarity, providing a balanced framework for analyzing complex or noisy dynamical systems.

Let us compare the above results with those obtained using the DoD and BEPE methods. 
In the first method, each vector field is projected onto the unit circle by normalising the displacement vectors $f(x)$. The resulting angles $\{\theta_i\}_{i=1}^N$ represent the directional distribution of the flow. To obtain a smooth representation of this distribution, a circular kernel density estimate is constructed using the von~Mises kernel.

Rather than estimating the concentration parameter $\kappa$ directly from the data, a fixed value is used for all parameter examples. This choice is motivated by empirical observations. While data-driven estimates of $\kappa$, for example those based on the mean resultant length or cross-validation, are statistically well-founded, they often yield values close to zero. In such cases, the resulting densities become nearly uniform, effectively suppressing meaningful directional structure. In contrast, moderate fixed values of $\kappa$ produce more informative and visually interpretable density profiles.

The angular resolution of the density is controlled by the number of sampling points used to discretise the interval $[-\pi,\pi)$. In our experiments we considered several resolutions (180, 360, 540 and 720 sampling points). In practice, this parameter was found to have negligible influence on the clustering structure. Increasing the number of bins refines the discretizations of the density function but does not alter its underlying shape. As a result, pairwise distances between densities change only marginally, leaving the hierarchical clustering invariant.

For each parameter configuration, the density is evaluated on a common angular grid, resulting in a discrete representation of $\rho$. The densities are then normalised and treated as feature vectors. Pairwise Euclidean distances between these vectors are computed and used as input for hierarchical clustering with Ward's linkage. 

The resulting function $\rho(\Theta)$ represents the estimated probability density of vector-field directions on the unit circle and constitutes the descriptor used in the subsequent clustering analysis.

A key observation is that the number of angular bins affects only the numerical approximation of the density, whereas the concentration parameter $\kappa$ fundamentally changes its structure. This distinction explains the empirical behaviour observed in our experiments: the dendrogram topology remains stable under changes in the number of bins, while it can vary significantly with $\kappa$, particularly for systems lacking rotational symmetry. For highly symmetric systems, such as system with Hopf Bifurcation, the method becomes largely insensitive to $\kappa$, leading to stable clustering results across a wide range of parameter values.

\begin{figure}[h]
    \centering
    \subfigure[]{\includegraphics[width=0.49\linewidth]{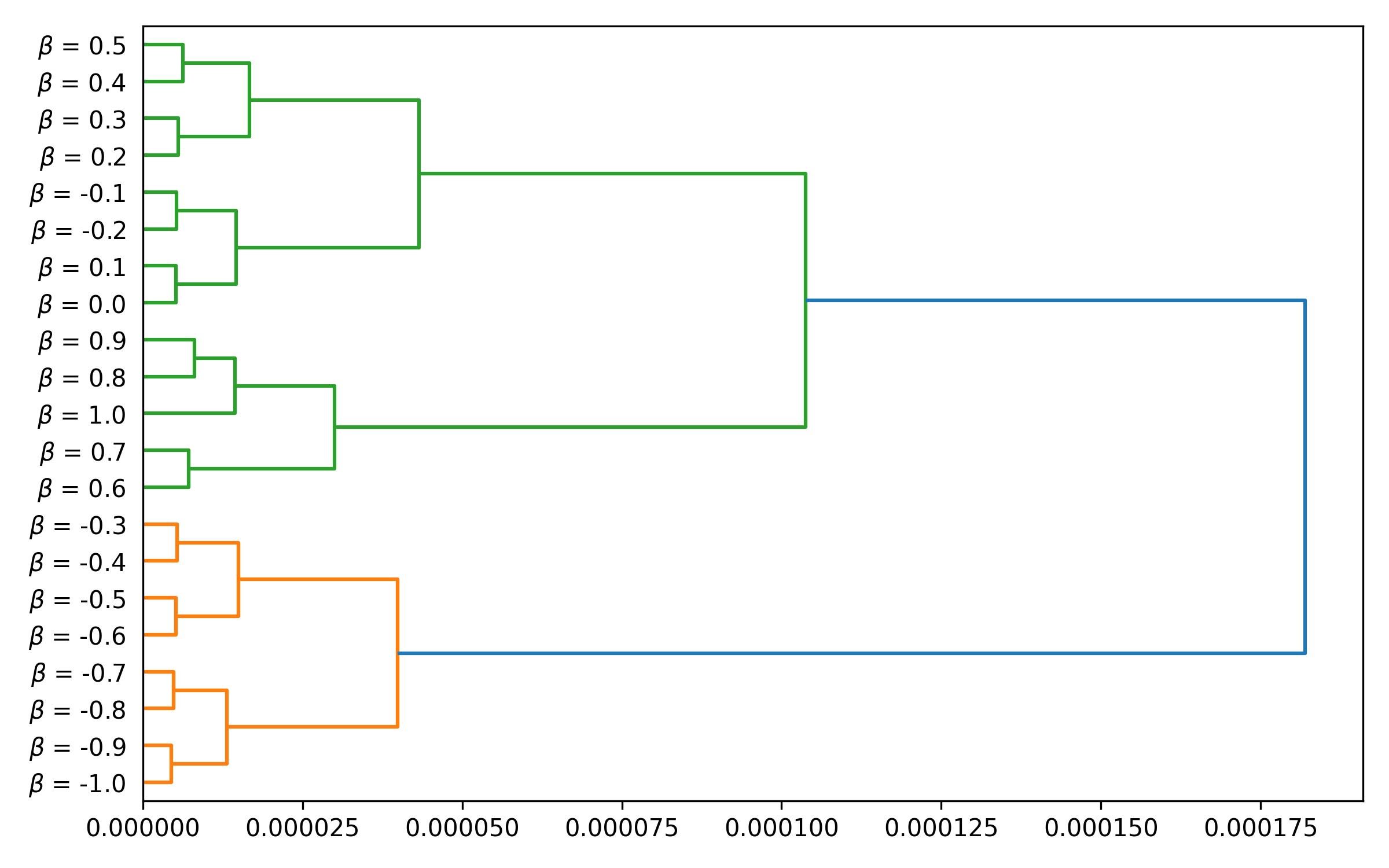}}
    \subfigure[]
    {\includegraphics[width=0.49\linewidth]{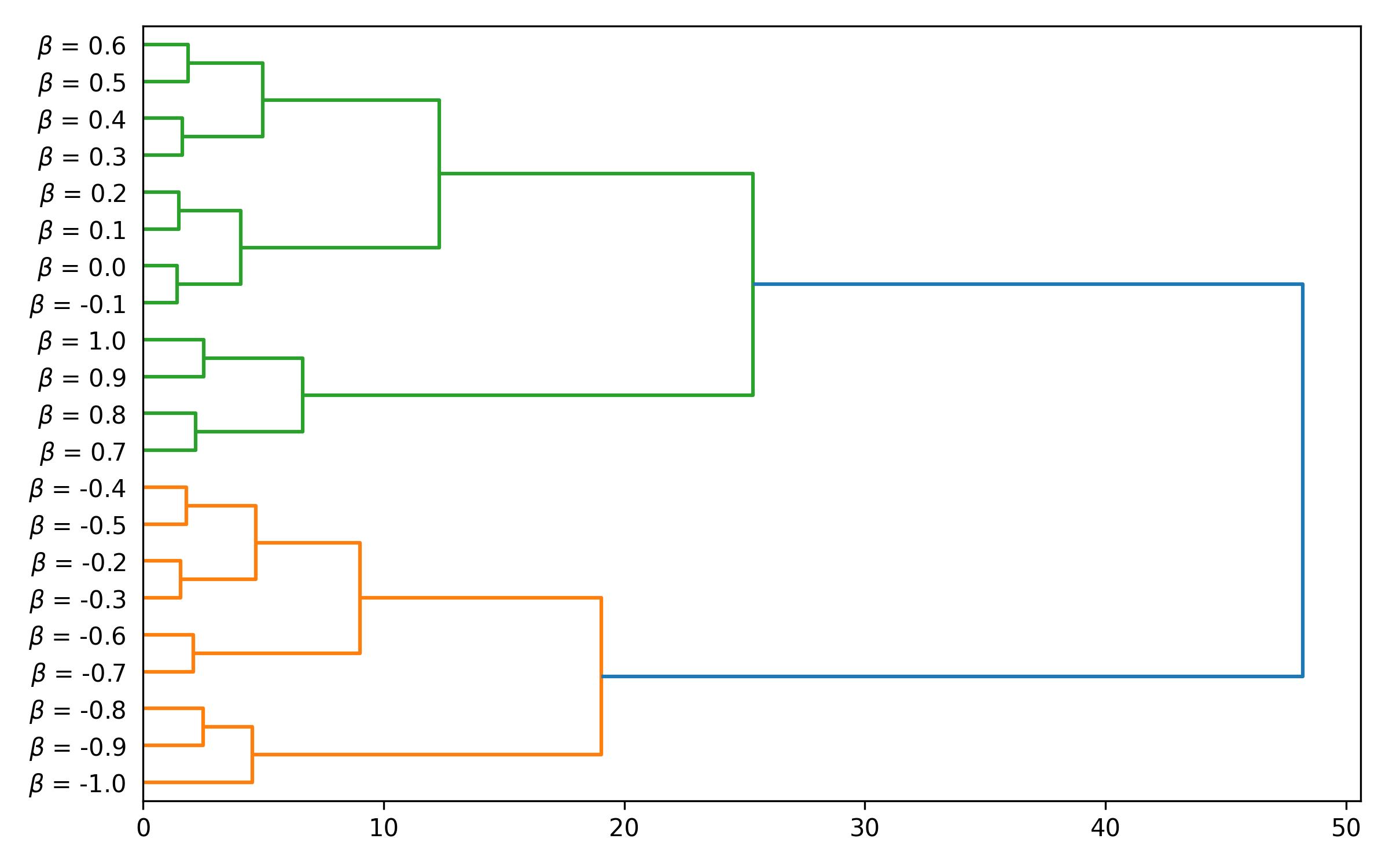}}
    \caption{Hierarchical clustering of Hopf bifurcation systems based on the proposed vector-field descriptors: (a) Density of Directions (DoD) and (b) Begin–End Point Embedding (BEPE). The vector fields were sampled on a 51x51 computational grid} 
    \label{fig:HB_Dod}
\end{figure}

In BEPE method, every grid point $(x,y)$ is mapped to a four-dimensional vector $(x, y, f_x, f_y)$ containing both the spatial location and the corresponding vector field value. This produces a point cloud in $\mathbb{R}^4$ representing the full field in a way suitable for transport-based comparisons.

Once the point clouds are constructed, we compute pairwise dissimilarities between all parameter instances using the EMD. Each cloud is treated as a uniform empirical distribution, and the ground cost is defined via Euclidean distances in $\mathbb{R}^4$. Solving the optimal transport problem for every pair yields a symmetric distance matrix that quantifies how much geometric deformation is required to transform one vector field into another. EMD distances are presented on a dendrogram Fig. ~\ref{fig:HB_Dod}(b).

The EMD-based dendrograms obtained with the BEPE method reveal a clear and physically meaningful organization of the parameter space. In particular, the hierarchical clustering distinctly separates vector fields corresponding to parameter regimes before and after the bifurcation. The method therefore captures qualitative structural transitions in the vector field in a robust and interpretable manner.

Despite these strong results, the BEPE approach has several inherent limitations. The computation of the EMD between point clouds is computationally demanding, with complexity increasing rapidly with the number of grid points, which restricts the applicability of the method to moderate spatial resolutions unless approximate optimal transport solvers are employed. Furthermore, the embedding of vector fields into $\mathbb{R}^4$ couples spatial coordinates and vector values through a fixed Euclidean ground cost, implicitly assuming that spatial displacements and differences in vector magnitude or direction are commensurable. While this assumption proved effective in the present setting, it may not be universally valid across different physical systems or scaling regimes. In addition, assigning equal mass to all grid points causes regions of weak flow to contribute equally to the transport cost as dynamically dominant regions, potentially diluting the influence of physically relevant structures.

We also investigated the dependence of the proposed descriptors on the resolution of the underlying spatial discretization. In particular, computations were performed on grids of varying sizes. The resulting dendrograms and induced cluster structures were found to be consistent across resolutions, yielding the same qualitative classification of the considered vector fields. 
This empirical observation indicates that the proposed framework is robust with respect to the choice of discretization and suggests that the extracted topological signatures capture intrinsic features of the dynamics rather than artifacts of the numerical grid.

\subsection{FitzHugh-Nagumo Model}
We consider a simplified model of a spiking neuron based on the FitzHugh-Nagumo system, which captures the essential qualitative features of excitability and oscillatory behavior observed in biological neurons. The model exhibits both resting states and periodic oscillations, making it a useful example for studying nonlinear dynamics. The system is given by:
\begin{equation}\label{eq:fhnagumo}
\begin{aligned}
\dot{x}=x-\frac{x^3}{3}-y+R I,\\
\dot{y}=\frac{1}{\tau} (x+ a - b y),
\end{aligned}
\end{equation}
where $x$ is membrane voltage, $y$ is linear recovery variable, $I$ is external stimulus.
We study the FitzHugh-Nagumo system with fixed parameters $a=0.7$, $\tau=12.5$, $R=0.1$, and we vary the parameters $b$ and the input current~$I$. Fig.~\ref{fig:FHN} a shows the phase profile of the system for the parameters and the corresponding ECP. Below we list the parameter sets considered, together with the corresponding fixed point structure and qualitative dynamics. The classification is based on the eigenvalues of the Jacobian at each equilibrium.

\begin{figure}[h]
    \centering
    \subfigure[]{\includegraphics[width=0.43\textwidth]{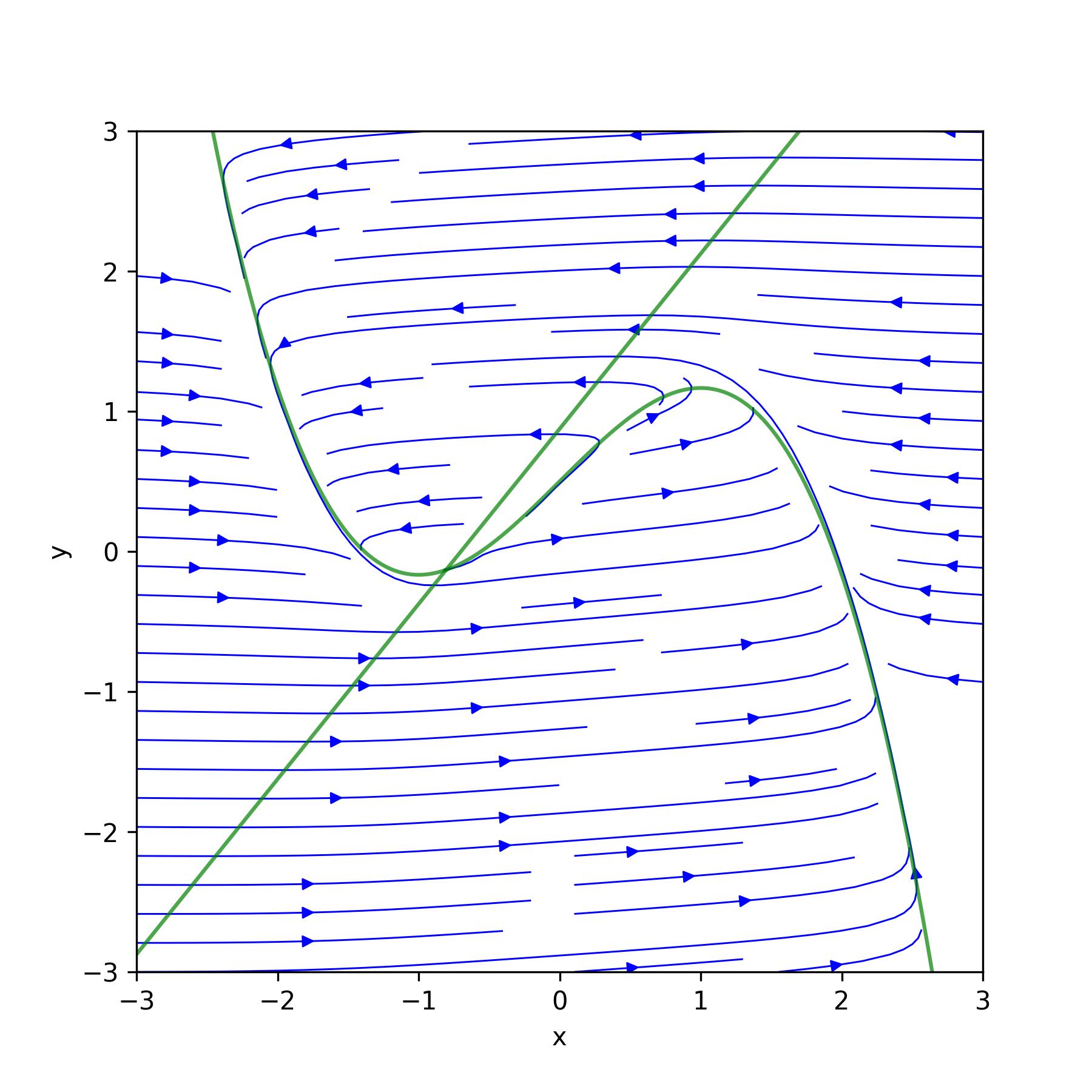}}
    \subfigure[]{\includegraphics[width=0.50\textwidth]{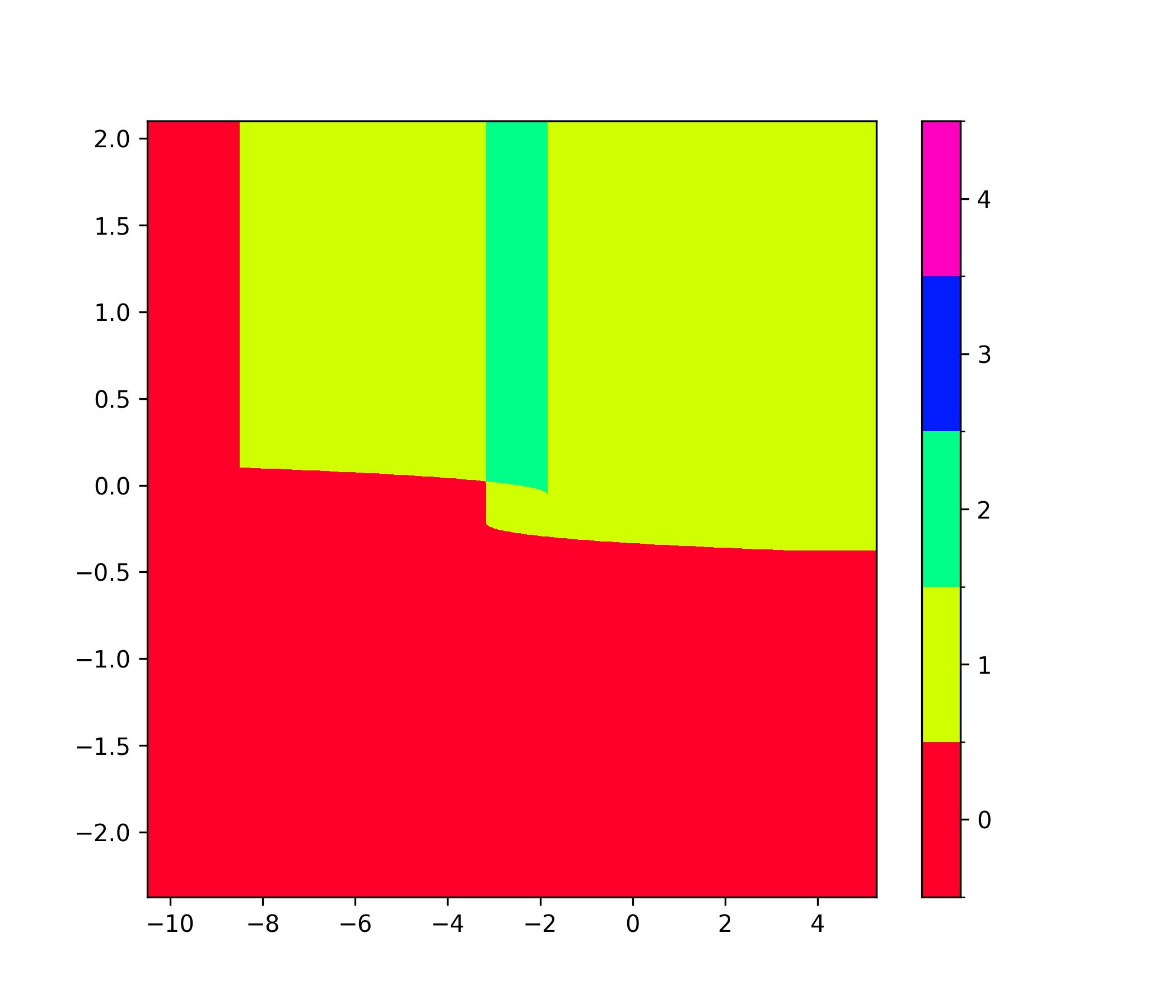}}
    \caption{(Representative phase portrait of the FitzHugh–Nagumo model for $b= 1$ and $I=5.00$: (a) vector field and trajectories, and (b) the corresponding Euler Characteristic Profile (ECP)}
    \label{fig:FHN}
\end{figure}

\begin{enumerate}

\item \textbf{$b=0.5$, $I=3$}  
    (unstable focus; stable limit cycle expected).  
    \\
    The system has a single equilibrium at $(x^*,y^*) \approx (-0.8760,\,-0.3519)$ with eigenvalues  
    $\lambda_{1,2} \approx 0.0964 \pm 0.2478\, i$.  
    Hence the equilibrium is an \emph{unstable focus}, and the dynamics typically exhibits a surrounding \emph{stable periodic orbit}.

\item \textbf{$b=0.8$, $I=-16$}  
    (single stable equilibrium on the left branch).  
    \\
    A unique equilibrium occurs at $(x^*,y^*) \approx (-1.8229,\,-1.4037)$ with eigenvalues  
    $\lambda_1 \approx -2.2871$, $\lambda_2 \approx -0.1000$.  
    Both eigenvalues are negative, so the fixed point is a \emph{stable node}.

\item \textbf{$b=0.8$, $I=0$}  
    (single stable focus; excitable regime).  
    \\
    The system has a unique equilibrium at $(x^*,y^*) \approx (-1.1994,\,-0.6243)$ with eigenvalues  
    $\lambda_{1,2} \approx -0.2513 \pm 0.2119\, i$.  
    This is a \emph{stable focus}, placing the system in a \emph{resting excitable regime}.

\item \textbf{$b=0.8$, $I=5$}  
    (unstable focus; oscillatory regime).  
    \\
    The single equilibrium $(x^*,y^*) \approx (-0.8048,\,-0.1311)$ has eigenvalues  
    $\lambda_{1,2} \approx 0.1441 \pm 0.1916\, i$.  
    It is therefore an \emph{unstable focus}, and the system typically supports a surrounding \emph{stable limit cycle} (post-Hopf oscillatory behavior).

\item \textbf{$b=0.8$, $I=10$}  
    (single unstable node; strongly driven regime).  
    \\
    A unique equilibrium occurs at $(x^*,y^*) \approx (0.4089,\,1.3861)$, with eigenvalues  
    $\lambda_1 \approx 0.7324$, $\lambda_2 \approx 0.0365$.  
    Both are positive, identifying the fixed point as an \emph{unstable node}.  
    Trajectories are repelled from the equilibrium, leading to large-amplitude nonlinear excursions.

\item \textbf{$b=1.5$, $I=4$}  
    (bistability: stable focus, saddle, stable node).  
    \\
    Three equilibria exist:
    \[
    \begin{aligned}
        (x^*_1,y^*_1) &\approx (-1.0880,\,-0.2587), && \text{stable focus},\\
        (x^*_2,y^*_2) &\approx (-0.2058,\,0.2471), && \text{saddle},\\
        (x^*_3,y^*_3) &\approx (1.3146,\,1.0073), && \text{stable node}.
    \end{aligned}
    \]
    This configuration yields \emph{bistability} with two attracting steady states separated by the saddle.

\item \textbf{$b=2$, $I=3.5$}  
    (bistability: two stable foci and one saddle).  
    \\
    Three equilibria are present:
    \[
    \begin{aligned}
        (x^*_1,y^*_1) &\approx (-1.2247,\,-0.2624), && \text{stable focus},\\
        (x^*_2,y^*_2) &\approx (0.0,\,0.35), && \text{saddle},\\
        (x^*_3,y^*_3) &\approx (1.2247,\,0.9624), && \text{stable focus}.
    \end{aligned}
    \]
    The system is therefore \emph{bistable}, with two oscillatory-type attractors.

\item \textbf{$b=2$, $I=4.2$}  
    (bistability: two stable foci and one saddle).  
    \\
    Three equilibria exist: the outer equilibria are \emph{stable foci}, and the central equilibrium is a \emph{saddle}.  
    The phase portrait again exhibits \emph{bistability} with two basins of attraction.

\item \textbf{$b=2$, $I=5.45$}  
    (bistability: stable focus, stable node, and central saddle).  
    \\
    The system admits three equilibria:
    \[
    \begin{aligned}
        (x^*_1,y^*_1) &\approx (-0.9351,\,-0.1175), && \text{stable focus},\\
        (x^*_2,y^*_2) &\approx (-0.4513,\,0.1244), && \text{saddle},\\
        (x^*_3,y^*_3) &\approx (1.3864,\,1.0432), && \text{stable node}.
    \end{aligned}
    \]
    This is again a \emph{bistable} regime, with the saddle determining the separatrix.  
    An \emph{unstable limit cycle} can also occur around the saddle, a configuration typical for FitzHugh-Nagumo dynamics.
\end{enumerate}

The results of ECP clustering depending on the $(b,I)$ parameters are shown in Fig. ~\ref{fig:FHN_dendrogram}.

\begin{figure}[h]
    \centering
    \subfigure[]{\includegraphics[width=0.46\textwidth]{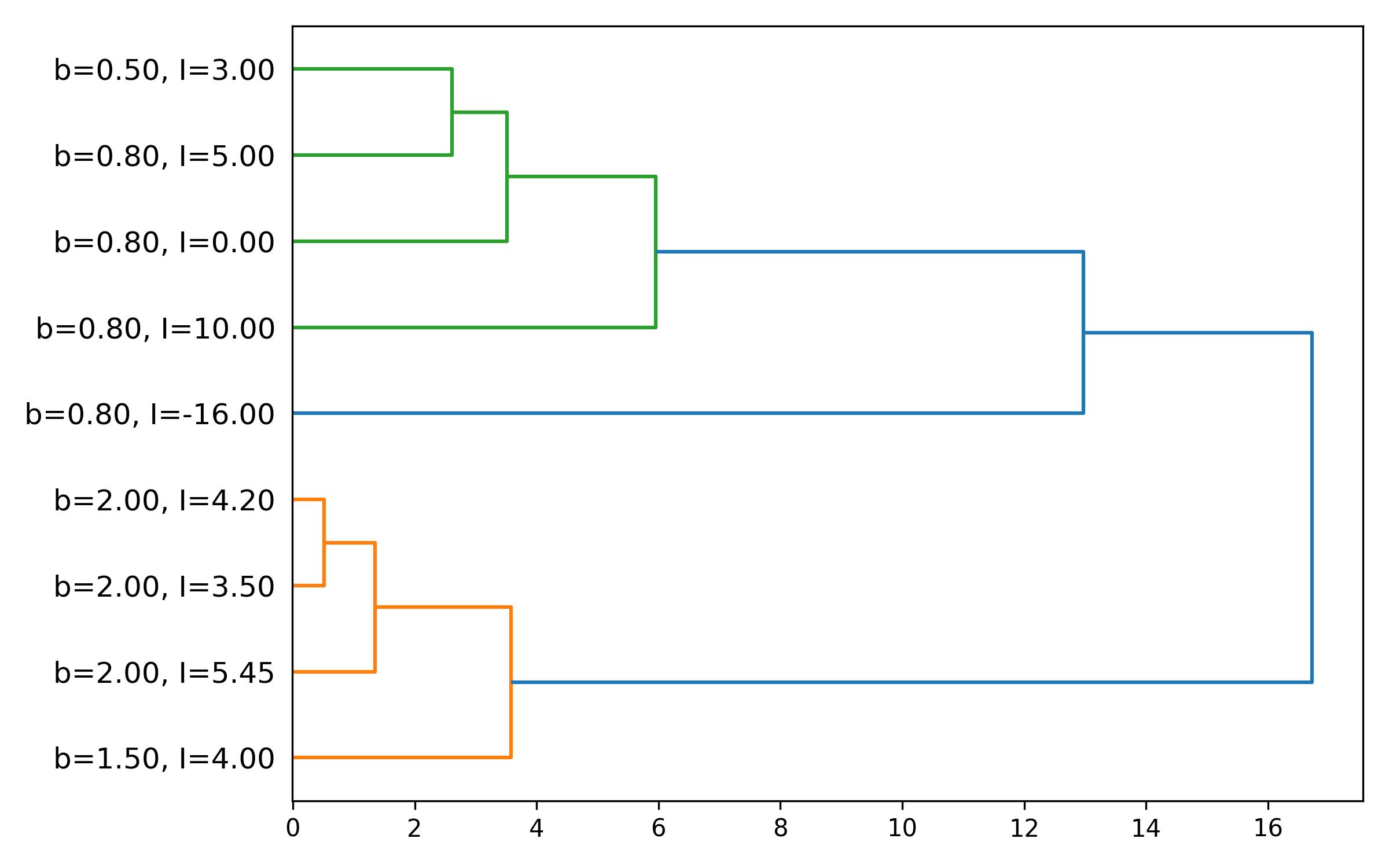}} 
    \subfigure[]{\includegraphics[width=0.46\textwidth]{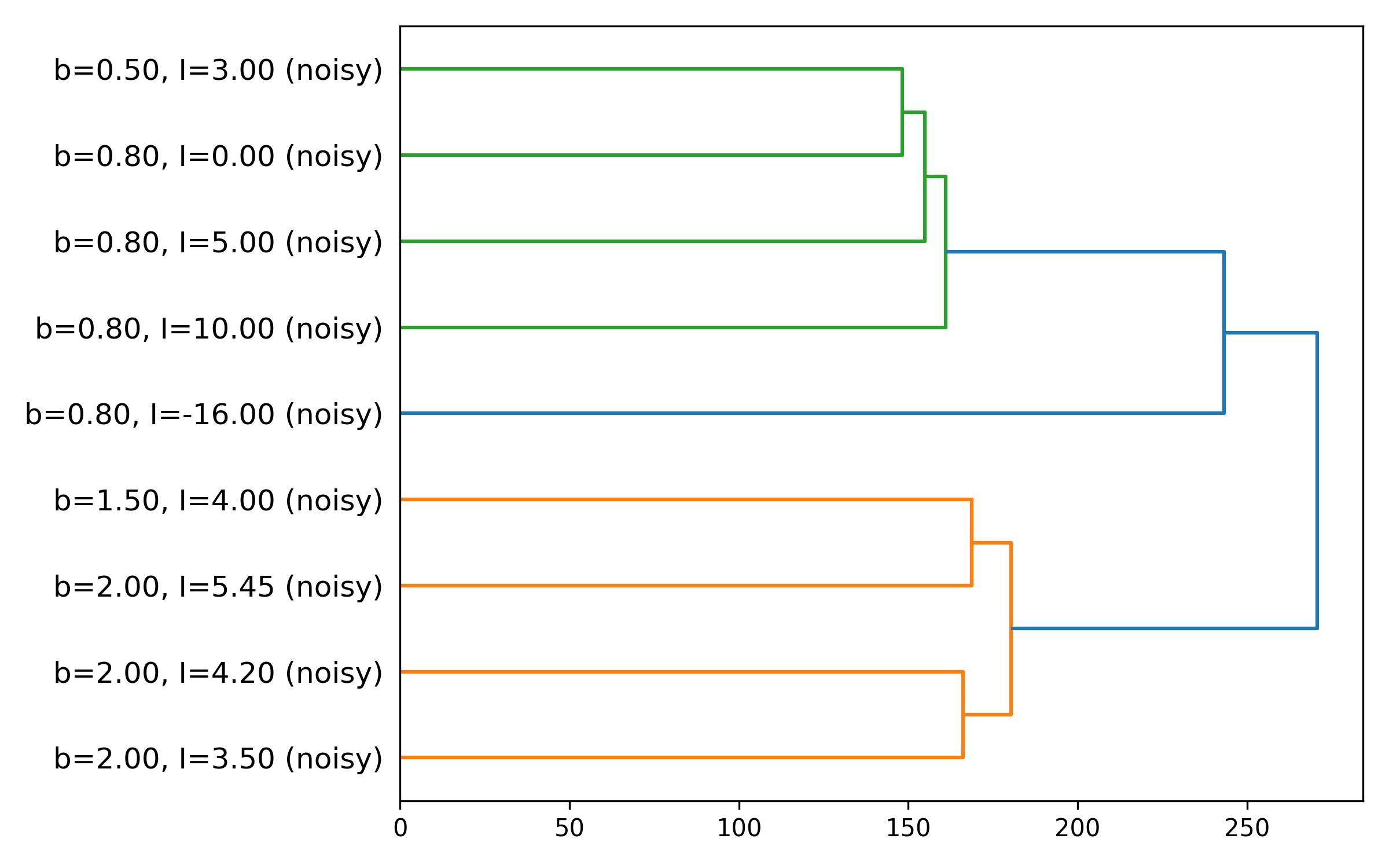}}
    \caption{Hierarchical clustering of FitzHugh–Nagumo systems using Euler Characteristic Profiles (ECPs) computed from (a) noise-free data and (b) data contaminated with additive Gaussian noise. The results demonstrate the robustness of the proposed descriptors with respect to measurement perturbations}
    \label{fig:FHN_dendrogram}
\end{figure}

The resulting dendrogram revealed a clear and interpretable organization of the parameter space. The first cluster to emerge consisted of the parameter sets $(0.50,3.00)$, $(0.80,5.00)$, $(0.80,0.00)$, and $(0.80,10.00)$. Despite differences in their exact stability properties, all four cases shared the same global dynamical structure: each system possessed a single equilibrium point that was dynamically unstable (either an unstable focus or an unstable node), implying that trajectories generically evolved toward oscillatory or runaway activity. These systems therefore correspond to monostable but oscillatory regimes.

As the clustering threshold increased, the parameter set $(0.80,-16.00)$ joined this group. This case also exhibited a single equilibrium, but in contrast to the previously mentioned systems the equilibrium was a stable node. Its later joining indicates partial similarity in equilibrium structure (monostability) but also a clear dynamic distinction due to its purely attracting behavior and the absence of oscillatory instability. The hierarchical organization therefore reflects a continuum within the monostable region of parameter space, ranging from strongly oscillatory regimes (unstable foci) to non-oscillatory.

Only at a substantially higher linkage distance did this combined monostable cluster merge with a second, well-separated cluster containing the parameter sets $(2.00,4.20)$, $(2.00,3.50)$, $(2.00,5.45)$, and $(1.50,4.00)$. These systems all exhibited three equilibrium points and therefore shared a qualitatively different phase-space topology. In each case, the system contained at least one saddle point and one or two stable attractors, leading to bistability or mixed stability configurations. The fact that these four parameter sets merged into a single cluster prior to joining the monostable group highlights that their multistable geometric organization is dynamically far more similar to each other than to any of the monostable systems.

Overall, the hierarchical clustering reflects two dominant dynamical regimes: a monostable regime characterized by either oscillatory instability or stable contraction, and a multistable regime characterized by the coexistence of multiple attractors separated by saddle points. The large linkage distance between these regimes indicates a qualitatively sharp distinction in their phase-space geometry. These results demonstrate that the $(b,I)$ parameter space naturally partitions into regions corresponding to fundamentally different computational capabilities of the model.

To evaluate the stability of the clustering with respect to measurement variability, we repeated the analysis using noise-contaminated trajectories. Noise was added directly to the simulated data in the form of additive Gaussian perturbations whose amplitude was proportional to the empirical standard deviation of the clean signal. More precisely, each data $X$ was transformed into
\[
\widetilde{X} = X + 0.05\,\sigma_X\,\eta,
\qquad \eta \sim \mathcal{N}(0,1),
\]
where $\sigma_X$ denotes the component-wise standard deviation of the noise-free data and the factor $0.05$ specifies the noise intensity. Remarkably, the resulting clustering was qualitatively identical to the one obtained from the original trajectories: although some individual parameter pairs exhibited minor shifts within branches of the dendrogram, the overall cluster structure and the separation between monostable and multistable regimes were fully preserved. This demonstrates that the proposed clustering approach is robust against moderate levels of additive Gaussian noise and reliably captures the underlying dynamical organization of the system.

\subsection{Lorenz System}
The Lorenz system is a classical model of deterministic chaos, originally introduced by Edward Lorenz in 1963 \cite{lorenz1963} to describe simplified thermal convection in the atmosphere. It is a classic example of chaotic dynamics, known for its sensitivity to initial conditions and the presence of strange attractors (see also \cite{tucker1999}).  In its vector field formulation, the system can be expressed as:

\begin{equation}
\label{eq:lorenz}
\begin{aligned}
\dot{x} &= \sigma (y - x), \\
\dot{y} &= x(\rho - z) - y, \\
\dot{z} &= xy - \beta z,
\end{aligned}
\end{equation}
where:
\begin{itemize}
  \item $x$ represents the convective motion intensity (the rate of fluid circulation),
  \item $y$ corresponds to the horizontal temperature difference,
  \item $z$ denotes the vertical temperature deviation from equilibrium,
  \item $\sigma$ is the Prandtl number, which characterizes the ratio between momentum and thermal diffusivity,
  \item $\beta$ is a geometric factor related to the aspect ratio of the convection layer,
  \item $\rho$ is the Rayleigh number, which acts as a control parameter determining the onset of convection.
\end{itemize}

For small values of $\rho$, the system exhibits a single stable equilibrium, representing a motionless fluid state. As $\rho$ increases beyond a critical threshold, a pair of symmetric nontrivial equilibria appear via a pitchfork bifurcation, corresponding to steady convection rolls. Further increase of $\rho$ leads to the onset of chaotic dynamics through a Hopf bifurcation and subsequent period-doubling bifurcations, giving rise to the well-known Lorenz attractor.

In the present study, the parameters were fixed as $\sigma = 10$ and $\beta = \tfrac{8}{3}$, while the bifurcation parameter $\rho$ was varied within the range $\rho \in [24.0, 28.0]$.  The three-dimensional computational domain was defined as $x, y \in [-20, 20]$, $ z \in [0, 40]$ and discretized into a uniform $21 \times 21 \times 21$ mesh of points $(X, Y, Z)$, over which the vector field was evaluated for each $\rho_i$. To clarify, the numerical examples are performed on a deliberately small computational grid, allowing the behavior of the proposed method to be illustrated in a transparent manner. To verify that the observed results are not an artifact of the chosen discretization, additional simulations were carried out on finer meshes. These computations confirmed that the qualitative features of the solutions remain unchanged, with the same classification consistently recovered. Consequently, the coarse-grid examples shown here should be regarded as illustrative, while the proposed method exhibits the same qualitative behavior on significantly larger computational grid.

\begin{figure}[h]
    \centering
    \includegraphics[width=0.7\linewidth]{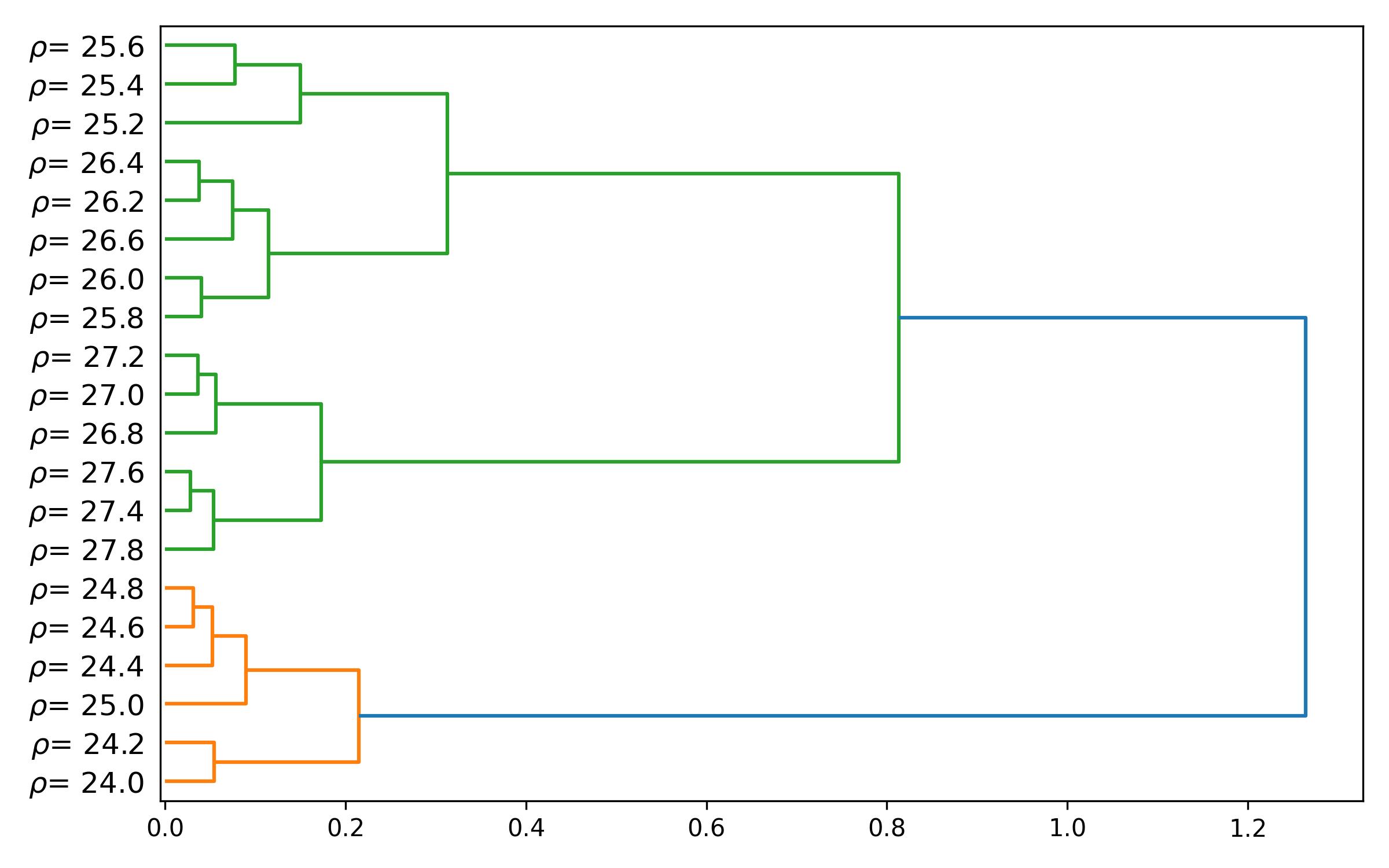}
    \caption{Hierarchical clustering of Lorenz systems obtained from Euler Characteristic Profiles (ECPs) computed for three-dimensional vector fields. The resulting clusters reflect changes in the qualitative behaviour of the system as the control parameter varies}
    \label{fig:dendrogram_lorenz}
\end{figure}

For $\rho$ in the interval $[24.0, 25.0]$, the vector field displayed a simple topological structure with two symmetric attracting regions near the nontrivial equilibria. The ECPs corresponding to these configurations were highly similar and grouped into a single cluster in the dendrogram (Fig. \ref{fig:dendrogram_lorenz}), reflecting a coherent, non-chaotic dynamical regime.

For $\rho > 25.0$, the flow topology underwent a pronounced qualitative change associated with the onset of chaotic oscillations. The ECPs exhibited increased variability and complexity, resulting in a distinct second cluster in the dendrogram. This separation indicates the emergence of the Lorenz attractor, consistent with the classical bifurcation sequence where chaos appears near $\rho \approx 24.74$.

The topological segmentation of the ECP space thus produced two distinct clusters: one representing stable fixed-point dynamics ($24.0 \leq \rho \leq 25.0$) and another corresponding to chaotic behavior ($\rho > 25.0$). This outcome demonstrates the sensitivity of the ECP to structural transitions in the flow.

\section{Numerical Examples: Discrete Dynamical System (The Hénon Map)}
\label{sec:henon}
The Hénon map is a discrete-time dynamical system that exhibits chaotic behavior introduced by Michel Hénon in 1976 \cite{henon1976} as a simplified model of the Poincaré section of the Lorenz system. It is one of the most studied examples of low-dimensional systems that exhibit chaotic behavior while remaining mathematically tractable. The map is defined by the following recurrence relations:
\begin{equation}
\begin{aligned}
x_{n+1} &= 1 - a x_n^2 + y_n, \\
y_{n+1} &= b x_n,
\end{aligned}
\end{equation}
where $a$ and $b$ are real parameters that control the non-linear dynamics of the system. For the classical choice of parameters $a = 1.4$ and $b = 0.3$, the system produces a strange attractor, now commonly referred to as the \textit{Hénon attractor}. This attractor is characterized by its fractal structure, sensitivity to initial conditions and long-term unpredictability, making it a paradigmatic example of deterministic chaos.

The role of the parameters is particularly significant in shaping the system's behavior. The parameter $a$ primarily governs the strength of the quadratic nonlinearity and thus has a direct influence on the folding and stretching mechanisms that generate chaotic dynamics. Smaller values of $a$ typically yield more regular trajectories, while larger values push the system toward strongly chaotic regimes. On the other hand, the parameter $b$ controls the contraction along the $y$-axis; it acts as a damping coefficient. For $|b| < 1$, the system dissipates volume in phase space, which is essential for the existence of an attractor.

In the present analysis, we examined the evolution of the vector displacement field
\[
F(x,y;a,b) = \left( x_{n+1} - x_n,\, y_{n+1} - y_n \right),
\]
which represents the local update from each point $(x,y)$ to its image under the Hénon map. The computations were performed for a fixed $b = 0.3$ and for parameter $a$ varying in the range
$a \in [0.10,1.60]$, with step $0.05$.
The state space was discretized on a uniform grid:
$x \in [-2, 2], \ y \in [-2, 2]$,
with $201$ equidistant samples per dimension, producing a $201\times201$ lattice for each value of $a$.

\begin{figure}[h]
    \centering
    \includegraphics[width=0.7\linewidth]{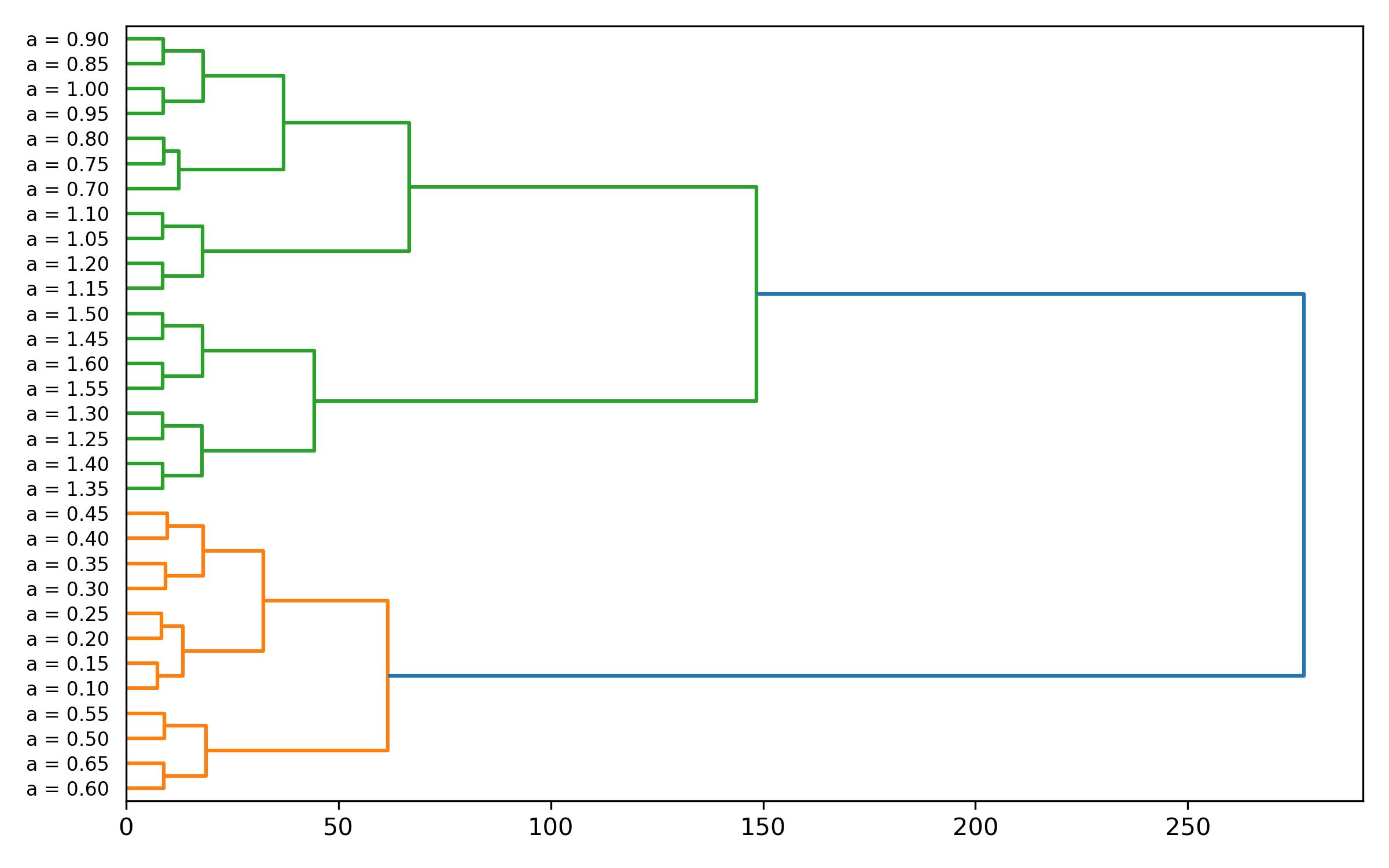}
    \caption{Hierarchical clustering of Hénon map realizations based on Euler Characteristic Profiles (ECPs) computed from the corresponding embedded point clouds. The obtained clustering captures similarities between different dynamical regimes of the discrete-time system}
    \label{fig:dendrogram_henon}
\end{figure}

For $b=0.3$, the Hénon map follows the classical period-doubling route to chaos. The attracting fixed point loses stability at $a=0.3675$, after which a cascade of period-doubling bifurcations generates stable periodic orbits of periods $2$, $4$, $8$, $\ldots$, accumulating at $a_\infty \approx 1.05805$. Beyond this accumulation point, the dynamics are predominantly chaotic.

The ECP dendrogram reveals two dominant topological regimes (Fig.~\ref{fig:dendrogram_henon}). The first cluster, comprising parameter values $a \leq 0.65$,  corresponds to the early stages of the period-doubling scenario so the low-complexity regime of the Hénon map. This interval begins with the stable fixed-point regime observed for small values of $a$ and subsequently encompasses the classical cascade of period-doubling bifurcations, leading to progressively more complex periodic attractors. In general, this cluster corresponds to regular dynamics in which there is a low-period attractor: an attracting fixed point or an attracting periodic orbit with period 2.

The second major cluster, beginning at $a \geq 0.70$, contains parameter values approaching and exceeding the accumulation point of the period-doubling cascade. It is divided into two principal subclusters. The first, covering $a = 0.70$--$1.20$, indicating continued geometric evolution of the attractor while preserving its overall topological character. The second subcluster, spanning $a = 1.25$--$1.60$, corresponds to the fully developed chaotic regime, characterized by increasingly folded and stretched invariant sets. Overall, the obtained hierarchy is consistent with the well-established bifurcation scenario of the Hénon map, in which the loss of stability of the fixed point is followed by a cascade of period-doubling bifurcations leading to chaos, whereas subsequent parameter variation primarily modifies the geometry of the chaotic attractor rather than introducing a single additional qualitative transition.

The clear separation observed in the dendrogram indicates that the ECP effectively captures major changes in the global topology of the displacement fields associated with the Hénon map. Rather than identifying individual bifurcations, the obtained clustering reflects the transition from relatively simple dynamical regimes to geometrically more complex dynamics, including the onset and subsequent development of chaotic attractors. These results demonstrate that the ECP provides a reliable descriptor of structural reorganization in discrete nonlinear systems and is capable of distinguishing parameter regions characterized by qualitatively different phase-space geometries.

\section{Numerical Examples: High-dimensional Turbulent Dynamical Systems}
\label{sec:TDS}
In this study, we analyze a collection of high-dimensional turbulent dynamical systems obtained from the Johns Hopkins Turbulence Databases (JHTDB) \cite{jhtdb,jhtdb2}, which provide publicly accessible, well-documented numerical simulations of canonical flows in fluid dynamics and magnetohydrodynamics. The data were accessed through the JHTDB web services\footnote{\url{https://turbulence.idies.jhu.edu/staging/home}} and constitute benchmark examples of complex, nonlinear, spatio-temporal dynamical systems governed by partial differential equations.

The following four data sets were investigated:
\begin{itemize}
    \item \textbf{Forced Isotropic Turbulence} \footnote{\url{https://doi.org/10.7281/T1KK98XB}} (\texttt{isotropic1024coarse}) — an extended data set representing statistically stationary, homogeneous and isotropic turbulence driven by external forcing \cite{isotropic}.
    \item \textbf{Forced Magnetohydrodynamic (MHD) Turbulence} \footnote{\url{https://doi.org/10.7281/T1930RBS}} (\texttt{mhd1024}) — a three-dimensional turbulent flow coupled to a magnetic field, governed by the incompressible MHD equations \cite{mhd}.
    \item \textbf{Homogeneous Buoyancy-Driven Turbulence} \footnote{\url{https://doi.org/10.7281/T1VX0DPC}} (\texttt{mixing}) — a data set describing turbulent mixing induced by buoyancy effects in a homogeneous configuration \cite{hbdt}.
    \item \textbf{Turbulent Channel Flow} \footnote{\url{https://doi.org/10.7281/T10K26QW}}(\texttt{channel}) — a wall-bounded shear flow representing fully developed turbulence in a planar channel geometry \cite{channel}.
\end{itemize}
For each system, the state of the flow was sampled at 16x16x16 grid at multiple temporal instances $t \in \{0.0,\,0.25,\,0.5,\,0.75,\,1.0\}$, ensuring that the analysis captures temporal variability while remaining within a statistically stationary regime for the forced configurations. At each time instant, the corresponding velocity fields define a point in a high-dimensional state space, yielding a collection of realizations for each underlying dynamical system.

\begin{figure}[h]
    \centering
    \includegraphics[width=0.7\linewidth]{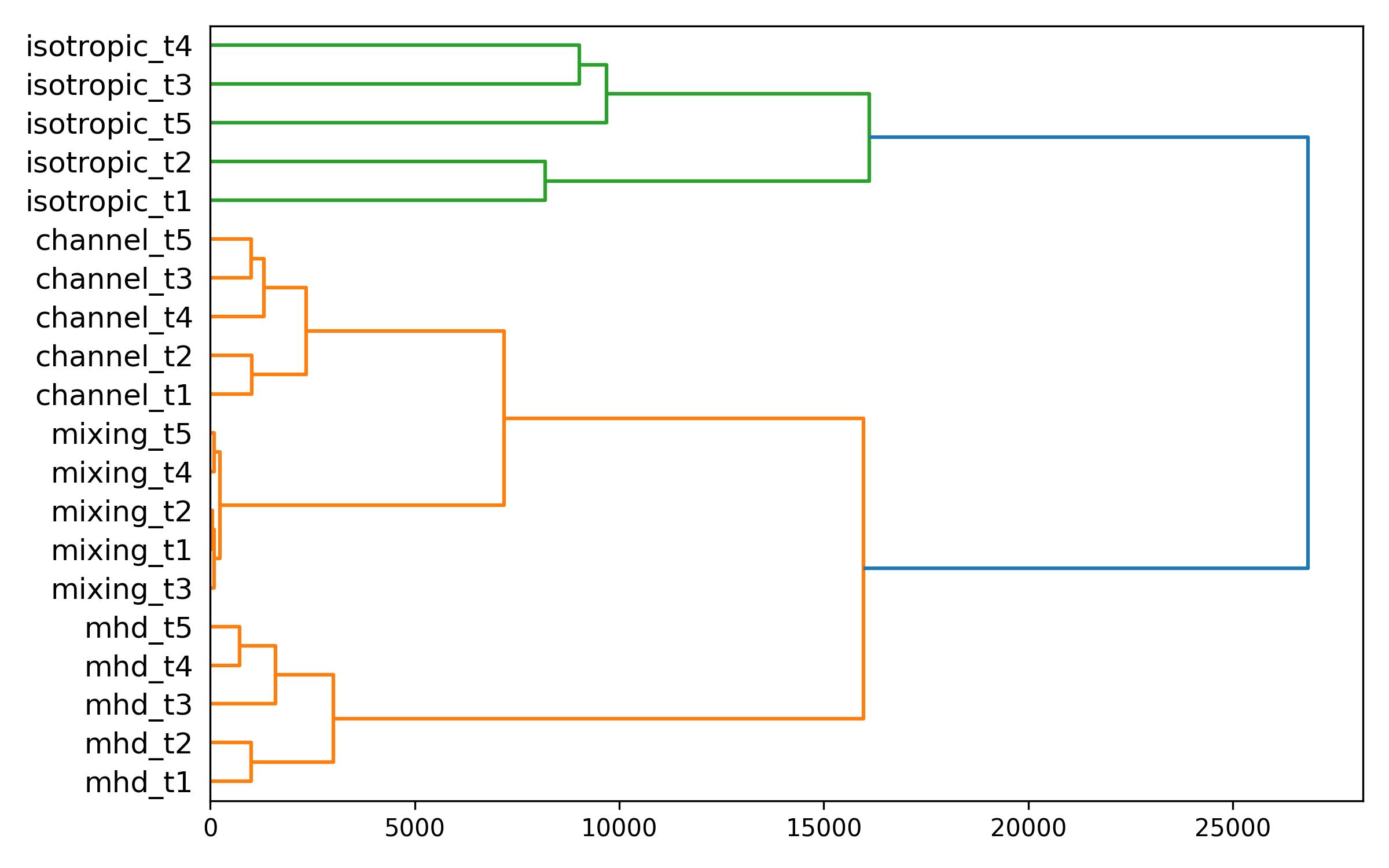}
    \caption{Hierarchical clustering of velocity fields extracted from the Johns Hopkins Turbulence Databases (JHTDB) using the proposed topological descriptors. The dendrogram illustrates similarities and differences between representative turbulent flow configurations, including isotropic turbulence, magnetohydrodynamic turbulence, buoyancy-driven turbulence, and turbulent channel flow}
    \label{fig:dendrogram_realdata}
\end{figure}
The extracted data were subsequently analyzed using the ECP method. The resulting dendrogram, shown in Fig. ~\ref{fig:dendrogram_realdata}, demonstrates a clear and unambiguous hierarchical clustering structure. Each cluster corresponds exactly to one of the four investigated dynamical systems, independent of the sampling time. This result indicates that the ECP captures intrinsic dynamical features that distinguish different classes of turbulent flows, while remaining robust with respect to temporal evolution within each system.

\section{Discussion and Conclusions}
In this work, we introduced a new class of topological descriptors for dynamical systems generated by vector fields, based on ECCs and ECPs. 
By shifting the focus from individual trajectories to global, filtration-based summaries of the geometry of vector fields, the proposed framework provides a qualitative yet computationally tractable perspective on dynamical behavior. 
The resulting descriptors can be constructed both in continuous settings and from finite samples of the vector field or trajectories, making them applicable in both analytical and data-driven contexts.

Through numerical experiments on linear and nonlinear systems, including systems exhibiting bifurcations and chaotic dynamics, we demonstrated that Euler characteristic–based descriptors capture essential qualitative features of the underlying dynamics. In particular, the proposed profiles respond sensitively to changes in dynamical regimes, reflect the presence of invariant structures, and produce signatures that vary smoothly with respect to system parameter. Comparisons with classical tools such as the Conley index and $L_p$-type metrics indicate that ECCs and ECPs offer complementary information while avoiding some of the computational and practical limitations of trajectory-based or index-based methods.

From a methodological perspective, a key advantage of ECP lies in their computational efficiency and scalability. Unlike persistent homology, which often becomes prohibitively expensive for high-dimensional or densely sampled systems, ECCs and ECPs can be computed efficiently on fine discretizations of phase space and naturally extend to multiparameter filtrations. This makes them particularly attractive for applications involving large datasets, parameter continuation, or repeated evaluations across families of vector fields.

In addition to Euler characteristic–based descriptors, we have also investigated alternative constructions, namely the Density of Directions (DoD) and the Beginning–End Point Embedding (BEPE). These approaches provide complementary perspectives on the geometry of vector fields by capturing directional distributions and trajectory-based embeddings, respectively. While DoD emphasizes the global organization of vector directions and BEPE encodes coarse dynamical transport between regions of phase space, both methods share with ECP the advantage of being computable directly from sampled data.

At the same time, the proposed approach has limitations that warrant further investigation. Euler characteristic–based descriptors provide coarse topological summaries and do not encode homological dimension or geometric scale in the same way as persistent homology-based methods. As a consequence, different dynamical systems may occasionally produce similar profiles, especially when their qualitative structures are topologically but not geometrically distinct. Moreover, the construction of filtrations derived from the vector field plays a crucial role in the expressiveness of the resulting descriptors and currently relies on problem-specific insight.

Several directions for future work naturally follow from these observations. Most importantly, a deeper theoretical understanding of what dynamical information is encoded by the proposed characteristics is still needed. In particular, their relationship to standard concepts from dynamical systems theory, such as invariant sets, structural stability, bifurcations, Lyapunov exponents, and dynamical equivalence, should be investigated systematically. From a theoretical standpoint, further analysis of the stability properties of ECP under perturbations of the vector field and discretization schemes would strengthen the mathematical foundations of the framework. On the computational side, adaptive and data-driven strategies for constructing informative filtrations could enhance the sensitivity of the descriptors while preserving their efficiency. Finally, extending the present approach to time-dependent vector fields or real data represents a promising avenue for broadening its applicability.

Overall, the results presented here suggest that ECCs and ECPs provide a practical and theoretically grounded addition to the toolbox of qualitative dynamical systems analysis. The proposed framework offers a scalable and robust approach to the comparison and characterization of vector fields, with potential applications ranging from low-dimensional models to complex, data-driven systems.

\vskip 0.2cm
\noindent \textbf{Data Availability.}  
The source codes used for the computation of selected examples are available in the GitHub repository \cite{github}. 
The individual data used in some of the studies are publicly available through the Johns Hopkins Turbulence Database (JHTDB) online services.

\vskip 0.2cm
\noindent \textbf{Competing Interests.}  
The authors have no relevant financial or non-financial interests to disclose.

\vskip 0.2cm
\noindent \textbf{Acknowledgements.}  
Marta Marszewska and Justyna Signerska were supported by the Dioscuri program initiated by the Max Planck Society, jointly managed with the NCN (National Science Centre, Poland) and mutually funded by the Polish Ministry of Science and Higher Education and the German Federal Ministry of Research, Technology and Space. 
Justyna Signerska was also supported by NCN grant no.~2019/35/D/ST1/02253. 
Pawe{\l} D{\l}otko was supported by the project "Center for Trustworthy Artificial Intelligence for Life Sciences" implemented under the International Research Agendas programme of the Foundation for Polish Science co-financed by the European Union under the European Funds for Smart Economy 2021-2027 (FENG). 
Computations were carried out using the computers of Centre of Informatics Tricity Academic Supercomputer \& Network (Grant ID: PT01147).

\appendix
\section*{Appendix A: Optimized Algorithms for ECP Computations and Visualization}
This appendix presents implementation details and algorithmic optimizations used for efficient computation and visualization of Euler Characteristic Profiles (ECPs). In practical applications involving large datasets or high-resolution vector fields, computational efficiency becomes a critical factor.

In order to understand how contributions propagate through the ECP scheme, it is useful to analyze the structure of the $d$-dimensional Boolean lattice.  Vertices are indexed by binary strings $x \in \{0,1\}^d$, where the number of ones represents the number of incident edges.  Contributions introduced at a given vertex propagate downwards to all of its coordinate-wise subsets.  This naturally leads to over-counting, which must be corrected by subtracting the appropriate multiplicities.

\paragraph{The 3D case.}
The vertices of the $3$-cube may be arranged according to their
Hamming weight:
\begin{verbatim}
111
3 edges here
110  101  011
2 edges here
100  010  001
1 edge here
000
\end{verbatim}
If contributions are introduced at vertices \texttt{110}, \texttt{101} and \texttt{011}, then each of the weight-1 vertices (\texttt{100}, \texttt{010}, \texttt{001}) receives two inherited contributions. Therefore, one must subtract the excess once so that their net contribution is counted only once. Next, these three weight-1 vertices each introduce a contribution at \texttt{000}, so the bottom vertex is counted three times and must therefore be corrected by subtracting the inherited contribution twice.

\paragraph{The 4D case.}
A similar pattern becomes more visible in four dimensions:
\begin{verbatim}
1111
4 edges here
1110  1101  1011  0111
3 edges here
1100 1010 1001 0110 0101 0011
2 edges here
1000 0100 0010 0001
1 edge here
0000
\end{verbatim}
Introducing contributions at the weight-3 vertices (\texttt{1110}, \texttt{1101}, \texttt{1011}, \texttt{0111}) causes each weight-2 vertex to inherit three contributions, although only one should remain. Thus, two copies must be subtracted. Each weight-1 vertex then receives three inherited contributions from the weight-2 level, so again two contributions must be removed. Finally, the vertex \texttt{0000} inherits four contributions (one from each weight-1 vertex) and therefore three of them must be subtracted.

\paragraph{Emerging pattern and the 5D case.}
To observe the systematic structure of these multiplicities one must look at the $5$-dimensional cube as well. In general, a vertex of Hamming weight $s$ inherits contributions from all vertices of weight $t > s$. The number of such supersets is exactly
\[
\binom{d-s}{t-s},
\]
which explains why, for example, weight-1 vertices in $4$D inherit $\binom{3}{2}=3$ contributions from weight-2 vertices and why the bottom vertex inherits $\binom{4}{1}=4$ contributions from weight-1.

\paragraph{General combinatorial rule.}
Let $A(x)$ denote the accumulated contribution at vertex $x$ and
let $C(x)$ denote the original contribution introduced at $x$.
Contribution propagation satisfies
\begin{equation}
A(x)=\sum_{y \succeq x} C(y),
\end{equation}
where $y \succeq x$ means that $y$ contains all $1$-positions of $x$. Recovering the original (non-duplicated) contributions requires inversion on the Boolean lattice:
\begin{equation}
\label{eq:correct}
C(x)=\sum_{y\succeq x} (-1)^{|y|-|x|}\, A(y).
\end{equation}
This formula exactly reproduces all correction factors observed in the 3D, 4D and 5D cases above.

A practical implementation of this can be obtained via the Let's try to implement this solution in practice. Vertices are indexed by integers $m \in \{0,\ldots,2^d-1\}$, whose binary expansions correspond to the coordinates of the hypercube.

\begin{verbatim}
Inputs:
  d - dimension (positive integer)
  A[0..2^d-1] - array of accumulated values, indexed by integer masks

Output:
  C[0..2^d-1] - array of original contributions (after correction)

Algorithm:
  1. Copy accumulated values into working array F:
       for m = 0 to 2^d - 1:
           F[m] := A[m]

  2. For each bit i from 0 to d-1 perform elimination:
       for i = 0 to d - 1:
           for m = 0 to 2^d - 1:
               if ((m >> i) & 1) == 0 then
                   F[m] := F[m] - F[m | (1 << i)]
               end if
           end for
       end for

  3. Now F contains the original contributions:
       for m = 0 to 2^d - 1:
           C[m] := F[m]
       end for
\end{verbatim}

This procedure eliminates inherited contributions bit by bit, beginning with the highest Hamming-weight vertices and terminating with the true local contributions $C$.  The algorithm is exact and dimension-agnostic, and therefore suitable for ECP analysis and visualization across arbitrary dimensions.

\section*{Appendix B: Efficient Computation of the ECP Representation on a Regular Grid}

In this section we provide a linear algorithm for ECP computation on a given grid of points. We begin by introducing the representation of the ECP. As in the following section, we express it as:
\[
ECP = \{ ((x_1^1, \ldots, x_n^1), val^1), ((x_1^2, \ldots, x_n^2), val^2), \ldots, ((x_1^l, \ldots, x_n^l), val^l) \},
\]

where each $((x_1^k, \ldots, x_n^k), val^k)$ represents a sample point in the $n$-dimensional space together with its associated value.

Next, we define a regular grid of points:
\[
G = (x_1^{\min} + i_1\, \Delta x_1,\; x_2^{\min} + i_2\, \Delta x_2,\; \ldots,\; x_n^{\min} + i_n\, \Delta x_n),
\]

for indices $i_1, \ldots, i_n$ varying from $1$ to $max_i$. Our goal is to compute the ECP value at every grid point.

A naive approach would, for each grid point, determine which of the ECP sample points
\[
\{(x_1^1, \ldots, x_n^1), (x_1^2, \ldots, x_n^2), \ldots, (x_1^l, \ldots, x_n^l)\}
\]
are dominated by $(x_1^{\min} + i_1\, \Delta x_1,\, x_2^{\min} + i_2\, \Delta x_2,\, \ldots,\, x_n^{\min} + i_n\, \Delta x_n)$ and then sum their corresponding contributions. However, this naive computation has a complexity of
\[
O(max_1 \cdot max_2 \cdot \ldots \cdot max_n \cdot l),
\]
where $l$ is the number of points in the ECP representation. Fortunately, a more efficient procedure exists, achieving complexity proportional to $O(max_1 \cdot max_2 \cdot \ldots \cdot max_n)$. The algorithm proceeds as follows:

\begin{itemize}
    \item \textbf{Initialization.}  
    For each ECP sample point $(x_1^k, \ldots, x_n^k)$, we identify the nearest grid point.  
    Each grid node stores a pair, initially set to $(0, 0)$.  
    The first component represents the final ECP value (to be computed at the end of the algorithm), while the second stores the cumulative contribution from the nearest ECP points.  
    The ECP points are thus rounded to the grid and their corresponding $val^k$ values are assigned to the second component of the appropriate grid pairs.  
    This step has a computational complexity of $O(l)$, proportional to the number of ECP points. 
    
    \item \textbf{Accumulation over the grid.}  
    We traverse the grid in lexicographic order (or equivalently, using an $n$-dimensional binary counter).  
    If the second component of a grid point is nonzero, we update the first component accordingly.  
    Subsequently, we propagate this modification to the succeeding grid points in the traversal order.  
    This step has complexity proportional to the grid size multiplied by the grid dimensionality (which is typically treated as constant).
\end{itemize}

This approach significantly reduces computational cost by exploiting structured traversal and local accumulation, yielding a total complexity linear in the number of grid points.
\end{document}